\documentclass[11pt]{amsart}

\usepackage{changes}
\usepackage{booktabs}
\usepackage{multirow}

\usepackage[normalem]{ulem}
\usepackage{kbordermatrix}
\usepackage{graphicx}
\usepackage{wrapfig}
\usepackage{amssymb}
\usepackage{algorithm}
\usepackage{mathtools}
\usepackage[shortlabels]{enumitem}
\usepackage{times}
\usepackage[top=20mm, bottom=20mm, left=20mm, right=20mm]{geometry}
\usepackage{color}
\usepackage{bm}
\usepackage{algorithmic}
\usepackage{url}

\newcommand{\mtx}[1]{\bm{\mathsf{#1}}}

\newcommand{\defeq}{\vcentcolon=}
\newcommand{\si}{(i)} 
\newcommand{\simo}{(i-1)} 

\newtheorem{thm}{Theorem}

\newtheorem{corollary}[thm]{Corollary}

\theoremstyle{definition}
\newtheorem{remark}{Remark}

\usepackage{epstopdf}
\usepackage{blkarray}
\usepackage{amsaddr}
\usepackage{subcaption}

\graphicspath{{figs/}}

\begin{document}

\begin{center}
\textbf{\large{Efficient algorithms for computing rank-revealing factorizations on a GPU}}

\vspace{5mm}


\textit{
Nathan~Heavner\footnote{nathan.d.heavner@gmail.com, Department of Applied Mathematics, University of Colorado at Boulder, USA},
Chao~Chen\footnote{chenchao.nk@gmail.com, Oden Institute, University of Texas at Austin, USA},
Abinand~Gopal\footnote{abinand.gopal@yale.edu, Department of Mathematics, Yale University, USA},
Per-Gunnar~Martinsson\footnote{pgm@oden.utexas.edu, Oden Institute \& Department of Mathematics, University of Texas at Austin, USA}
}

\vspace{5mm}

\begin{minipage}{140mm}
\textbf{Abstract:} Standard rank-revealing factorizations such as the singular value decomposition and column pivoted QR factorization are challenging to implement efficiently on a GPU. A major difficulty in this regard is the inability of standard algorithms to cast most operations in terms of the Level-3 BLAS. This paper presents two alternative algorithms for computing a rank-revealing factorization of the form $\mtx{A} = \mtx{U} \mtx{T} \mtx{V}^*$, where $\mtx{U}$ and $\mtx{V}$ are orthogonal and $\mtx{T}$ is trapezoidal (or triangular if $\mtx{A}$ is square). Both algorithms use randomized projection techniques to cast most of the flops in terms of matrix-matrix multiplication, which is exceptionally efficient on the GPU. Numerical experiments illustrate that these algorithms achieve significant acceleration over finely tuned GPU implementations of the SVD while providing low rank approximation errors close to that of the SVD.

\end{minipage}

\end{center}


\section{Introduction}

\subsection{Rank-revealing factorizations} \label{ss:rrf}

Given an $m \times n$ matrix $\mtx{A}$ with $m \ge n$, it is often desirable to compute a factorization of $\mtx{A}$ that uncovers some of its fundamental properties. {One such factorization, the \textit{rank-revealing UTV factorization}, is characterized as follows.} We say that a matrix factorization
\begin{equation*}
\begin{array}{ccccc}
\mtx{A} & = & \mtx{U} & \mtx{T} & \mtx{V}^*, \\
m \times n & & m \times m & m \times n & n \times n
\end{array}
\end{equation*}
is \textit{rank-revealing} if $\mtx{U}$ and $\mtx{V}$ are orthogonal matrices, $\mtx{T}$ is an upper trapezoidal matrix\footnote{{A  matrix $\mtx{A} \in \mathbb{R}^{m \times n}$ is trapezoidal if $\mtx{A}(i,j) = 0$ for all $i>j$ (upper trapezoidal) or all $i<j$ (lower trapezoidal).}}, and
for {all} $k$ such that $1 \leq k < n$,
it is the case that
%
%
\begin{equation} \label{e:rr}
e_k = \| \mtx{A} - \mtx{U}(:,1:k) \mtx{T}(1:k,:) \mtx{V}^* \| \approx \inf \{\|\mtx{A} - \mtx{B} \| : \mtx{B} \text{ has rank } k\},
\end{equation}
where $\mtx{B}$ is an arbitrary matrix of the same size as $\mtx{A}$, the norm is the spectral norm, $\mtx{U}(:,1:k)$ and $\mtx{T}(1:k,:)$ denote the first $k$ columns and the first $k$ rows of corresponding matrices, respectively (Matlab notation).
This informal definition is a slight generalization of the usual definitions of rank-revealing decompositions that appear in the literature, minor variations of which appear in, \textit{e.g.}~\cite{chandrasekaran1994rank,stewart1992updating,chan1987rank,chan1992some}. Rank-revealing factorizations are useful in solving problems such as least squares approximation \cite{chan1992some,chan1990computing,golub1965numerical,lawson1995solving,bjorck1996numerical,golub1980analysis}, rank estimation \cite{chan1987rank,stewart1984rank,stewart1998matrix}, subspace tracking \cite{bischof1992updating,stewart1992updating}, and low-rank approximation \cite{liberty2007randomized,2006_drineas_kannan_mahoney,kirsteins1994adaptive,clarkson2017low,bebendorf2003adaptive}, among others.
%

Perhaps the two most commonly known and used rank-revealing factorizations are the singular value decomposition (SVD) and the column pivoted QR decomposition (CPQR).\footnote{See Sections \ref{sec:ch5-svd} and \ref{sec:ch5-cpqr}, respectively, for a brief overview of these factorizations.} A singular value decomposition provides a theoretically optimal rank-revealing decomposition, in that the error $e_k$ in (\ref{e:rr}) is minimum.
The SVD has relatively high computational cost, however. The CPQR is less expensive, and also has the advantage that it builds the factorization incrementally, and can halt once a specified tolerance has been reached. This latter advantage is very valuable when working with matrices that are substantially rank-deficient. 
{The drawback of CPQR is that it is much worse than the SVD at revealing the numerical rank (see, e.g., theoretical error bounds in \cite[Section 3.2]{dong2021simpler}, and empirical results in Figures \ref{fig:error} and \ref{fig:relerror}).
For many practical applications, the error incurred is noticeably worse but usually acceptable. There exist pathological matrices for which CPQR leads to very suboptimal approximation errors~\cite{kahan1966numerical}, and specialized pivoting strategies to remedy it in some situations have been developed~\cite{chan1987rank,gu1996efficient}.}

A third choice is the rank-revealing UTV factorization (RRUTV) \cite{stewart1992updating,1994_stewart_UTV,martinsson2019randutv,1997_pchansen_low_rank_UTV,barlow2002modification}. An RRUTV can be thought of as a compromise between the SVD and CPQR that is better at revealing the numerical rank than the CPQR, and faster to compute than the SVD. Traditional algorithms for computing an RRUTV have been deterministic and guarantee revealing the rank of a matrix up to a user-defined tolerance. It is not used as widely as the aforementioned SVD and CPQR, though, except in a few settings such as subspace tracking.

\subsection{Challenges of implementing the SVD and the CPQR on a GPU}

As focus in high performance computing has shifted towards parallel environments, the use of GPUs to perform scientific computations has gained popularity and success \cite{owens2008gpu,luebke2008cuda,brodtkorb2010state}. The power of the GPU lies in its ability to execute many tasks in parallel extremely efficiently, and software tools have rapidly developed to allow developers to make full use of its capabilities. Algorithm design, however, is just as important. Classical algorithms for computing both the SVD and CPQR, still in use today, were designed with a heavier emphasis on reducing the number of floating point operations (flops) than on running efficiently on parallel systems. Thus, it is difficult for either factorization to maximally leverage the computing power of a GPU.


For CPQR, the limitations of the parallelism are well understood, at least relative to comparable matrix computations. The most popular algorithm for computing a CPQR uses Householder transformations and chooses the pivot columns by selecting the column with the largest norm. We will refer to this algorithm as \textsc{HQRCP}. See Section \ref{sec:ch5-cpqr} for a brief overview of \textsc{HQRCP}, or, \textit{e.g.}~\cite{businger1965linear,golub1965numerical} for a thorough description. {The process of selecting pivot columns inherently prevents full parallelization. In particular, \textsc{HQRCP} as written originally in \cite{golub1965numerical} uses no higher than Level-2 BLAS. Quintana-Ort\'{i} \textit{et al.}~developed \textsc{HQRCP} further in \cite{QRP:SIAM}, casting about half of the flops in terms of Level-3 BLAS kernels. Additional improvement in this area, though, is difficult to find for this algorithm.}
Given a sequence of matrix operations, it is well known that an appropriate implementation using Level-3 BLAS, or matrix-matrix, operations will run more efficiently on modern processors than an optimal implementation using Level-2 or Level-1 BLAS \cite{blackford2002updated}. This is largely due to the greater potential for the Level-3 BLAS to make more efficient use of memory caching in the processor.

The situation for the SVD is even more bleak. It is usually computed in two stages. The first is a reduction to bidiagonal form via, \textit{e.g.}~Householder reflectors. Only about half the flops in this computation can be cast in terms of the Level-3 BLAS, similarly (and for similar reasons) to \textsc{HQRCP}. The second stage is the computation of the SVD of the bidiagonal matrix. This is usually done with either an iterative algorithm (a variant of the QR algorithm) or a recursive algorithm (divide-and-conquer) which reverts to the QR algorithm at the base layer. See \cite{trefethen1997numerical,golub,cuppen1980divide,gu1995divide} for details. The recursive option inherently resists parallelization, and the current widely-used implementations of the QR approach are cast in terms of an operation that behaves like a Level-2 BLAS.\footnote{The QR algorithm can be cast in terms of Level-3 BLAS, but for reasons not discussed here, this approach has not yet been adopted in most software. See \cite{van2014restructuring} for details.} {Another well-known method for computing the SVD is the Jacobi's method~\cite{demmel1997applied,golub}, which can compute the tiny singular values and the corresponding singular vectors much more accurately for some matrices. But it is generally slower than the aforementioned methods.}

\subsection{Proposed algorithms}
\label{sec:ch5-proposed}

In this paper, we present two randomized algorithms for computing an RRUTV. Both algorithms are designed to run efficiently on GPUs in that the majority of their flops are cast in terms of matrix-matrix multiplication. We show through extensive numerical experiments in Section \ref{sec:ch5-num} that each reveals rank nearly as well as the SVD but often costs less than \textsc{HQRCP} to compute on a GPU. For matrices with uncertain or high rank, then, these algorithms warrant strong consideration for this computing environment.

The first algorithm \textsc{powerURV}, discussed in Section \ref{sec:ch5-powerURV}, was first introduced in the technical report \cite{gopal2018powerURV}. \textsc{powerURV} is built on another randomized RRUTV algorithm developed by Demmel \textit{et al.}~in \cite{demmel2007fast}, adding better rank revelation at a tolerable increase in computational cost. The algorithm itself is quite simple, capable of description with just a few lines of code. The simplicity of its implementation is a significant asset to developers, and it has just one input parameter, whose effect on the resulting calculation can easily be understood.

The second algorithm, \textsc{randUTV}, was first presented in \cite{martinsson2019randutv}. \textsc{randUTV} is a \textit{blocked} algorithm, meaning it operates largely inside a loop, ``processing'' multiple columns of the input matrix during each iteration. Specifically, for an input matrix $\mtx{A} \in \mathbb{R}^{m \times n}$ with $m \ge n$,\footnote{if $m < n$, we may simply factorize the transpose and take the transpose of the result.} a block size $b$ is chosen, and the bulk of \textsc{randUTV}'s work occurs in a loop of $s = \lceil n/b \rceil$ steps. During step $i$, orthogonal matrices $\mtx{U}^{\si}$ and $\mtx{V}^{\si}$ are computed which approximate singular vector subspaces of a relevant block of $\mtx{T}^{\simo}$. Then, $\mtx{T}^{\si}$ is formed with
\[
\mtx{T}^{\si} \defeq \left(\mtx{U}^{\si}\right)^* \mtx{T}^{\simo} \mtx{V}^{\si}.
\]
The leading $ib$ columns of $\mtx{T}^{\si}$ are upper  {trapezoidal} (see Figure \ref{fig:ch5-randutv-pattern} for an illustration of the sparsity pattern), so we say that \textsc{randUTV} drives $\mtx{A}$ to upper {trapezoidal} form $b$ columns at a time. After the final step in the loop, we obtain the final $\mtx{T}$, $\mtx{U}$, and $\mtx{V}$ factors with
\begin{align*}
\mtx{T} & \defeq \mtx{T}^{(s)}, \\
\mtx{U} & \defeq \mtx{U}^{(1)} \mtx{U}^{(2)} \cdots \mtx{U}^{(s)}, \\
\mtx{V} & \defeq \mtx{V}^{(1)} \mtx{V}^{(2)} \cdots \mtx{V}^{(s)}.
\end{align*}
See Section \ref{sec:ch5-randutvmod} for the full algorithm. A major strength of \textsc{randUTV} is that it may be adaptively stopped at any point in the computation, for instance when the singular value estimates on the diagonal of the $\mtx{T}^{\si}$ matrices drop below a certain threshold. If stopped early after $k \le \min(m,n)$ steps, the algorithm incurs only a cost of $\mathcal{O}(mnk)$ for an $m \times n$ input matrix. Each matrix $\mtx{U}^{\si}$ and $\mtx{V}^{\si}$ is computed using techniques similar to that of the randomized SVD \cite{halko2011finding}, which spends most of its flops in matrix multiplication and therefore makes efficient use of GPU capabilities.

In this paper, we propose several modifications to the original \textsc{randUTV} algorithm given in \cite{martinsson2019randutv}.
{
In particular, we add oversampling and orthonormalization to enhance the accuracy of the rank-revealing properties of the resulting RRUTV factorization. {These changes lead to additional computational cost on a CPU as observed in \cite{martinsson2019randutv}.} Here, we introduce an efficient algorithm to minimize the additional cost of oversampling and orthonormalization. The new algorithm takes advantage of the fact that matrix-matrix multiplication is far more efficient on a GPU than unpivoted QR, \textsc{randUTV}'s other building block.
}

{In summary, we present \textsc{powerURV} and \textsc{randUTV} for computing rank-revealing factorizations on a GPU. Both methods are much faster than the SVD. Compared to \textsc{HQRCP}, they are faster for sufficiently large matrices and much more accurate. As an example, Figure \ref{fig:ch1-times} shows the running time of the four methods on two Intel 18-core CPUs and an NVIDIA GPU.}

\begin{figure}[h]
\centering
\includegraphics[width=0.48\textwidth]{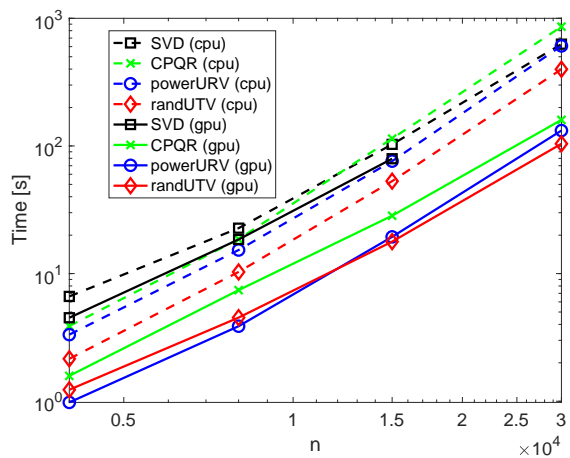}
\caption{Computation time of the SVD, \textsc{HQRCP}, \textsc{powerURV}, and \textsc{randUTV} for an $n \times n$ input matrix (SVD ran out of memory for the largest matrix on the GPU). CPU: two Intel Xeon Gold 6254 18-core CPUs at 3.10GHz; GPU: NVIDIA Tesla V100. {Results for SVD and \textsc{HQRCP} were obtained through the Intel MKL library on the CPU and the MAGMA library~\cite{tdb10,tnld10,dghklty14} on the GPU.}}
\label{fig:ch1-times}
\end{figure}

\begin{remark}
In this manuscript, we {assume that the input and output matrices reside in CPU main memory in our numerical experiments, and reported compute times \textit{include} the communication time for transferring data to and from the GPU. The storage complexity of all methods discussed is $O(n^2)$ for an $n \times n$ matrix.} We restrict our attention to the case where all the data used in the computation fits in RAM on the GPU, which somewhat limits the size of matrices that can be handled. For instance, in the numerical experiments reported in Section \ref{sec:ch5-num}, the largest problem size we could handle involved matrices of size about $30\,000 \times 30\,000$. The techniques can be modified to allow larger matrices to be handled and for multiple GPUs to be deployed, but we leave this extension for future work.
\end{remark}

\subsection{Outline of paper}

In Section \ref{sec:ch5-prel}, we survey previous work in rank-revealing factorizations, discussing competing methods as well as previous work in randomized linear algebra that is foundational for the methods presented in this article. Section \ref{sec:ch5-powerURV} presents the first algorithmic contribution of this article, \textsc{powerURV}. In Section \ref{sec:ch5-randutvmod}, we discuss and build on the recently developed \textsc{randUTV} algorithm, culminating in a modification of the algorithm \textsc{randUTV\_boosted} with greater potential for low rank estimation. Finally, Section \ref{sec:ch5-num} presents numerical experiments which demonstrate the computational efficiency of \textsc{powerURV} and \textsc{randUTV\_boosted} as well as their effectiveness in low rank estimation.

\section{Preliminaries}
\label{sec:ch5-prel}

\subsection{Basic notation}
\label{sec:ch5-notation}

In this manuscript, we write $\mtx{A} \in \mathbb{R}^{m \times n}$ to denote a real-valued matrix with $m$ rows and $n$ columns, and $\mtx{A}(i,j)$ refers to the element in the $i$-th row and $j$-th column of $\mtx{A}$. The indexing notation $\mtx{A}(i:j,k:l)$ is used to reference the submatrix of $\mtx{A}$ consisting of the entries in the $i$-th through $j$-th rows of the $k$-th through $l$-th columns. $\sigma_i(\mtx{A})$ is the $i$-th singular value of $\mtx{A}$, and $\mtx{A}^*$ is the transpose. The row and column spaces of $\mtx{A}$ are denoted as Row($\mtx{A}$) and Col($\mtx{A}$), respectively. An orthonormal matrix is a matrix whose columns have unit norm and are pairwise orthogonal, and an orthogonal matrix is a square orthonormal matrix. The default norm $\| \cdot \|$ is the spectral norm. If all the entires of a matrix $\mtx{G}  \in \mathbb{R}^{m \times n}$ are independent, identically distributed standard Gaussian variables, we call $\mtx{G}$ a standard Gaussian matrix, and we may denote it as $\mtx{G} = \textsc{randn}(m, n)$. $\epsilon_{\text{machine}}$ denotes the machine epsilon, say, $2.22 \times 10^{-16}$ in IEEE double precision

\subsection{The singular value decomposition (SVD)}
\label{sec:ch5-svd}


Given a matrix $\mtx{A} \in \mathbb{R}^{m \times n}$ and $r = \min(m,n)$, the (full) SVD of $\mtx{A}$ takes the form
\[
\begin{array}{ccccc}
\mtx{A} & = & \mtx{U}_{\text{opt}} & \mtx{\Sigma} & \mtx{V}_{\text{opt}}^*, \\
m \times n & & m \times m & m \times n & n \times n
\end{array}
\]
where $\mtx{U}_{\text{opt}}$ and $\mtx{V}_{\text{opt}}$ are orthogonal, and $\mtx{\Sigma}$ is (rectangular) diagonal. The diagonal elements $\{\sigma_{i}\}_{i=1}^{r}$ of $\mtx{\Sigma}$ are the \textit{singular values} of $\mtx{A}$ and satisfy $\sigma_1 \ge \sigma_2 \ge \ldots \ge \sigma_r \geq 0$. The columns $\mtx{u}_i$ and $\mtx{v}_i$ of $\mtx{U}_{\text{opt}}$ and $\mtx{V}_{\text{opt}}$ are called the \textit{left} and \textit{right singular vectors}, respectively, of $\mtx{A}$.
 In this article, we write $[\mtx{U}_{\text{opt}}, \mtx{\Sigma}, \mtx{V}_{\text{opt}}] =  \text{SVD}(\mtx{A})$ for computing the (full) SVD decomposition of $\mtx{A}$.
Importantly, the SVD provides theoretically optimal rank-$k$ approximations to $\mtx{A}$. Specifically, the Eckart-Young-Mirsky Theorem \cite{eckart1936approximation,mirsky1960symmetric} states that given the SVD of a matrix $\mtx{A}$ and a fixed $k \in \{1,2,\dots,\,r\}$, we have that
\[
\| \mtx{A} - \mtx{U}_{\text{opt}}(:,1:k) \mtx{\Sigma}(1:k,1:k) \mtx{V}_{\text{opt}}(:,1:k)^* \| = \inf \{ \| \mtx{A} - \mtx{B} \| : \mtx{B} \text{ has rank } k \},
\]
{and the relation also holds with respect to the Frobenius norm.}

The thin SVD of $\mtx{A}$ takes the form
\[
\begin{array}{ccccc}
\mtx{A} & = & \hat{\mtx{U}}_{\text{opt}} & \hat{\mtx{\Sigma}} & \hat{\mtx{V}}_{\text{opt}}^*, \\
m \times n & & m \times r & r \times r & r \times n
\end{array}
\]
where $\hat{\mtx{U}}_{\text{opt}}$ and $\hat{\mtx{V}}_{\text{opt}}$ are orthonormal, and $\hat{\mtx{\Sigma}}$ is diagonal containing the singular values.

\subsection{The column pivoted QR (CPQR) decomposition}
\label{sec:ch5-cpqr}

Given a matrix $\mtx{A} \in \mathbb{R}^{m \times n}$, the (full) CPQR decomposition of $\mtx{A}$ takes the form
\[
\begin{array}{ccccc}
\mtx{A} & = & \mtx{Q} & \mtx{R} & \mtx{P}^*, \\
m \times n & & m \times m & m \times n & n \times n
\end{array}
\]
where $\mtx{Q}$ is orthogonal, $\mtx{R}$ is trapezoidal, and $\mtx{P}$ is a permutation matrix. {There exists a number of algorithms for choosing the permutation,} but a general option, as implemented in LAPACK\footnote{\url{https://www.netlib.org/lapack/lug/node42.html}},  ensures monotonic decay in magnitude of the diagonal entries of $\mtx{R}$ so that $|\mtx{R}(1,1)| \ge |\mtx{R}(2,2)| \ge \ldots \ge |\mtx{R}(r,r)|$.
The details of the most popular algorithm for computing such a factorization, called \textsc{HQRCP} hereafter, are not essential to this article, but they may be explored by the reader in, \textit{e.g.}~\cite{chan1987rank,golub,trefethen1997numerical,stewart1998matrix}. In this article, we write $[\mtx{Q}, \mtx{R}, \mtx{P}] =  \textsc{HQRCP}(\mtx{A})$ for computing the CPQR decomposition of $\mtx{A}$.

{It is well known that \textsc{HQRCP} is not guaranteed to be   rank-revealing, and  it can fail by an exponentially large factor (on, \textit{e.g.} Kahan matrices)~\cite{gu1996efficient}.}
Such pathological cases are rare, particularly in practice, and \textsc{HQRCP} is so much faster than computing an SVD that \textsc{HQRCP} is used ubiquitously for low rank approximation. 
A communication avoiding variant based on ``tournament pivoting'' is given in~\cite{demmel2015communication}, with a related method for fully pivoted LU described in \cite{grigori2018low}.

\subsection{The unpivoted QR decomposition}
\label{sec:qr}

Given a matrix $\mtx{A} \in \mathbb{R}^{m \times n}$, the full (unpivoted) QR decomposition of $\mtx{A}$ takes the form
\[
\begin{array}{ccccc}
\mtx{A} & = & {\mtx{Q}} & {\mtx{R}} \\
m \times n & & m \times m & m \times n
\end{array}
\]
where ${\mtx{Q}}$ is orthogonal, and ${\mtx{R}}$ is trapezoidal. When $m \ge n$, the thin (unpivoted) QR decomposition of $\mtx{A}$ takes the form
\[
\begin{array}{ccccc}
\mtx{A} & = & \hat{\mtx{Q}} & \hat{\mtx{R}} \\
m \times n & & m \times n & n \times n
\end{array}
\]
where $\hat{\mtx{Q}}$ is orthonormal, and $\hat{\mtx{R}}$ is upper triangular. Unpivoted QR decompositions have no rank-revealing properties, but in this article we make critical use of the fact that {if $m \ge n$ and $\mtx{A}$ has full column rank,} then the columns of $\hat{\mtx{Q}}$ form an orthonormal basis for Col($\mtx{A}$).

The standard algorithm for computing an unpivoted QR factorization relies on Householder reflectors. We shall call this algorithm as \textsc{HQR} in this article and write $[{\mtx{Q}}, {\mtx{R}}] = \textsc{HQR\_full}(\mtx{A})$ or $[\hat{\mtx{Q}}, \hat{\mtx{R}}] = \textsc{HQR\_thin}(\mtx{A})$.  We refer the reader once again to textbooks such as \cite{golub,trefethen1997numerical,stewart1998matrix} for a complete discussion of the algorithm. The lack of the pivoting also allows \textsc{HQR} to rely more heavily on the level-3 BLAS than \textsc{HQRCP}, translating to better performance in parallel environments.

Of particular interest is the fact that the output orthogonal matrix ${\mtx{Q}}$ of the \textsc{HQR\_full} algorithm can be stored and applied efficiently even when $m \gg n$.
In \textsc{HQR\_full}, suppose that we have determined $n$ Householder matrices $\mtx{H}_1,\mtx{H}_2,\ldots, \mtx{H}_n \in \mathbb{R}^{m \times m}$ such that
$
\mtx{H}_n^* \mtx{H}_{n-1}^* \cdots \mtx{H}_1^* \mtx{A} = \mtx{R}.
$
We have that
$
\mtx{H}_i = \mtx{I} - 2 \mtx{y}_i \mtx{y}_i^*, \; 1 \le i \le n,
$
where $\mtx{y}_i  \in \mathbb{R}^{m \times 1}$ is the Householder vector associated with the transformation. Then the matrix $\mtx{Q} = \mtx{H}_1 \mtx{H}_2 \cdots \mtx{H}_n$ can be represented as
\[
\mtx{Q} = \mtx{I} - \mtx{Y} \mtx{T} \mtx{Y}^*,
\]
where $\mtx{T} \in \mathbb{R}^{n \times n}$ is upper triangular and $\mtx{Y} \in \mathbb{R}^{m \times n}$ is lower trapezoidal with columns containing the $\mtx{y}_i$ \cite{schreiber1989storage}. The form $\mtx{I} - \mtx{Y} \mtx{T} \mtx{Y}^*$ of $\mtx{Q}$ is called the \textit{compact-WY} representation of a product of Householder reflectors.

\begin{remark}\label{rmk:qr}
{Observe that the compact-WY form reduces the storage requirement of $\mtx{Q}$ from $\mathcal{O}(m^2)$ to $\mathcal{O}(mn + n^2)$ and that the \textsc{HQR\_full} algorithm requires $O(mn^2)$ work.} More importantly,   this representation  of $\mtx{Q}$  allows an {efficient} application of $\mtx{Q}$ using matrix-matrix multiplications, which is crucial for efficiently building factorizations on a GPU.
\end{remark}

{If matrix $\mtx{A}$ is rank-deficient, i.e., $\mbox{rank}(\mtx{A}) = k < \min(m, n)$, but the first $k$ columns in $\mtx{A}$ are linearly independent, then the \textsc{HQR} algorithm will detect this situation as $\mtx{R}(k+1:m, k+1:n) = \mtx{0}$ and $\hat{\mtx{R}}(k+1:n, k+1:n) = \mtx{0}$ in exact arithmetic.}

\subsection{The randomized range finder}
\label{sec:ch5-rsvd}

{Consider a matrix $\mtx{A} \in \mathbb{R}^{m \times n}$ and a block size $b$ such that $1 \le b \le \text{rank}(\mtx{A})$. The randomized range finder algorithm~\cite{halko2011finding} computes a set of $b$ orthonormal vectors that approximately span Row($\mtx{A}$) or Col($\mtx{A}$).} To be precise, suppose we seek to find an orthonormal matrix  $\hat{\mtx{Q}} \in \mathbb{R}^{n \times b}$ that
\begin{equation}
\| \mtx{A} - \mtx{A} \hat{\mtx{Q}} \hat{\mtx{Q}}^* \| \approx \inf \{ \| \mtx{A} - \mtx{B} \| : \mtx{B} \text{ has rank } b \}.
\end{equation}
In other words, the columns of $\hat{\mtx{Q}}$ approximately span the same space as the dominant $b$ right singular vectors of $\mtx{A}$. This task can be accomplished using randomized projections. An extremely simple way of building $\hat{\mtx{Q}}$ is the following:
\begin{enumerate}[1.]
\item Generate a standard Gaussian matrix $\mtx{G} \in \mathbb{R}^{m \times b}$.
\item Compute a ``sampling matrix'' $\mtx{Y} = \mtx{A}^* \mtx{G} \in \mathbb{R}^{n \times b}$.
\item Build an orthonormal basis of Col($\mtx{Y}$) via $[\hat{\mtx{Q}}, \sim] = \textsc{HQR\_thin}(\mtx{Y})$\footnote{{According to Theorem~\ref{th:rank} in Appendix~\ref{a:proof}, matrix $\mtx{Y}$ has full rank with probability 1.}}.
\end{enumerate}
This method will yield a reasonably good approximation {when the singular values of $\mtx{A}$ decay fast  (see theoretical error bounds in, \textit{e.g.}~\cite[Section 10]{halko2011finding})}. However, for certain matrices, particularly when the decay in singular values is slow, an improved approximation will be desired. We may improve the approximation provided by $\hat{\mtx{Q}}$ in two ways:
\begin{enumerate}[(i)]
    \item \textbf{Oversampling:} We may interpret $\mtx{Y}$ as $b$ different random projections onto Row($\mtx{A}$). As shown in \cite{halko2011finding}, the approximation of Col($\mtx{Y}$) to Row($\mtx{A}$) may be improved by gathering a few, say $p$, extra projections. In practice $p=5$ or $p=10$ is sufficient, adding very little additional computational cost {($O \big((mn+nb)p \big)$ extra work)} to the algorithm. Using a small amount of oversampling also improves the expected error bounds and vastly decreases the probability of deviating from those bounds. {For example, assume $b \ge 2$ and $p \ge4$ while $b+p \le \min\{m,n\}$, it holds that~\cite[Corollary 10.9]{halko2011finding}}
\[
\| \mtx{A} - \mtx{A} \hat{\mtx{Q}} \hat{\mtx{Q}}^* \|
\le \left(1+16 \sqrt{1+\frac{b}{p+1}} \right) \sigma_{b+1} + \frac{8\sqrt{k+p}}{p+1} \left(\sum_{j>b} \sigma_{j}^2 \right)^{1/2},
\]
    with failure probability at most $3e^{-p}$, where $\sigma_{1} \ge \sigma_{2} \ge \ldots$ are the singular values of $\mtx{A}$. Thus, this technique makes the uses of randomization in the algorithm safe and reliable.
\item \textbf{Power iteration:}  In the construction of $\mtx{Y}$, we may replace the matrix $\mtx{A}^*$ with $(\mtx{A}^* \mtx{A})^q \mtx{A}^*$ for $q \in \mathbb{N}$, {a matrix with the same column space as $\mtx{A}^*$} but whose singular values are $(\sigma_i(\mtx{A}))^{2q+1}$. The contributions of singular vectors that correspond to smaller singular values will be diminished using this new matrix,
{thus improving the approximation of Col($\mtx{Y}$) to the desired singular vector subspace.}
\footnote{{See~\cite{saibaba2019randomized} for bounds on principle angles between the two subspaces.}}  {Again, assume $b \ge 2$ and $p \ge4$ while $b+p \le \min\{m,n\}$, it holds that~\cite[Corollary 10.10]{halko2011finding}}
\[
\mathbb{E}\| \mtx{A} - \mtx{A} \hat{\mtx{Q}} \hat{\mtx{Q}}^* \|
\le C^{1/(2q+1)}  \sigma_{b+1},
\]
{where the constant $C$ depends on $k$, $p$, and $\min\{m,n\}$. In other words, the power scheme drives the extra constant factor to one exponentially fast. In practice, the above error bound can be too pessimistic,} and choosing $q$ to be $0$, $1$ or $2$ gives excellent results.  When using a power iteration scheme, numerical linear dependence of the samples becomes a concern since the information will be lost from singular vectors corresponding to singular values less than {$\epsilon_{\text{machine}}^{1/(2q+1)} \sigma_{\max}(\mtx{A}) $}. To stabilize the algorithm, we build $\mtx{Y}$ incrementally, orthonormalizing the columns of the intermediate matrix in between applications of $\mtx{A}^*$ and $\mtx{A}$.
\end{enumerate}

\section{The \textsc{powerURV} algorithm.}
\label{sec:ch5-powerURV}

Let $\mtx{A} \in \mathbb{R}^{m \times n}$ be a given matrix. Without loss of generality, we assume $m \ge n$; otherwise, we may operate on $\mtx{A}^*$ instead.
The \textsc{powerURV} algorithm is a slight variation of a randomized algorithm proposed by
Demmel \textit{et al.}~\cite{demmel2007fast} for  computing a
rank-revealing URV factorization of $\mtx{A}$. The modification we propose makes the
method more accurate in revealing the numerical rank, at only a marginal increase in
computational time in a GPU environment. In addition, \textsc{powerURV} is remarkably simple to implement and has a close connection to the randomized SVD (RSVD)~\cite{halko2011finding}.
%

\subsection{A randomized algorithm for computing a UTV factorization proposed by Demmel, Dumitriu, and Holtz}
\label{sec:ch5-rurv}

In \cite{demmel2007fast}, Demmel \textit{et al.}~give a randomized algorithm RURV (randomized URV) for computing a rank-revealing factorization of matrix $\mtx{A}$.
 The algorithm can be written quite simply as follows:
\begin{enumerate}[1.]
\item Generate a standard Gaussian matrix $\mtx{G} \in \mathbb{R}^{n \times n}$.
\item Build an orthonormal basis of Col($\mtx{G}$) via $[\mtx{V}, \sim] = \textsc{HQR\_full}(\mtx{G})$\footnote{According to Theorem~\ref{th:invert} in Appendix~\ref{a:proof}, $\mtx{G}$ is invertible.}, where $\mtx{V} \in \mathbb{R}^{n \times n}$.
\item Compute $\hat{\mtx{A}} = \mtx{A} \mtx{V} \in \mathbb{R}^{m \times n}$.
\item {Perform the full (unpivoted) QR factorization $[\mtx{U}, \mtx{R}] = \textsc{HQR\_full}(\hat{\mtx{A}})$,\footnote{{According to Corollary~\ref{th:rank2} in Appendix~\ref{a:proof}, we know that, with probability 1, $\mbox{rank}(\hat{\mtx{A}}) = \mbox{rank}({\mtx{A}})$ and the first $\mbox{rank}({\mtx{A}})$ columns in $\hat{\mtx{A}}$ are linearly independent.
}} where $\mtx{U} \in \mathbb{R}^{m \times m}$ is orthogonal and  $\mtx{R} \in \mathbb{R}^{m \times n}$ is upper trapezoidal.}
\end{enumerate}
Note that after step (4), we have
\[
\mtx{A} \mtx{V} = \mtx{U} \mtx{R}\qquad \Rightarrow\qquad \mtx{A} = \mtx{U} \mtx{R} \mtx{V}^*,
\]
a URV decomposition of $\mtx{A}.$ (See~\cite[Sec.~5]{demmel2007fast} for its rank-revealing properties.)
%
%
%
A key advantage of this algorithm is its simplicity. Since it relies only on unpivoted QR and matrix multiplication computations, it can easily be shown to be stable (see \cite{demmel2007fast} for the proof). Furthermore, both of these operations are relatively well-suited for parallel computing environments like GPUs. Since they are extremely common building blocks for problems involving matrices, highly optimized implementations for both are readily available.
Thus, a highly effective implementation of RURV may be quickly assembled by a non-expert.

Demmel \textit{et al.}~find a use for RURV as a fundamental component of a fast, stable solution to the eigenvalue problem. For this application, RURV is used iteratively, so not much is required of the accuracy of the rank-revealing aspect of the resulting factorization. For other problems  such as low-rank approximations, though,
{the error $e_k$ in (\ref{e:rr}) can be quite large (compared to, e.g., the \textsc{powerURV} algorithm that we are going to introduce; for interested readers, see~\cite[Figure 3]{gopal2018powerURV} for a comparison of errors between RURV and \textsc{powerURV}).}
%
This is because the matrix $\mtx{V}$ does not incorporate any information from the row space of the input $\mtx{A}$.

\subsection{\textsc{powerURV}: A randomized algorithm enhanced by power iterations}
\label{sec:ch5-powerURValg}

The \textsc{powerURV} algorithm is inspired by the RURV of Section \ref{sec:ch5-rurv} combined with the observation that the optimal matrix $\mtx{V} \in \mathbb{R}^{n \times n}$ to use for rank-revealing purposes (minimizing the error $e_k$ in (\ref{e:rr})) is the matrix whose columns are the right singular vectors of the input matrix $\mtx{A} \in \mathbb{R}^{m \times n}$. If such a $\mtx{V}$ were used, then the columns of $\hat{\mtx{A}} = \mtx{A} \mtx{V}$ would be the left singular vectors of $\mtx{A}$ scaled by its singular values. Thus finding the right singular vectors of $\mtx{A}$ yields theoretically optimal low rank approximations. This subproblem is as difficult as computing the SVD of $\mtx{A}$, though, so we settle for choosing $\mtx{V}$ as an efficiently computed approximation to the right singular vectors of $\mtx{A}.$

{Specifically, we compute an approximation to the row space of $\mtx{A}$ using a variant of the randomized range finder algorithm in Section \ref{sec:ch5-rsvd}.} The computation of $\mtx{V}$ consists of three steps:
\begin{enumerate}[1.]
\item Generate a standard Gaussian matrix  $\mtx{G} \in \mathbb{R}^{n \times n}$.
\item Compute a ``sampling matrix'' $\mtx{Y} = (\mtx{A}^* \mtx{A})^q \mtx{G} \in \mathbb{R}^{n \times n}$, where $q$ is a small positive integer.
\item Build an orthonormal basis of Col($\mtx{Y}$) via $[\mtx{V} , \sim] = \textsc{HQR\_full}(\mtx{Y})$.\footnote{{According to Theorem~\ref{th:rank3} and Corollary~\ref{th:rank2} in Appendix~\ref{a:proof},
we know that, with probability 1,  $\mbox{rank}(\mtx{Y}) = \mbox{rank}(\mtx{A})$, and the first $\mbox{rank}(\mtx{A})$ columns in ${\mtx{Y}}$ are linearly independent.
}}
\end{enumerate}

The matrix $\mtx{Y}$ can be thought of as a random projection onto the row space of $\mtx{A}$. Specifically, the columns of $\mtx{Y}$ are random linear combinations of the columns of $(\mtx{A}^*\mtx{A})^q$. {The larger the value of $q$, the faster the singular values of the sampled matrix $(\mtx{A}^*\mtx{A})^q$ decay. Therefore, choosing $q$ to be large makes $\mtx{V}$ better aligned with the right singular vectors. It also increases the number of matrix multiplications required, but they execute efficiently on a GPU.} 

Unfortunately, a naive evaluation of $\mtx{Y} = (\mtx{A}^* \mtx{A})^q \mtx{G}$ is {sensitive to the effects of roundoff errors.
In particular, the columns of $\mtx{Y}$ will lose the information provided by singular vectors with corresponding singular values smaller than {$\epsilon_{\text{machine}}^{1/2q} \sigma_{\max}(\mtx{A})$}. This problem can be remedied by orthonormalizing the columns of the {intermediate} matrices in between each application of $\mtx{A}$ and $\mtx{A}^*$, employing $2q-1$ extra unpivoted QR factorizations. The complete instruction set for \textsc{powerURV} is given in Algorithm \ref{alg:powerURV-stable}.
%

 \begin{algorithm}
 \caption{[\mtx{U},\mtx{R},\mtx{V}] = \textsc{powerURV}($\mtx{A},q$)}
 \label{alg:powerURV-stable}
 \begin{algorithmic}[1]
 \REQUIRE {matrix $\mtx{A} \in \mathbb{R}^{m \times n} (m \ge n)$ and a non-negative integer $q$ (if $q=0$, this algorithm becomes the RURV.)}
 \ENSURE {$\mtx{A} = \mtx{U} \mtx{R} \mtx{V}^*$, where $\mtx{U} \in \mathbb{R}^{m \times m}$  and $\mtx{V} \in \mathbb{R}^{n \times n}$ are orthogonal, and $\mtx{R} \in \mathbb{R}^{m \times n}$ is upper trapezoidal.}
 \STATE {$\mtx{V} = \textsc{randn}(n,n)$}
 \FOR{$i=1$: $q$}
 \STATE $\hat{\mtx{Y}} = \mtx{A} \mtx{V}$
 \hfill  $\hat{\mtx{Y}} \in \mathbb{R}^{m \times n}$
 \STATE $[\hat{\mtx{V}},\sim] = \textsc{HQR\_thin}(\hat{\mtx{Y}})$
 \hfill  $\hat{\mtx{V}} \in \mathbb{R}^{m \times n}$
 \STATE $\mtx{Y} = \mtx{A}^* \hat{\mtx{V}}$
 \hfill  ${\mtx{Y}} \in \mathbb{R}^{n \times n}$
 \STATE $[\mtx{V},\sim] = \textsc{HQR\_thin}(\mtx{Y})$
 \hfill  ${\mtx{V}} \in \mathbb{R}^{n \times n}$
 \ENDFOR
 \STATE $\hat{\mtx{A}} = \mtx{A} \mtx{V}$
 \hfill  $\hat{\mtx{A}} \in \mathbb{R}^{m \times n}$
 \STATE $[\mtx{U},\mtx{R}] = \textsc{HQR\_full}(\hat{\mtx{A}})$ \\
 \hfill \footnotesize{Remark: According to theorems in Appendix~\ref{a:proof}, we know that, with probability 1, $\mbox{rank}(\hat{\mtx{Y}}) = \mbox{rank}({\mtx{Y}}) = \mbox{rank}(\hat{\mtx{A}}) = \mbox{rank}(\mtx{A})$, and the first $\mbox{rank}(\mtx{A})$ columns in these matrices are linearly independent, so the \textsc{HQR} algorithm can be applied.}
 \end{algorithmic}
 \end{algorithm}

\subsection{Relationship with RSVD}
\label{sec:ch5-rsvd-vs-powerURV}
The \textsc{powerURV} algorithm is closely connected to the standard randomized
singular value decomposition algorithm (RSVD).
The equivalency between RSVD and \textsc{powerURV} allows for much of the
theory for analyzing the RSVD in \cite{halko2011finding,saibaba2019randomized} to directly apply to the \textsc{powerURV}
algorithm.

Let $\mtx{A} \in \mathbb{R}^{m \times n}$ with $m \geq n$ and $\mtx{G} \in \mathbb{R}^{n \times n}$ be a standard Gaussian matrix. Let $\mtx{G}_{\rm rsvd} = \mtx{G}(:,1:\ell)$, where $\ell \le \min(n, \mbox{rank}(\mtx{A}))$ is a positive integer. Let $q$ be a {non-negative} integer.


\paragraph{\textbf{\textsc{powerURV}}}
Recall that \textsc{powerURV} has two steps. First, compute the full (unpivoted) QR factorization of $(\mtx{A}^* \mtx{A})^q \mtx{G}$:
\begin{equation}
(\mtx{A}^* \mtx{A})^q \mtx{G} = \mtx{V} \mtx{S},
\label{eq:powerurv1}
\end{equation}
where $\mtx{V} \in \mathbb{R}^{n \times n}$ is orthogonal and $\mtx{S} \in \mathbb{R}^{n \times n}$ is upper triangular.
Second, compute the full (unpivoted) QR factorization of $\mtx{A} \mtx{V}$:
\begin{equation}
\mtx{A} \mtx{V} = \mtx{U} \mtx{R},
\label{eq:powerurv2}
\end{equation}
where $\mtx{U} \in \mathbb{R}^{m \times m}$ is orthogonal and $\mtx{R} \in \mathbb{R}^{m \times n}$ is upper trapezoidal.

\paragraph{\textbf{RSVD}}
The RSVD builds an
approximation to a truncated SVD via the following steps. First, evaluate
\begin{equation}
\label{eq:rsvd1}
  \mtx{Y}_{\rm rsvd} = (\mtx{A} \mtx{A}^*)^q \mtx{A} \mtx{G}_{\rm rsvd} \in \mathbb{R}^{m \times \ell}.
\end{equation}
The columns of $\mtx{Y}_{\rm rsvd}$ are orthonormalized via a thin (unpivoted) QR factorization\footnote{{It is easy to show that $\mbox{rank}((\mtx{A} \mtx{A}^*)^q \mtx{A}) = \mbox{rank}(\mtx{A}) \ge \ell$, so $\mtx{Y}_{\rm rsvd}$ has full rank with probability 1 according to Theorem~\ref{th:rank} in Appendix~\ref{a:proof}.}}
\begin{equation}
\label{eq:rsvd2}
\mtx{Y}_{\rm rsvd} = \mtx{Q}_{\rm rsvd} \mtx{R}_{\rm rsvd},
\end{equation}
where $\mtx{Q}_{\rm rsvd} \in \mathbb{R}^{m \times \ell}$ is orthonormal, and $\mtx{R}_{\rm rsvd} \in \mathbb{R}^{\ell \times \ell}$ is upper triangular and invertible.
Next, the thin SVD of $\mtx{Q}_{\rm rsvd}^* \mtx{A}$ is computed to obtain
\begin{equation}
\label{eq:rsvd3}
\mtx{Q}_{\rm rsvd}^* \mtx{A} = \mtx{W}_{\rm rsvd}\mtx{\Sigma}_{\rm rsvd}(\mtx{V}_{\rm rsvd})^*,
\end{equation}
where $\mtx{W}_{\rm rsvd} \in \mathbb{R}^{\ell \times \ell}$ is orthogonal, $\mtx{V}_{\rm rsvd} \in \mathbb{R}^{n \times \ell}$ is orthonormal
and $\mtx{\Sigma}_{\rm rsvd} \in \mathbb{R}^{\ell \times \ell}$ is
diagonal with singular values. We know all the singular values are positive because the columns of $\mtx{Q}_{\rm rsvd}$ lie in Col$(\mtx{A})$. The final step is to define the $m\times \ell$ matrix
\begin{equation}
\label{eq:rsvd4}
\mtx{U}_{\rm rsvd} = \mtx{Q}_{\rm rsvd} \mtx{W}_{\rm rsvd}.
\end{equation}
The end result is an approximate SVD:
$$
\mtx{A} \approx \mtx{U}_{\rm rsvd}\mtx{\Sigma}_{\rm rsvd}\mtx{V}_{\rm rsvd}^{*}.
$$

The key claim in this section is the following theorem:

\begin{thm} \label{th:rsvd}
Given two matrices $\mtx{A}$, $\mtx{G}$, and two integers $\ell$, $q$, as described at the beginning of Section~\ref{sec:ch5-rsvd-vs-powerURV}, it holds that
$$
\mtx{U}(:,1:\ell)\mtx{U}(:,1:\ell)^{*}\mtx{A} =
\mtx{U}_{\rm rsvd}\mtx{U}_{\rm rsvd}^{*}\mtx{A},
$$
where the two matrices $\mtx{U}$ and $\mtx{U}_{\rm rsvd}$ are computed by the \textsc{powerURV} and the RSVD, respectively, {in exact arithmetic}.
\end{thm}

\begin{proof}

We will prove that $\mbox{Col}(\mtx{U}(:,1:\ell)) = \mbox{Col}(\mtx{U}_{\rm rsvd})$, which
immediately implies that the two projectors
$\mtx{U}(:,1:\ell)\mtx{U}(:,1:\ell)^{*}$ and
$\mtx{U}_{\rm rsvd}\mtx{U}_{\rm rsvd}^{*}$ are identical. Let us first
observe that restricting (\ref{eq:powerurv1}) to the first $\ell$ columns, we obtain
\begin{equation}
\label{eq:swim0}
(\mtx{A}^{*}\mtx{A})^{q}\mtx{G}_{\rm rsvd} =
\mtx{V}(:,1:\ell) \mtx{S} (1:\ell,1:\ell).
\end{equation}
We can then connect $\mtx{Y}_{\rm rsvd}$ and $\mtx{U}(:,1:\ell)$ via a simple computation
\begin{equation}
\label{eq:swim1}
\begin{aligned}
\mtx{Y}_{\rm rsvd} & \stackrel{(\ref{eq:rsvd1})}{=}
(\mtx{A}\mtx{A}^*)^{q} \mtx{A} \mtx{G}_{\rm rsvd} \\ & =
\mtx{A}(\mtx{A}^*\mtx{A})^q \mtx{G}_{\rm rsvd} \\ & \stackrel{(\ref{eq:swim0})}{=}
\mtx{A}\mtx{V}(:,1:\ell) \mtx{S} (1:\ell,1:\ell) \\ & \stackrel{(\ref{eq:powerurv2})}{=}
\mtx{U}(:,1:\ell)\mtx{R}(1:\ell,1:\ell) \mtx{S} (1:\ell,1:\ell).
\end{aligned}
\end{equation}
Next we link $\mtx{U}_{\rm rsvd}$ and $\mtx{Y}_{\rm rsvd}$ via
\begin{equation}
\label{eq:swim2}
\mtx{U}_{\rm rsvd} \stackrel{(\ref{eq:rsvd4})}{=}
\mtx{Q}_{\rm rsvd} \mtx{W}_{\rm rsvd} \stackrel{(\ref{eq:rsvd2})}{=}
\mtx{Y}_{\rm rsvd}\mtx{R}_{\rm rsvd}^{-1}\mtx{W}_{\rm rsvd}.
\end{equation}
Combining (\ref{eq:swim1}) and (\ref{eq:swim2}), we find that
\begin{equation}
\label{eq:swim3}
\mtx{U}_{\rm rsvd} =
\mtx{U}(:,1:\ell)\mtx{R}(1:\ell,1:\ell) \mtx{S} (1:\ell,1:\ell)
\mtx{R}_{\rm rsvd}^{-1}\mtx{W}_{\rm rsvd}.
\end{equation}
We know that, with probability 1, matrices
$\mtx{R}(1:\ell,1:\ell), \mtx{S} (1:\ell,1:\ell), \mtx{R}_{\rm rsvd}$ and $\mtx{W}_{\rm rsvd}$
are invertible, which establishes that $\mbox{Col}(\mtx{U}(:,1:\ell)) = \mbox{Col}(\mtx{U}_{\rm rsvd})$.
\end{proof}

It is important to note that while there is a close connection between RSVD and \textsc{powerURV},
the RSVD is able to attain substantially higher overall accuracy than \textsc{powerURV}. {Notice that the RSVD requires
$2q+2$ applications of either $\mtx{A}$ or $\mtx{A}^{*}$, while \textsc{powerURV} requires only $2q+1$. The RSVD takes advantage of one additional application of $\mtx{A}$ in (\ref{eq:rsvd3}), where matrix $\mtx{W}_{\rm rsvd}$ rearranges the columns inside $\mtx{Q}_{\rm rsvd}$ to make the leading columns much better aligned with the corresponding singular vectors.}
%
Another perspective to see it is that
the columns of $\mtx{V}_{\rm rsvd}$ end up being a far more accurate basis for
Row($\mtx{A}$) than the columns of the matrix $\mtx{V}$ resulting from the \textsc{powerURV}. This is of course
expected since for $q=0$, the matrix $\mtx{V}$ incorporates no information from $\mtx{A}$ at all.

\begin{remark}
Theorem~\ref{th:rsvd} extends to the case when $\ell > \mbox{rank}(\mtx{A})$. The same proof will apply for a modified
$\ell' = \mbox{rank}(\mtx{A})$, and it is easy to see that adding additional columns to the basis matrices will make no difference since in this case $\mtx{U}(:,1:\ell)\mtx{U}(:,1:\ell)^{*}\mtx{A} = \mtx{U}_{\rm rsvd}\mtx{U}_{\rm rsvd}^{*}\mtx{A} = \mtx{A}$.
\end{remark}

\begin{remark}
The RSVD requires an estimate of the numerical rank in order to compute a \emph{partial} factorization. By contrast, the other methods discussed in this paper, which include the SVD, the column-pivoted QR (CPQR), and the two new methods (powerURV and randUTV), compute \emph{full} factorizations \emph{without} any a priori information about the numerical rank.
\end{remark}

\section{The \textsc{randUTV} algorithm.}
\label{sec:ch5-randutvmod}

{In this section, we describe the \textsc{randUTV} algorithm.
Throughout this section, $\mtx{A} \in \mathbb{R}^{m \times n}$ is the input matrix. Without loss of generality, we assume $m \ge n$ (if $m < n$, we may operate on $\mtx{A}^*$ instead). {A factorization}
\[
\begin{array}{ccccc}
\mtx{A} & = & \mtx{U} & \mtx{T} & \mtx{V}^*. \\
m \times n & & m \times m & m \times n & n \times n
\end{array}
\]
{where $\mtx{U}$ and $\mtx{V}$ are orthogonal and $\mtx{T}$ is upper trapezoidal is called a \textit{UTV factorization}. Note that the SVD and CPQR are technically two examples of UTV factorizations. In the literature, however, a decomposition is generally given the UTV designation only if that is its most specific label; we will continue this practice in the current article. Thus, it is implied that $\mtx{T}$ is upper trapezoidal but not diagonal, and $\mtx{V}$ is orthogonal but not a permutation.}
{The flexibility of the factors of a UTV decomposition allow it to act as a compromise between the SVD and CPQR in that it is not too expensive to compute but can reveal rank quite well. A UTV factorization can also be updated and downdated easily; see \cite[Ch. 5, Sec. 4]{stewart1998matrix} and \cite{fierro1999utv,park1995downdating,stewart1992updating} for details.}

\textsc{randUTV} is a blocked algorithm for computing a rank-revealing UTV decomposition of a matrix. It is tailored to run particularly efficiently on parallel architectures due to the fact that the bulk of its flops are cast in terms of matrix-matrix multiplication. The block size $b$ is a user-defined parameter which in practice is usually a number such as $128$ or $256$ {(multiples of the tile size used in libraries such as cuBLAS)}. 
In Section \ref{sec:ch5-randutvalg}, we review the original \textsc{randUTV} algorithm in \cite{martinsson2019randutv}. Then, Section \ref{sec:ch5-oversampling}  presents methods of modifying \textsc{randUTV} to enhance the rank-revealing properties of the resulting factorization. Finally, in Section \ref{sec:ch5-reducecost} we combine these methods with a bootstrapping technique to derive an efficient algorithm on a GPU.

\subsection{The \textsc{randUTV} algorithm for computing a UTV decomposition}
\label{sec:ch5-randutvalg}

The algorithm performs the bulk of its work in a loop that executes  $s= \lceil n/b \rceil$ iterations. We start with $\mtx{T}^{(0)} \defeq \mtx{A}$. In the $i$-th iteration ($i=1,2,\ldots,s$), a matrix $\mtx{T}^{\si} \in \mathbb{R}^{m \times n}$ is formed by the computation
\begin{equation}
\mtx{T}^{\si} \defeq \left(\mtx{U}^{\si}\right)^* \mtx{T}^{\simo} \mtx{V}^{\si},
\end{equation}
where $\mtx{U}^{\si} \in \mathbb{R}^{m \times m}$ and $\mtx{V}^{\si} \in \mathbb{R}^{n \times n}$ are orthogonal matrices chosen to {ensure that $T^{(i)}$ satisfies} the sparsity (nonzero pattern) and rank-revealing properties of the final factorization. Consider the first $s-1$ steps. Similar to other blocked factorization algorithms, the $i$-th step is meant to ``process'' a set of $b$ columns of $\mtx{T}^{\simo}$, so that after step $i$, $\mtx{T}^{\si}$ satisfies the following sparsity requirements:
\begin{itemize}
\item $\mtx{T}^{\si}(:,1:ib)$ is upper trapezoidal, and
\item the $b \times b$ blocks on the main diagonal of $\mtx{T}^{\si}(:,1:ib)$ are themselves diagonal.
\end{itemize}
After $s-1$ iterations, we compute the SVD of the bottom right block $\mtx{T}^{(s-1)}\left(((s-1)b+1):m, ((s-1)b+1):n\right)$ to obtain $\mtx{U}^{(s)}$ and $\mtx{V}^{(s)}$.}
The sparsity pattern followed by the $\mtx{T}^{\si}$ matrices is demonstrated in Figure \ref{fig:ch5-randutv-pattern}.
\begin{figure}
\begin{center}
\begin{tabular}{cccc}
\includegraphics[width=.145\textwidth]{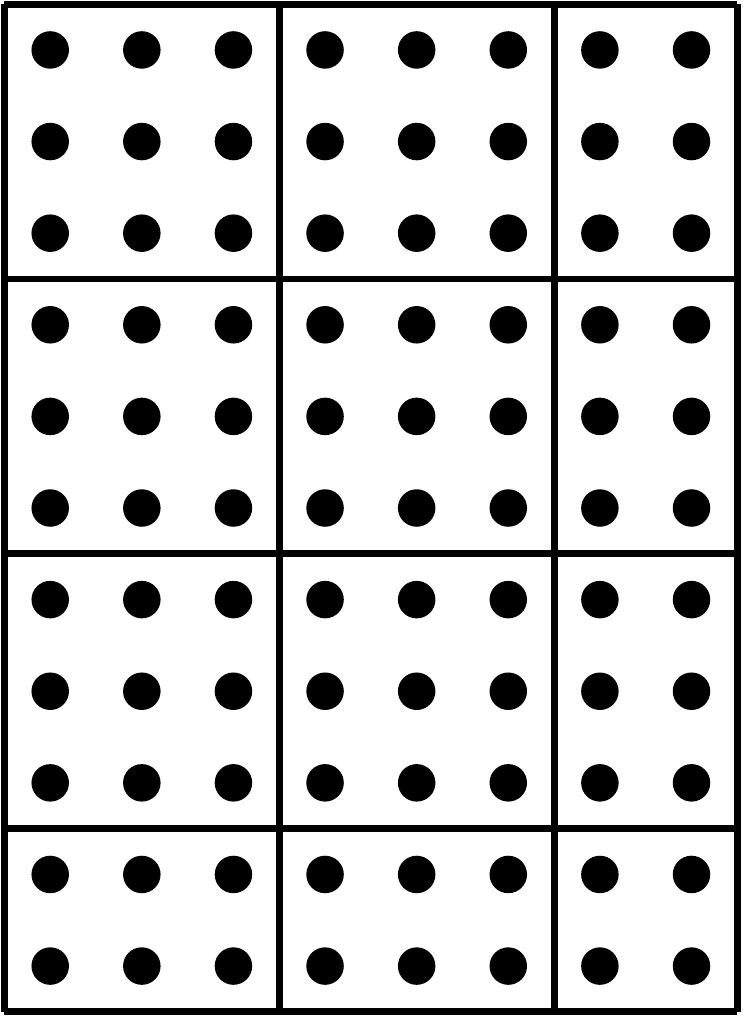}
&
\includegraphics[width=.145\textwidth]{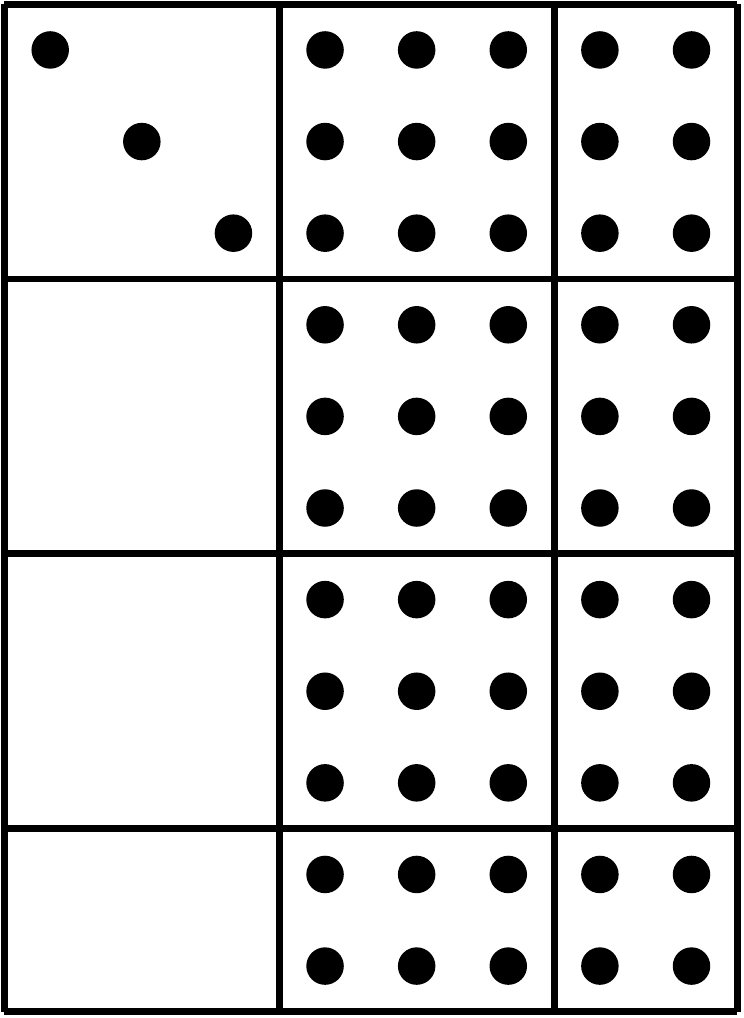}
&
\includegraphics[height=.2\textwidth]{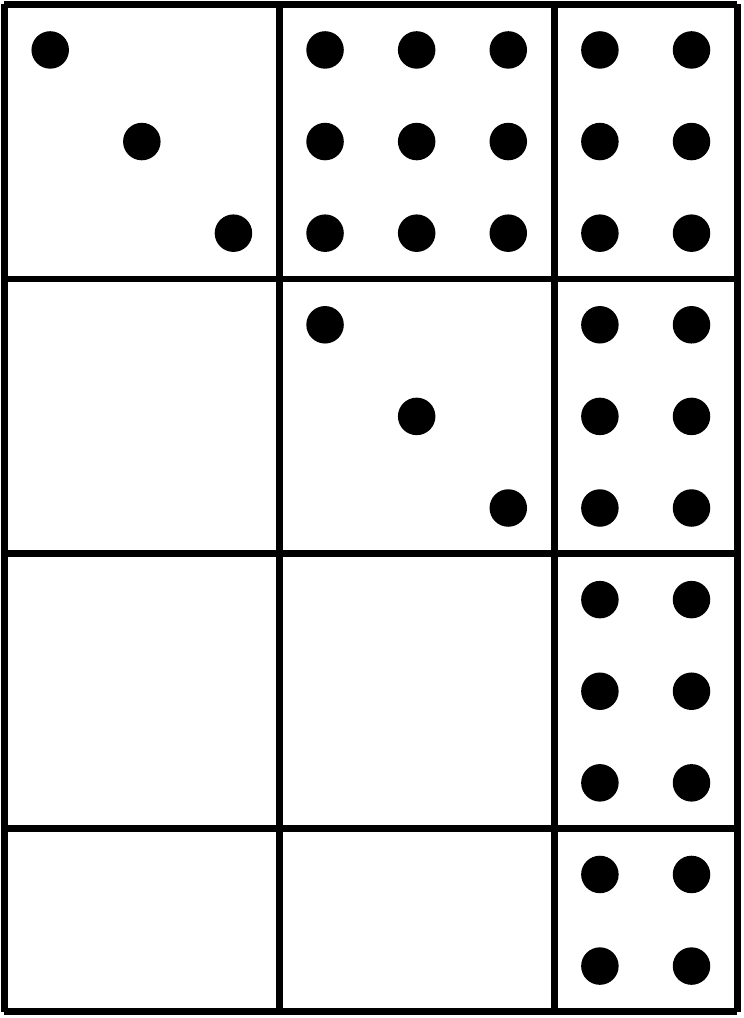}
&
\includegraphics[height=.2\textwidth]{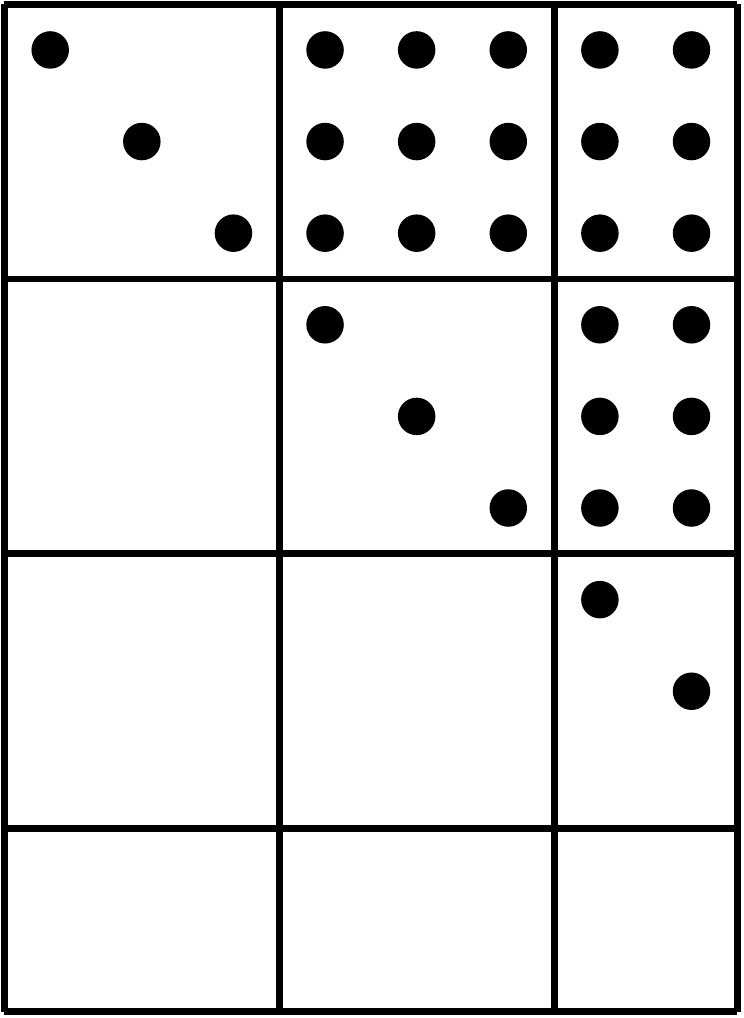}
\\
after 0 steps:  & after 1 step:  & after 2 steps: & after 3 steps: \\
$\mtx{T}^{(0)} \defeq \mtx{A}$ & $\mtx{T}^{(1)} \defeq (\mtx{U}^{(1)})^* \mtx{T}^{(0)} \mtx{V}^{(1)}$ & $ \mtx{T}^{(2)} \defeq (\mtx{U}^{(2)})^* \mtx{T}^{(1)} \mtx{V}^{(2)}$ & $ \mtx{T}^{(3)} \defeq (\mtx{U}^{(3)})^* \mtx{T}^{(2)} \mtx{V}^{(3)}$
\end{tabular}
\end{center}
\caption{
An illustration of the sparsity pattern of the four matrices $\mtx{T}^{\si}$ that appear in \textsc{randUTV}, for the particular case where $m=11, n=8$, and $b=3$.
}
\label{fig:ch5-randutv-pattern}
\end{figure}
When the $\mtx{U}$ and $\mtx{V}$ matrices are desired, they can be built with the computations
\begin{align*}
\mtx{U} & \defeq \mtx{U}^{(1)} \mtx{U}^{(2)} \cdots \mtx{U}^{(s)}, \\
\mtx{V} & \defeq \mtx{V}^{(1)} \mtx{V}^{(2)} \cdots \mtx{V}^{(s)}.
\end{align*}

In practice, the $\mtx{T}^{\si},$ $\mtx{U}^{\si}$ and $\mtx{V}^{\si}$ are not stored separately to save memory. Instead, the space for $\mtx{T}$, $\mtx{U}$, and $\mtx{V}$ is allocated at the beginning of randUTV, and at iteration $i$ each is updated with
\begin{align*}
\mtx{T} & \leftarrow \left(\mtx{U}^{\si}\right)^* \mtx{T} \mtx{V}^{\si}, \\
\mtx{V} & \leftarrow \mtx{V} \mtx{V}^{\si}, \\
\mtx{U} & \leftarrow \mtx{U} \mtx{U}^{\si},
\end{align*}
where $\mtx{U}^{\si}$ and $\mtx{V}^{\si}$ may be discarded or overwritten after an iteration completes.


To motivate the process of building $\mtx{U}^{\si}$ and $\mtx{V}^{\si}$, consider the first step of \textsc{randUTV}. We begin by initializing $\mtx{T} \defeq \mtx{A} \in \mathbb{R}^{m \times n}$ and creating a helpful partition
\[
\mtx{T} = \begin{bmatrix} \mtx{T}_{11} & \mtx{T}_{12} \\ \mtx{T}_{21} & \mtx{T}_{22} \end{bmatrix},
\]
where $\mtx{T}_{11}$ is $b \times b$, {$\mtx{T}_{12}$ is $b \times (n-b)$, $\mtx{T}_{21}$ is $(m-b) \times b$, and $\mtx{T}_{22}$ is $(m-b) \times (n-b)$}. The goal is to construct $\mtx{V}^{(1)} \in \mathbb{R}^{n \times n}$ and $\mtx{U}^{(1)} \in \mathbb{R}^{m \times m}$ such that, after the update $\mtx{T} \leftarrow (\mtx{U}^{(1)})^* \mtx{T} \mtx{V}^{(1)}$,
\begin{enumerate}
\item $\mtx{T}_{11}$ is {diagonal, (with entries that closely approximate the leading $b$ singular values of \mtx{A})}
\item $\mtx{T}_{21} = \mtx{0}$,
\item $\sigma_{\min}(\mtx{T}_{11}) \approx \sigma_b(\mtx{A})$,
\item $\sigma_{\max}(\mtx{T}_{22}) \approx \sigma_{b+1}(\mtx{A})$,
\item $\mtx{T}_{11}(k,k) \approx \sigma_k(\mtx{A}),$ $k=1,2,\ldots,b$.
\end{enumerate}
Items (1) and (2) in the list are basic requirements for any UTV factorization, and the rest of the items relate to the decomposition's rank-revealing properties.

{The key observation is that if $\mtx{V}^{(1)}$ and $\mtx{U}^{(1)}$ were orthogonal matrices whose leading $b$ columns spanned the same subspace as the leading $b$ right and left singular vectors, respectively, of $\mtx{A}$, items (2)-(4) in the previous list would be satisfied perfectly.} (Items (1) and (5) could then be satisfied with an inexpensive post-processing step: compute the SVD of $\mtx{T}_{11} \in \mathbb{R}^{b \times b}$ and update $\mtx{V}^{(1)}$ and $\mtx{U}^{(1)}$ accordingly.)
 However, determining the singular vector subspaces is as difficult as computing a partial SVD of $\mtx{A}$. We therefore content ourselves with the goal of building $\mtx{V}^{(1)}$ and $\mtx{U}^{(1)}$ such that the spans of the leading $b$ columns are \textit{approximations} of the subspaces spanned by the leading $b$ right and left singular vectors, respectively, of $\mtx{A}$.
%
%
%
We can achieve this goal efficiently using a variant of the randomized range finder algorithm discussed in Section \ref{sec:ch5-rsvd}. Specifically, we build $\mtx{V}^{(1)}$ as follows:
\begin{enumerate}[1.]
\item Generate a standard Gaussian matrix  $\mtx{G} \in \mathbb{R}^{m \times b}$.
\item Sample Row($\mtx{A}$) by calculating $\mtx{Y} = (\mtx{A}^* \mtx{A})^q \mtx{A}^* \mtx{G} \in \mathbb{R}^{n \times b}$, where $q$ is a small non-negative integer.
\item Compute the full (unpivoted) QR factorization of  $\mtx{Y}$ to obtain an orthogonal matrix $\mtx{V}^{(1)}$, \textit{i.e.}, $[\mtx{V}^{(1)}, \sim] = \textsc{HQR\_full}(\mtx{Y})$.
\footnote{{According to Theorem~\ref{th:rank} and Theorem~\ref{th:rank3} in Appendix~\ref{a:proof}, we know that, with probability 1, rank$(\mtx{Y})=\text{rank}(\mtx{A})$ and the first $\text{rank}(\mtx{A})$ columns in $\mtx{Y}$ are linearly independent.}}
\end{enumerate}

Once $\mtx{V}^{(1)}$ is built, the path to $\mtx{U}^{(1)}$ becomes clear after one observation. If the leading $b$ columns of $\mtx{V}^{(1)}$ span exactly the same subspace as the leading $b$ right singular vectors of $\mtx{A}$, then the leading $b$ columns of $\mtx{A} \mtx{V}^{(1)}$ span exactly the same subspace as the leading $b$ \textit{left} singular vectors of $\mtx{A}$. Therefore, after computing $\mtx{V}^{(1)}$, we build $\mtx{U}^{(1)}$ as follows:
\begin{enumerate}[1.]
\item Perform the matrix multiplication $\mtx{B} = \mtx{A} \mtx{V}^{(1)}(:,1:b) \in \mathbb{R}^{m \times b}$.
\item Compute the full (unpivoted) QR factorization of $\mtx{B}$ to obtain an orthogonal matrix $\mtx{U}^{(1)}$, \textit{i.e.}, $[\mtx{U}^{(1)}, \sim] = \textsc{HQR\_full}(\mtx{B})$.
\footnote{{It is easy to show that, with probability 1, rank$(\mtx{B})=\text{rank}(\mtx{A})$ and the first $\text{rank}(\mtx{A})$ columns in $\mtx{B}$ are linearly independent.}}
\end{enumerate}

{Following the same procedure, we can build $\hat{\mtx{V}}^{(i+1)} \in \mathbb{R}^{(n-ib) \times (n-ib)}$ and $\hat{\mtx{U}}^{(i+1)} \in \mathbb{R}^{(m-ib) \times (m-ib)}$ using the bottom right block $\mtx{T}^{(i)}\left((ib+1):m, (ib+1):n\right)$ for $i=1,2,\ldots,s-2$.
At the last step, the SVD of the remaining block $\mtx{T}^{(s-1)}\left(((s-1)b+1):m, ((s-1)b+1):n\right)$ gives
$\hat{\mtx{V}}^{(s)} \in \mathbb{R}^{(n-(s-1)b) \times (n-(s-1)b)}$ and $\hat{\mtx{U}}^{(s)} \in \mathbb{R}^{(m-(s-1)b) \times (m-(s-1)b)}$. Then, we have}
\[
\mtx{V}^{(i+1)} = \begin{bmatrix} \mtx{I}_{ib} &  \\  & \hat{\mtx{V}}^{(i+1)} \end{bmatrix}
\quad \text{and} \quad
\mtx{U}^{(i+1)} = \begin{bmatrix} \mtx{I}_{ib} &  \\  & \hat{\mtx{U}}^{(i+1)} \end{bmatrix},
\quad \text{for $i=0,1,2,\ldots,s-1$.}
\]
Notice that the transformation matrices $\mtx{U}^{\si}$ and $\mtx{V}^{\si}$ can be applied to the $\mtx{T}^{\si}$ matrices efficiently as discussed in Section~\ref{sec:qr}. We describe the basic \textsc{randUTV} algorithm, adapted from \cite{2015_martinsson_blocked}, in Appendix~\ref{a:utv}.




An important feature of \textsc{randUTV} is that {if a low-rank approximation of $\mtx{A}$ is needed, the factorization can be halted once a prescribed tolerance  has been met.} In particular, consider the following partition
\[
\mtx{A} = \mtx{U} \mtx{T} \mtx{V}^* =
\begin{bmatrix} \mtx{U}_1 & \mtx{U}_2 \end{bmatrix}
\begin{bmatrix} \mtx{T}_{11} & \mtx{T}_{12} \\ \mtx{0} & \mtx{T}_{22} \end{bmatrix}
\begin{bmatrix} \mtx{V}_1 & \mtx{V}_2 \end{bmatrix}^*,
\]
where $\mtx{U}_1$ and $\mtx{V}_1$ contain the first $k$ columns in the corresponding matrices and $\mtx{T}_{11}$ is $k \times k$. The rank-$k$ approximation from  \textsc{randUTV} is
\[
\mtx{A}_k = \mtx{U}_1 (\mtx{T}_{11} \mtx{V}_1^* + \mtx{T}_{12} \mtx{V}_2^*),
\]
and  the approximation error is
\begin{equation} \label{e:err}
\| \mtx{A} - \mtx{A}_k \|
= \| \mtx{U}_2 \mtx{T}_{22} \mtx{V}_2^* \|
= \| \mtx{T}_{22} \|,
\end{equation}
where $\mtx{U}_2$ and $\mtx{V}_2$ are orthonormal matrices. Therefore, the factorization can be terminated as long as $\| \mtx{T}_{22} \|$ becomes smaller than a prescribed tolerance. In our blocked algorithm, we can calculate and check
\[
e_i = \| \mtx{T}^{(i)}((ib+1):m, (ib+1):n) \|
\]
at the $i$-th iteration for $i=0,1,\ldots,s-1$.

{For applications where errors are measured  using the Frobenius norm, a more efficient algorithm is the following.} Notice that (\ref{e:err}) holds for the Frobenius norm as well. In fact, we have
\[
\| \mtx{A} - \mtx{A}_k \|_F^2
= \| \mtx{T}_{22} \|_F^2
= \|  \mtx{T} \|_F^2 - \| \mtx{T}_{11}  \|_F^2 - \| \mtx{T}_{12}  \|_F^2
= \|  \mtx{A} \|_F^2 - \| \mtx{T}_{11}  \|_F^2 - \| \mtx{T}_{12}  \|_F^2,
\]
where $\|\cdot\|_F$ denotes the Frobenius norm. So we need to only pre-compute $\|  \mtx{A} \|_F^2$ and update it with $\| \mtx{T}_{11}  \|_F^2 + \| \mtx{T}_{12}  \|_F^2$ involving two small matrices at every iteration in \textsc{randUTV}. Specifically, we have
\[
e_0 = \|  \mtx{A} \|_F
\]
and
\[
e_i^2 = e_{i-1}^2 -
\| \mtx{T}^{(i)}(((i-1)b+1):ib, ((i-1)b+1):n) \|_F^2
\]
for $i=1,2,\ldots,s-1$. A similar approach for the RSVD was proposed in~\cite{yu2018efficient}.

\begin{remark}
The \textsc{randUTV} can be said to be a ``blocked incremental RSVD'' in the sense that the first step of the method is \textit{identical} to the RSVD. In~\cite[Section 5.3]{martinsson2019randutv}, the authors demonstrate that the low-rank approximation error that results from the single-step \textsc{randUTV}  factorization is identical to the error produced by the corresponding RSVD.
\end{remark}



\subsection{Using oversampling in the \textsc{randUTV} algorithm}
\label{sec:ch5-oversampling}

In \textsc{randUTV}, the computation of matrix $\mtx{V}^{\si} \in \mathbb{R}^{n \times n}$ relies on the randomized range finder algorithm (discussed in Section \ref{sec:ch5-rsvd}). Just as the randomized range finder algorithm can use an oversampling parameter $p$ to improve the success probability, we add a similar parameter $p$ to the construction of $\mtx{V}^{\si}$ in Algorithm \ref{alg:randutv-basic} to improve the rank-revealing properties of the resulting factorization.

To do so, we first change the size of the random matrix $\mtx{G}$ from $m \times b$ to $m \times (b + p)$ (we once again consider only the building of $\mtx{V}^{(1)}$ to simplify matrix dimensions). This effectively increases the number of times we sample Row($\mtx{A}$) by $p$, providing a ``correction'' to the information in the first $b$ samples.

Next, we must modify how we obtain the orthonormal vectors that form $\mtx{V}^{(1)}$. Recall that \textit{the first $b$ columns} of $\mtx{V}^{(1)}$ must contain the approximation to the leading $b$ right singular vectors of $\mtx{A}$. If we merely orthonormalized the columns of $\mtx{Y} = (\mtx{A}^* \mtx{A})^q \mtx{A}^* \mtx{G}  \in \mathbb{R}^{n \times (b+p)}$ with \textsc{HQR} again, the first $b$ columns of $\mtx{V}^{(1)}$ would only contain information from the first $b$ columns of $\mtx{Y}$. We must therefore choose a method for building $\mtx{V}^{(1)}$ such that its first $b$ columns contain a good approximation of Col($\mtx{Y}$). The following approaches use two of the most common rank-revealing factorizations to solve this subproblem:
\begin{enumerate}[a)]
\item Perform a \textsc{CPQR} on $\mtx{Y}$ to obtain the orthogonal matrix $\mtx{V}^{(1)}\in \mathbb{R}^{n \times n}$.
The additional computational expense of computing \textsc{HQRCP} is relatively inexpensive for thin matrices like $\mtx{Y}$. However, the approximation provided by $\mtx{V}^{(1)}(:,1:b)$ to Col($\mtx{Y}$) in this case will be suboptimal as mentioned in Section \ref{sec:ch5-cpqr}.

\item Perform an SVD on $\mtx{Y}$ to obtain an orthogonal matrix $\mtx{W} \in \mathbb{R}^{n \times n}$.
Then $\mtx{W}(:,1:b)$ contains the optimal rank-$b$ approximation of Col($\mtx{Y}$). However, this method requires one more step since $\mtx{V}^{(1)}$ must be represented as a product of Householder vectors in order to compute $\mtx{A} \mtx{V}^{(1)}$ efficiently in the following step of \textsc{randUTV}. After computing the SVD, therefore, we must perform a full (unpivoted) QR decomposition on $\mtx{W}(:,1:b)$, \textit{i.e.}, $[\mtx{V}^{(1)}, \sim] = \textsc{HQR\_full}(\mtx{W}(:,1:b))$.
\footnote{{We would have detected rank deficiency $\mbox{rank}($\mtx{A}$) < b$ with the SVD, and the first $\mbox{rank}($\mtx{A}$)$ columns in $\mtx{W}$ must be linearly independent.}}
This method yields the optimal approximation of Col($\mtx{Y}$) given by $\mtx{V}^{(1)}(:,1:b)$.
\end{enumerate}


In this article, we use method b) because it provides better approximations.
As discussed in Remark \ref{remark:tall-thin-svd} below, method b) requires a full (unpivoted) \textsc{HQR} of size $n \times (b+p)$, an SVD of size $(b+p) \times (b+p)$, and a full (unpivoted) \textsc{HQR} of size $n \times b$. While the SVD is small and therefore quite cheap, the first \textsc{HQR} is an extra expense which is nontrivial when aggregated across every iteration in \textsc{randUTV}.
This extra cost is addressed and mitigated in Section \ref{sec:ch5-reducecost}.

\begin{remark}
The SVD of method b) above may look expensive at first glance. However, recall that $\mtx{Y}$ is of size $n \times (b + p)$, where $n \gg b + p$. For tall, thin matrices like this, the SVD may be inexpensively computed as follows~\cite{chan1982improved}:
\begin{enumerate}[1.]
\item Compute a full (unpivoted) QR decomposition of $\mtx{Y}$ to obtain an orthogonal matrix $\mtx{Q}\in \mathbb{R}^{n \times n}$ and an upper trapezoidal matrix $\mtx{R}\in \mathbb{R}^{n \times (b+p)}$, \textit{i.e.}, $[\mtx{Q}, \mtx{R}] = \textsc{HQR\_full}(\mtx{Y})$.\footnote{{According to Theorem~\ref{th:rank} and Theorem~\ref{th:rank3} in Appendix~\ref{a:proof}, we know that, with probability 1, rank$(\mtx{Y})=\text{rank}(\mtx{A})$ and the first $\text{rank}(\mtx{A})$ columns in $\mtx{Y}$ are linearly independent.}}
\item Compute the SVD of the square matrix $\mtx{R}(1:(b+p),:)$ to obtain an orthogonal matrix $\hat{\mtx{U}} \in \mathbb{R}^{(b + p) \times (b + p)}$ with left singular vectors of $\mtx{R}(1:(b+p),:)$, \textit{i.e.}, $[\hat{\mtx{U}}, \sim, \sim] = \text{SVD}(\mtx{R}(1:(b+p),:))$.
\end{enumerate}
After these steps, we recognize that the matrix of left singular vectors of matrix $\mtx{Y}$ is
\[
\mtx{W} = \mtx{Q}
\begin{bmatrix} \hat{\mtx{U}} & \\ & \mtx{I}_{n-(b+p)} \end{bmatrix}
  \in \mathbb{R}^{n \times n}.
\]
{The costs of the first step dominate the entire procedure, which are $O(n (b+p))$ storage and  $O(n (b+p)^2)$ work, respectively, according to Remark~\ref{rmk:qr}.}

\label{remark:tall-thin-svd}
\end{remark}


For randomized subspace iteration techniques like the one used to build $\mtx{V}^{\si},$ the stability of the iteration is often a concern. As in the \textsc{powerURV} algorithm, the information from any singular values less than {$\epsilon_{\text{machine}}^{1/(2q+1)} \sigma_{\max}(\mtx{A})$} will be lost unless an orthonormalization procedure is used during intermediate steps of the computation $\mtx{Y} = \left( \mtx{A}^* \mtx{A} \right)^q \mtx{A}^* \mtx{G}$. For \textsc{randUTV}, only $b$ singular vectors are in view in each iteration, so orthonormalization is not required as often as it is for \textsc{powerURV}.

However, if oversampling is used ($p>0$), performing one orthonormalization before the final application of $\mtx{A}^*$ markedly benefits the approximation of Col($\mtx{Y}$) to the desired singular vector subspace. Numerical experiments show that this improvement occurs even when the sampling matrix is not in danger of loss of information from roundoff errors. Instead, we may intuitively consider that using orthonormal columns to sample $\mtx{A}^*$ ensures that the ``extra'' $p$ columns of $\mtx{Y}$ contain information not already accounted for in the first $b$ columns.

\subsection{Reducing the cost of oversampling and orthonormalization}
\label{sec:ch5-reducecost}

Adding oversampling to \textsc{randUTV}, as discussed in Sections \ref{sec:ch5-oversampling}, adds noticeable cost to the algorithm. First, it requires a costlier method of building $\mtx{V}^{\si}$. Second, it increases the dimension of the random matrix $\mtx{G}$, increasing the cost of all the operations involving $\mtx{G}$ and, therefore, $\mtx{Y}$. 
{We will in this section demonstrate that the overhead caused by the latter cost can be essentially
eliminated by recycling the ``extra'' information we collected in one step when we execute the next step.}

To demonstrate, consider the state of \textsc{randUTV} for input matrix $\mtx{A} \in \mathbb{R}^{m \times n}$ with $q=2, p = b$ after one step of the main iteration. Let $\mtx{A}_q = (\mtx{A}^* \mtx{A})^q \mtx{A}^*$ denote the matrix for power iteration. At this point, we have computed and stored the following relevant matrices:
\begin{itemize}
\item $\mtx{Y}^{(1)} \in \mathbb{R}^{n \times (b+p)}$: The matrix containing random samples of Row($\mtx{A}$), computed with \[\mtx{Y}^{(1)} = \mtx{A}_q \mtx{G}^{(1)},\] where $\mtx{G}^{(1)} \in \mathbb{R}^{m \times (b+p)}$ is a standard Gaussian matrix.

\item $\mtx{W}^{(1)} \in \mathbb{R}^{n \times n}$: The matrix whose columns are the left singular vectors of $\mtx{Y}^{(1)}$, computed with \[[\mtx{W}, \sim, \sim] = \text{SVD}(\mtx{Y});\] see Section \ref{sec:ch5-oversampling} method b).

\item $\mtx{V}^{(1)} \in \mathbb{R}^{n \times n}$:  The right transformation in the main step of \textsc{randUTV}, computed with \[[\mtx{V}^{(1)}, \sim] = \textsc{HQR\_full}(\mtx{W}^{(1)}(:,1:b));\] see Section \ref{sec:ch5-oversampling} method b).

\item $\mtx{U}^{(1)} \in \mathbb{R}^{m \times m}$: The left transformation in the main step of \textsc{randUTV}, computed with \[[\mtx{U}^{(1)},\sim] = \textsc{HQR\_full}(\mtx{A} \mtx{V}^{(1)}(:,1:b));\] see Section \ref{sec:ch5-randutvalg}.

\item $\mtx{T}^{(1)} \in \mathbb{R}^{m \times n}$: The matrix being driven to upper trapezoidal form. At this stage in the algorithm, \[\mtx{T}^{(1)} = \left(\mtx{U}^{(1)}\right)^* \mtx{A} \mtx{V}^{(1)}.\]
\end{itemize}

Finally, consider the partitions
\begin{align*}
\mtx{Y}^{(1)} & = \begin{bmatrix} \mtx{Y}^{(1)}_1 & \mtx{Y}^{(1)}_2 \end{bmatrix}, \\
\mtx{W}^{(1)} & = \begin{bmatrix} \mtx{W}^{(1)}_1 & \mtx{W}^{(1)}_2 \end{bmatrix}, \\
\mtx{V}^{(1)} & = \begin{bmatrix} \mtx{V}^{(1)}_1 & \mtx{V}^{(1)}_2 \end{bmatrix}, \\
\mtx{U}^{(1)} & = \begin{bmatrix} \mtx{U}^{(1)}_1 & \mtx{U}^{(1)}_2 \end{bmatrix}, \\
\mtx{T}^{(1)} & = \begin{bmatrix} \mtx{T}_{11} & \mtx{T}_{12} \\ \mtx{T}_{21} & \mtx{T}_{22} \end{bmatrix}, \\
\end{align*}
where $\mtx{Y}^{(1)}_1, \mtx{W}^{(1)}_1,\mtx{V}^{(1)}_1$ and $\mtx{U}^{(1)}_1$ contain the first $b$ columns in corresponding matrices, and $\mtx{T}_{11}$ is $b \times b$.

In the second iteration of \textsc{randUTV}, the first step is to approximate Row($\mtx{T}_{22}$), where
\[
\mtx{T}_{22}^* = \left(\mtx{V}^{(1)}_2\right)^* \mtx{A}^* \mtx{U}^{(1)}_2.
\]
 Next, we make the observation that, just as Col($\mtx{W}^{(1)}_1$) approximates the subspace spanned by the leading $b$ right singular vectors of $\mtx{A}$, the span of the first $p$ columns of $\mtx{W}^{(1)}_2$ approximates the subspace spanned by the leading $(b+1)$-th through $(b+p)$-th right singular vectors of $\mtx{A}$. Thus, Col$\left(\left(\mtx{V}^{(1)}_2\right)^* \mtx{W}^{(1)}_2(:,1:p)\right)$ is an approximation of Col$\left(\left(\mtx{V}^{(1)}_2\right)^*\mtx{A}^*\right)$. This multiplication involving two small matrix dimensions is much cheaper than carrying out the full power iteration process.


Putting it all together, on every iteration of \textsc{randUTV} after the first, we build the sampling matrix $\mtx{Y}^{\si}$ of $b+p$ columns in two stages. First, we build $\mtx{Y}^{\si}(:,1:b)$ \emph{without} oversampling or orthonormalization as in Section \ref{sec:ch5-randutvalg}. Second, we build
\[
\mtx{Y}^{\si}(:,(b+1):(b+p)) = \left(\mtx{V}^{\simo}(:,(b+1):(n-b(i-1)))\right)^* \mtx{W}^{\simo}(:,(b+1):(b+p))
\]
 by reusing the samples from the last iteration. The complete algorithm, adjusted to improve upon the low-rank estimation accuracies of the original \textsc{randUTV}, is given in {Algorithm} \ref{alg:randutv-boosted}.

\begin{algorithm}
\begin{algorithmic}[1]
 \REQUIRE {matrix $\mtx{A} \in \mathbb{R}^{m \times n}$, positive integers $b$ and $p$, and a non-negative integer $q$.}
 \ENSURE {$\mtx{A} = \mtx{U} \mtx{T} \mtx{V}^*$, where $\mtx{U} \in \mathbb{R}^{m \times m}$  and $\mtx{V} \in \mathbb{R}^{n \times n}$ are orthogonal, and $\mtx{T} \in \mathbb{R}^{m \times n}$ is upper trapezoidal.}
\STATE $\mtx{T} = \mtx{A}; \mtx{U} = \mtx{I}_m; \mtx{V} = \mtx{I}_n;$
\FOR{$i=1$:min$(\lceil m/b, \rceil, \lceil n/b \rceil)$}
  \STATE $I_1 = 1:(b(i-1)); I_2 = (b(i-1)+1):\min(bi,m); I_3 = (bi+1):m;$
  \STATE $J_1 = 1:(b(i-1)); J_2 = (b(i-1)+1):\min(bi,n); J_3 = (bi+1):n;$
  \IF{($I_3$ and $J_3$ are both nonempty)}
    \IF{$(i == 1)$}
      \STATE {$\mtx{G} = \textsc{randn}(m,b+p)$}
      \STATE $\mtx{Y} = \mtx{T}^* \mtx{G}$
      \FOR {$j=1:q$}
      	\STATE $\mtx{Y} = \mtx{T}^* (\mtx{T} \mtx{Y})$
      \ENDFOR
    \ELSE
      \STATE {$\mtx{G} = \textsc{randn}(m - b(i-1),b)$}
    \STATE $\mtx{Y}(:,1:b) = \mtx{T}([I_2,I_3],[J_2,J_3])^* \mtx{G}$
    \FOR{$j=1:(q-1)$}
	\STATE $\mtx{Y}(:,1:b) = \mtx{T}([I_2,I_3],[J_2,J_3])^*(\mtx{T}([I_2,I_3],[J_2,J_3])\mtx{Y}(:,1:b))$
    \ENDFOR
	\STATE $\mtx{X} = \mtx{T}([I_2,I_3],[J_2,J_3])\mtx{Y}(:,1:b)$
	\STATE $\mtx{X} = \mtx{X} - \mtx{W}_{\text{next}}\mtx{W}_{\text{next}}^* \mtx{X}$
        \STATE $[\mtx{Q},\sim] = \textsc{HQR\_full}(\mtx{X})$
         \STATE $\mtx{Y} =  \mtx{T}([I_2,I_3],[J_2,J_3])^*[\mtx{Q}(:,1:b), \mtx{W}_{\text{next}}]$
    \ENDIF

    \STATE $[\mtx{W}_Y, \sim, \sim] = \textsc{SVD}(\mtx{Y})$
    \STATE $[\mtx{V}^{\si} , \sim] = \textsc{HQR\_full}(\mtx{W}_Y(:,1:b))$
    \STATE $\mtx{T}(:,[J_2,J_3]) = \mtx{T}(:,[J_2,J_3]) \mtx{V}^{\si}$
    \STATE $\mtx{V}(:,[J_2,J_3]) = \mtx{V}(:,[J_2,J_3]) \mtx{V}^{\si} $
    \STATE $\mtx{W}_{\text{next}} = \left(\mtx{V}^{\si}(:,(b+1):(n-b(i-1))\right)^* \mtx{W}_Y(:,(b+1):(b+p)) $
    \STATE
    \STATE $[\mtx{U}^{\si},\mtx{R}] = \textsc{HQR\_full}(\mtx{T}([I_2,I_3],J_2))$
    \STATE $\mtx{U}(:,[I_2,I_3]) = \mtx{U}(:,[I_2,I_3]) \mtx{U}^{\si}$
    \STATE $\mtx{T}([I_2,I_3],J_3) = \left(\mtx{U}^{\si}\right)^* \mtx{T}([I_2,I_3,J_3])$
    \STATE $\mtx{T}(I_3,J_2) = \textsc{zeros}(m-bi,b)$
    \STATE
    \STATE $[\mtx{U}_{\text{small}},\mtx{D}_{\text{small}},\mtx{V}_{\text{small}}] = \textsc{svd}(\mtx{R}(1:b,1:b))$
    \STATE $\mtx{U}(:,I_2) = \mtx{U}(:,I_2) \mtx{U}_{\text{small}}$
    \STATE $\mtx{V}(:,J_2) = \mtx{V}(:,J_2) \mtx{V}_{\text{small}}$
    \STATE $\mtx{T}(I_2,J_2) = \mtx{D}_{\text{small}}$
    \STATE $\mtx{T}(I_2,J_3) = \mtx{U}_{\text{small}}^* \mtx{T}(I_2,J_3)$
    \STATE $\mtx{T}(I_1,J_2) = \mtx{T}(I_1,J_2) \mtx{V}_{\text{small}}$
  \ELSE
    \STATE $[\mtx{U}_{\text{small}},\mtx{D}_{\text{small}},\mtx{V}_{\text{small}}] = \textsc{svd}(\mtx{T}([I_2,I_3],[J_2,J_3]))$
    \STATE $\mtx{U}(:,[I_2,I_3]) = \mtx{U}(:,[I_2,I_3]) \mtx{U}_{\text{small}}$
    \STATE $\mtx{V}(:,[J_2,J_3]) = \mtx{V}(:,[J_2,J_3]) \mtx{V}_{\text{small}}$
    \STATE $\mtx{T}([I_2,I_3],[J_2,J_3]) = \mtx{D}_{\text{small}}$
    \STATE $\mtx{T}([I_1,[J_2,J_3]) = \mtx{T}(I_1,[J_2,J_3]) \mtx{V}_{\text{small}}$
  \ENDIF
\ENDFOR
\end{algorithmic}
\caption{[\mtx{U},\mtx{T},\mtx{V}] = \textsc{randUTV\_boosted}(\mtx{A},b,q,p)}
\label{alg:randutv-boosted}
\end{algorithm}

\section{Implementation details} \label{sec:gpu}

As mentioned earlier, \textsc{powerURV} (Algorithm~\ref{alg:powerURV-stable}) and \textsc{randUTV} (Algorithm~\ref{alg:randutv-boosted}) mainly consist of Level-3 BLAS routines such as matrix-matrix multiplication (\textsc{gemm}), which can execute extremely efficiently on modern  computer architectures such as GPUs. Two essential features of GPUs from the algorithmic design perspective are the following: (1) The amount of parallelism available is massive. For example, an NVIDIA V100 GPU has 5120 CUDA cores. (2) The costs of data-movement against computation are orders of magnitude higher. As a result, Level-1 or Level-2 BLAS routines do not attain a significant portion of GPUs' peak performance.

Our GPU implementation of the \textsc{powerURV} algorithm and \textsc{randUTV} algorithm uses a mix of routines from the cuBLAS library\footnote{https://docs.nvidia.com/cuda/cublas/index.html} from NVIDIA and the MAGMA library~\cite{tdb10,tnld10,dghklty14}. {The MAGMA library is a collection of next generation linear algebra routines with GPU acceleration. 
%
For the \textsc{powerURV} algorithm, we use the \textsc{gemm} routine from cuBLAS and routines related to (unpivoted) QR decomposition from MAGMA. For the \textsc{randUTV} algorithm, we mostly use MAGMA routines except for the \textsc{gemm} routine from cuBLAS to apply the Householder reflectors (see Section~\ref{sec:qr}). These choices are mainly guided by empirical performance.

{Our implementation consists of a sequence of GPU level-3 BLAS calls. The matrix is copied to the GPU at the start, and all computations (with one exception, see below) are done on the GPU with no communication back to main memory until the computation completes.

The exception is that the SVD subroutines in MAGMA do not support any GPU interface that takes an input matrix  on the GPU. This is a known issue of the MAGMA library, and we follow the common workaround: copy the input matrix from GPU to CPU and then call a MAGMA SVD subroutine (MAGMA  copies the matrix back to GPU and computes its SVD).

Fortunately, the extra cost of data transfer is negligible because the matrices whose SVDs are needed in Algorithm 2 are all very small (of dimensions at most $(b+p) \times (b+p)$).}

\newpage

\section{Numerical results}
\label{sec:ch5-num}

In this section, we present numerical experiments to demonstrate the performance of \textsc{powerURV} (Algorithm~\ref{alg:powerURV-stable}) and \textsc{randUTV} (Algorithm~\ref{alg:randutv-boosted}). In particular, we compare them to the SVD and the \textsc{HQRCP} in terms of speed and accuracy. Since the SVD is the most accurate method, we use it as the baseline to show the speedup and the accuracy of other methods. Results of the SVD were obtained using the MAGMA routine \textsc{MAGMA\_DGESDD},\footnote{{The DGESDD algorithm uses a divide-and-conquer algorithm, which is different from the DGESVD algorithm based on QR iterations. The former is also known to be faster for large matrices;} see \url{https://www.netlib.org/lapack/lug/node71.html} and Table 1 with CPU timings of both DGESVD and DGESDD in~\cite{martinsson2019randutv}} where all orthogonal matrices are calculated. Results of the \textsc{HQRCP} were obtained using the MAGMA routine \textsc{MAGMA\_DGEQP3}, where the orthogonal matrix $\mtx{Q}$ was calculated. 


%
%
%
%

All experiments were performed on an NVIDIA Tesla V100 graphics card with 32 GB of memory, which is connected to  two Intel Xeon Gold 6254 18-core CPUs at 3.10GHz. Our code makes extensive use of the MAGMA (Version 2.5.4) library linked with the Intel MKL library (Version 20.0). It was compiled with the NVIDIA compiler NVCC (Version 11.3.58) on the Linux OS (5.4.0-72-generic.x86\_64).

\subsection{Computational speed}
\label{sec:ch5-speed}

In this section, we investigate the speed of {\textsc{powerURV} and} \textsc{randUTV} on the GPU {and compare them to highly optimized implementations of the SVD and the \textsc{HQRCP} for the GPU.} In Figures \ref{fig:ch5-powurv-times}, \ref{fig:ch5-randutv-basic}, and \ref{fig:ch5-randutv-boosted}, every factorization is computed on a standard Gaussian matrix $\mtx{A} \in \mathbb{R}^{n \times n}$. {All methods discussed here are ``direct'' methods, whose running time does not depend on the input matrix but its  size.} We compare the times (in seconds) of different algorithms/codes, {where the input and output matrices exist on the \emph{CPU} (time for moving data between CPU and GPU is included).}

In each plot, we divide the given time by $n^3$ to make the asymptotic relationships between each experiment more clear, where {$n=3\,000, 4\,000, 5\,000, 6\,000, 8\,000, 10\,000, 12\,000, 15\,000, 20\,000, 30\,000$. The  MAGMA routine \textsc{magma\_dgesdd} for computing the SVD run out of memory when $n=30\,000$.} {Raw timing results are given in Table~\ref{t:raw} in Appendix~\ref{a:data}.}


\begin{figure}[h]
\centering
\includegraphics[width=0.48\textwidth]{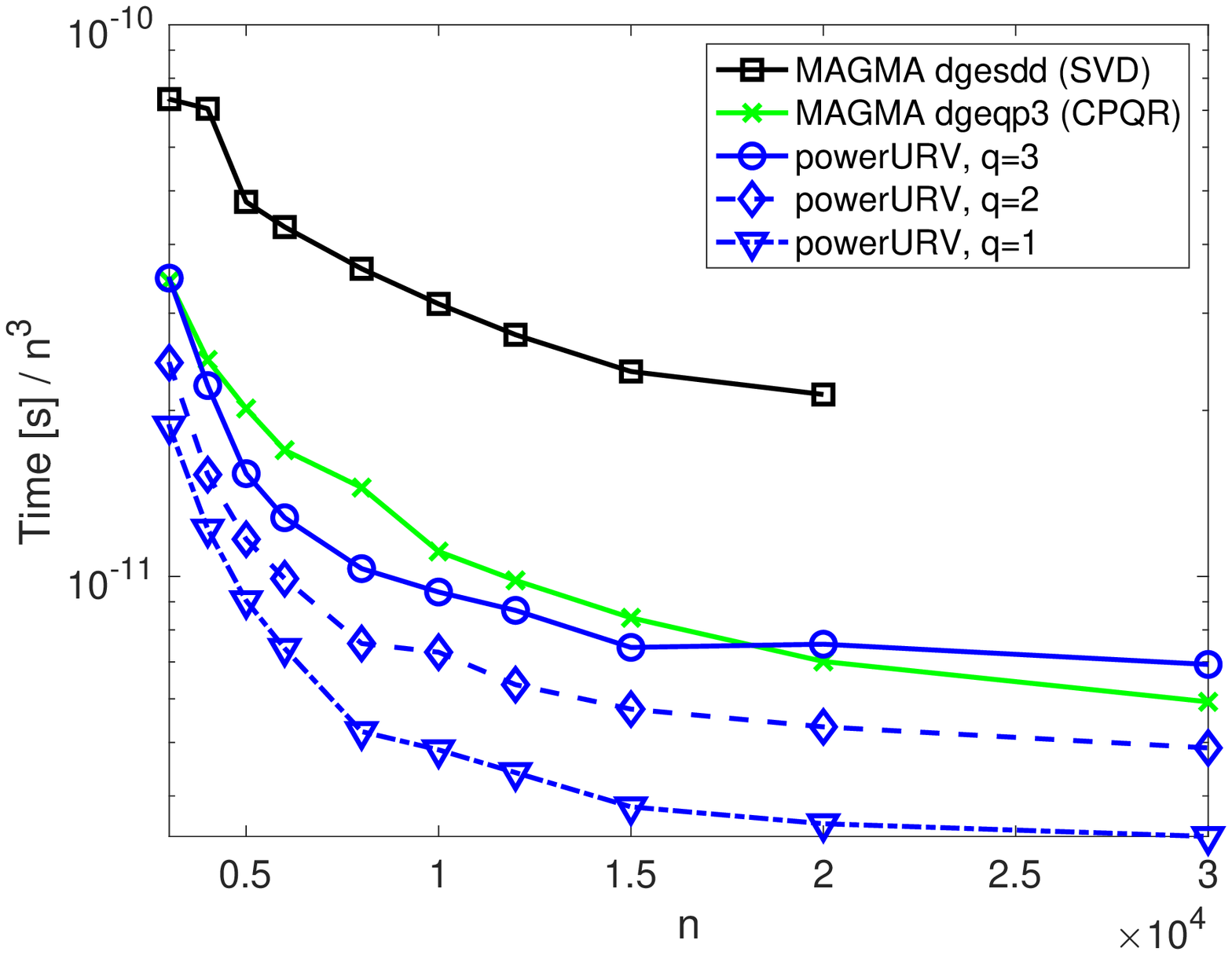}
\includegraphics[width=0.48\textwidth]{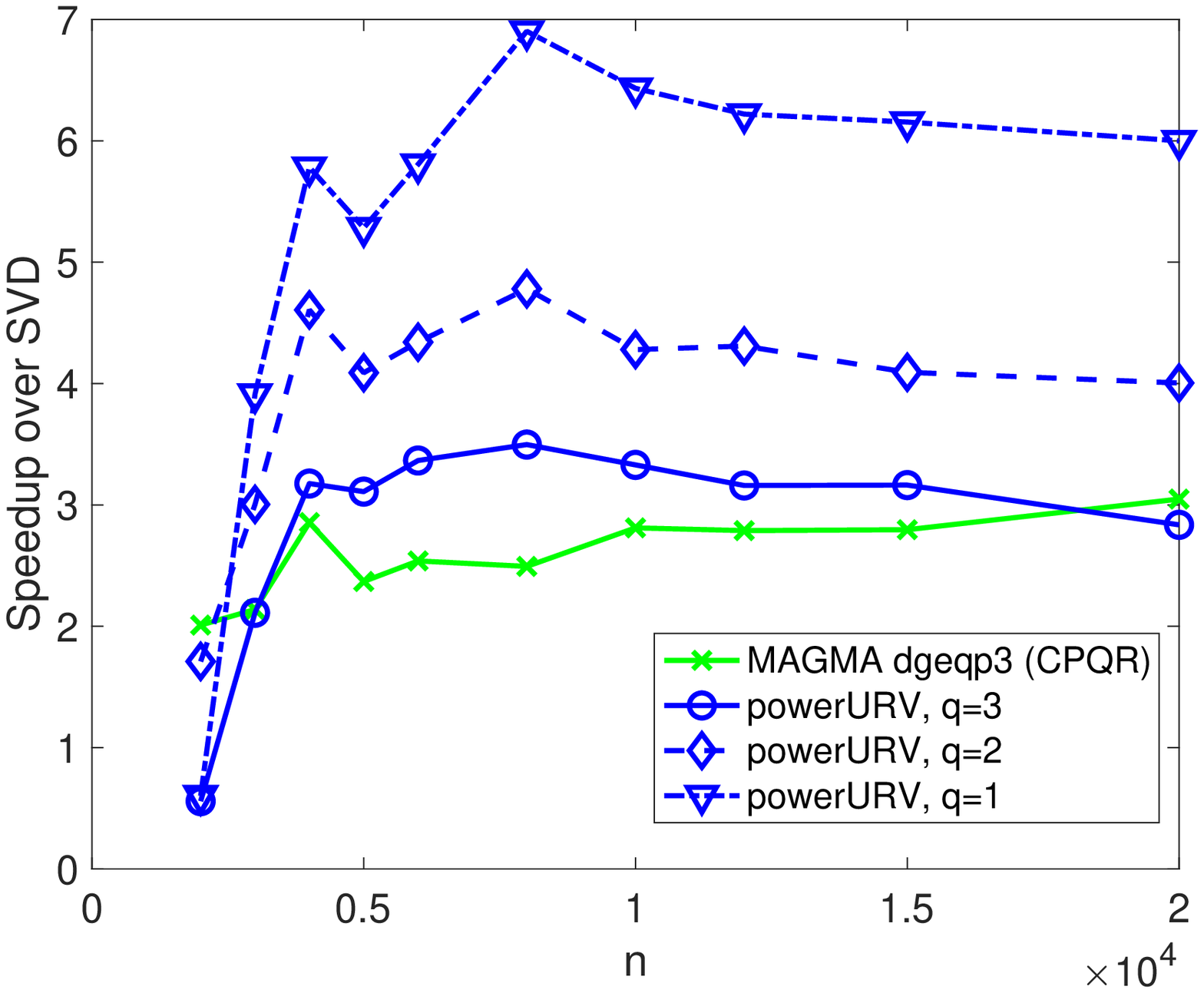}
\caption{{(Left) computation times for the SVD,  \textsc{HQRCP} and \textsc{powerURV}  on the GPU. (Right) speedups of the \textsc{HQRCP} and \textsc{powerURV} over the SVD.}}
\label{fig:ch5-powurv-times}
\end{figure}



\begin{figure}[h]
\includegraphics[width=0.48\textwidth]{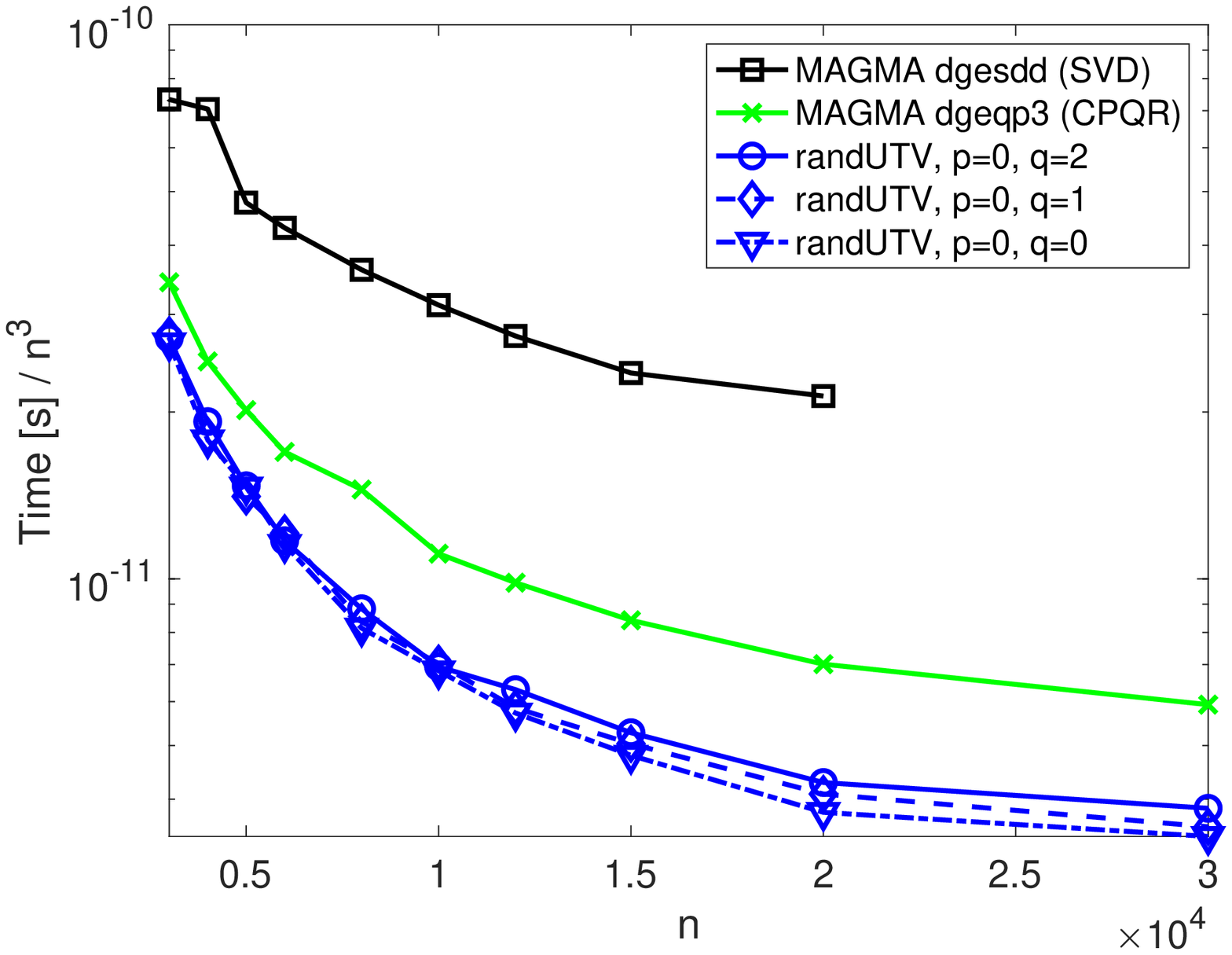}
\includegraphics[width=0.48\textwidth]{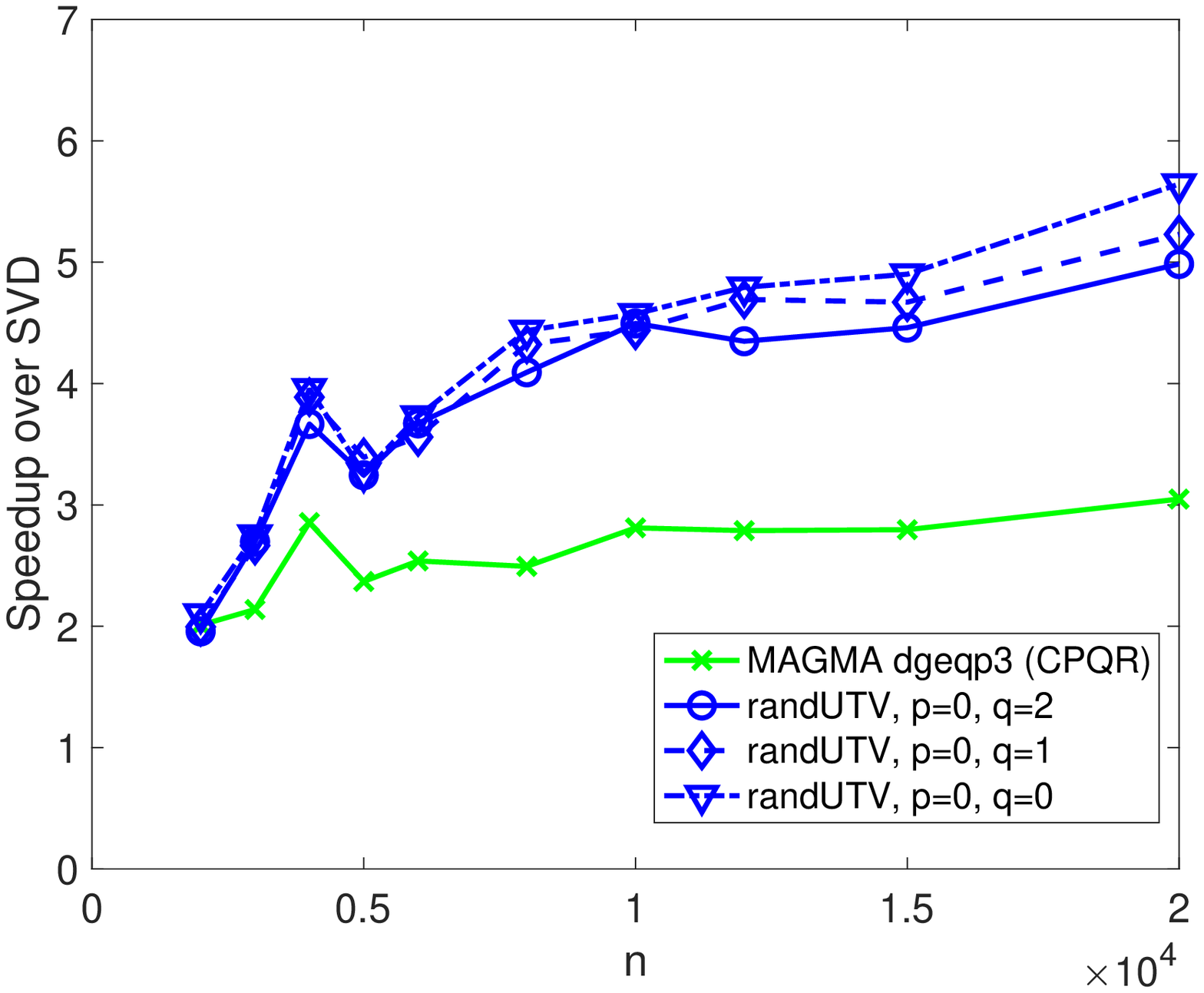}
\caption{{(Left) computation times for the \textsc{randUTV} without oversampling ($b=128$) plotted against the computation time for \textsc{HQRCP} and the SVD on the GPU. (Right) Speedups of the \textsc{randUTV} without oversampling and the \textsc{HQRCP} over the SVD.}}
\label{fig:ch5-randutv-basic}
\end{figure}

\begin{figure}[h]
\includegraphics[width=0.48\textwidth]{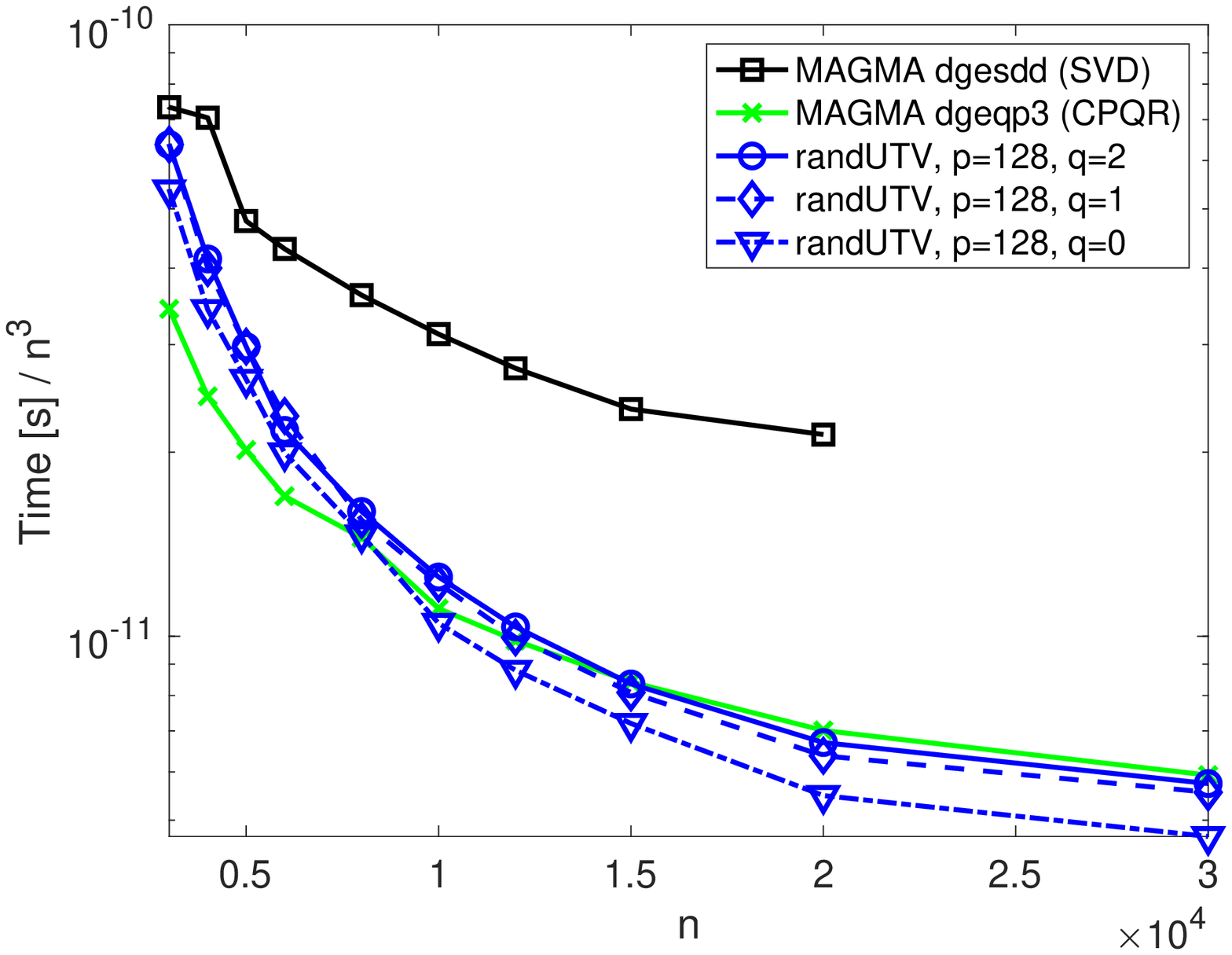}
\includegraphics[width=0.48\textwidth]{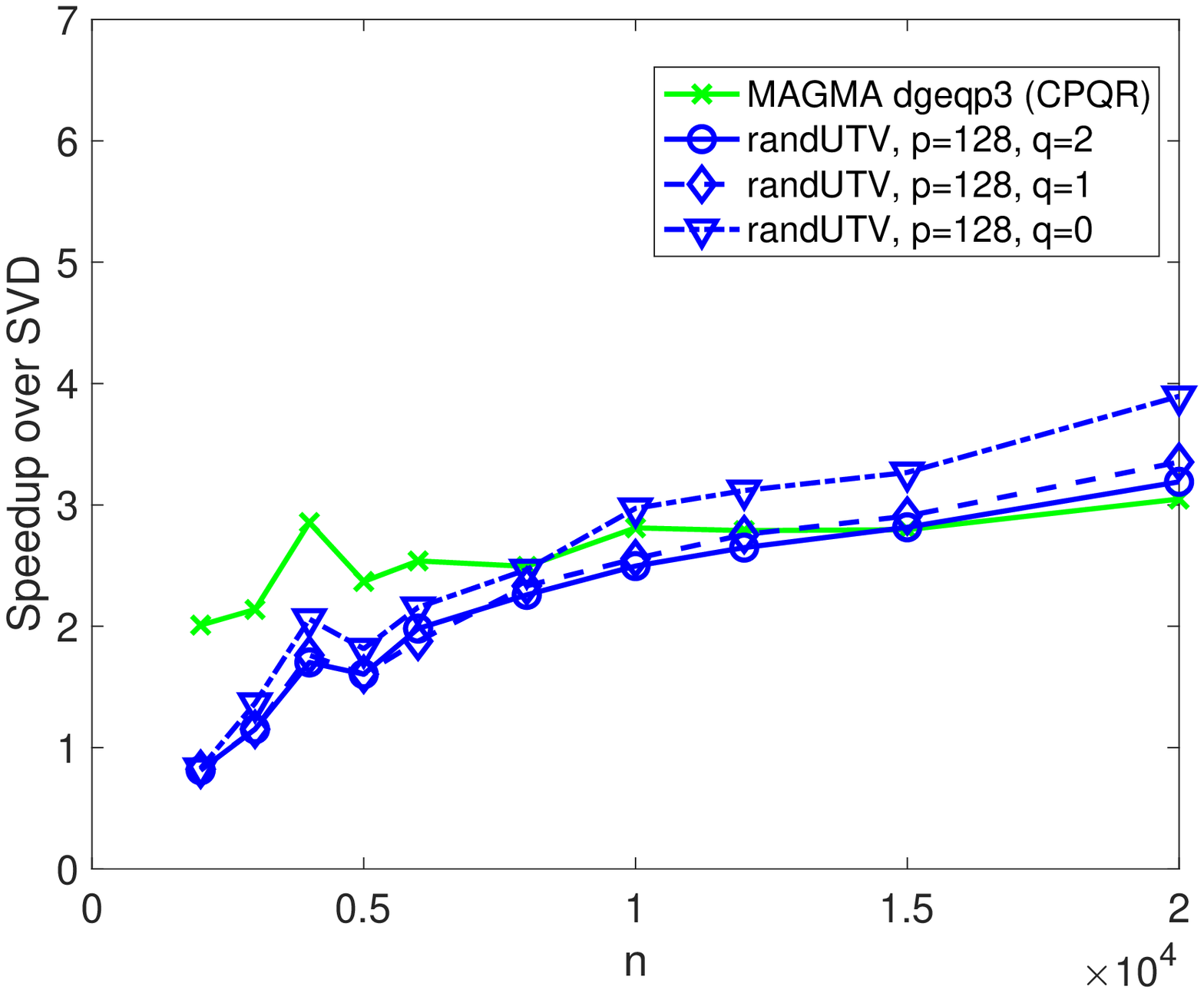}
\caption{{(Left) computation times for the \textsc{randUTV} with oversampling ($b=128$) plotted against the computation time for \textsc{HQRCP} and the SVD on the GPU. (Right) Speedups of the \textsc{randUTV} with oversampling and the \textsc{HQRCP} over the SVD.}}
\label{fig:ch5-randutv-boosted}
\end{figure}

{We observe in Figure \ref{fig:ch5-powurv-times} that \textsc{powerURV} with $q=1,2$ power iterations consistently outperforms \textsc{HQRCP}. \textsc{powerURV} with $q=3$ power iterations delivers similar performance with  \textsc{HQRCP} for large matrices, but arrives at the asymptotic region much faster. As expected, the SVD is much slower than the other two methods.}

We observe in Figure \ref{fig:ch5-randutv-basic} that \textsc{randutv} without {oversampling} handily outperforms \textsc{HQRCP}. The cost of increasing the parameter $q$ is also quite small due to the high efficiency of matrix multiplication on the GPU.

We observe in Figure \ref{fig:ch5-randutv-boosted} that \textsc{randutv} with {oversampling} still outperforms \textsc{HQRCP} when $n\ge 15\,000$. In addition, observe that the distance between the lines for $q=2$ and $q=1$ is less than the distance between the lines for $q=0$ and $q=1$. This difference is representative of the savings gained with bootstrapping technique whereby extra samples from one iteration are carried over to the next iteration.


{To summarize, our results show that the two newly proposed algorithms (\textsc{powerURV} and \textsc{randutv}) both achieved clear speedups over the SVD. They are also faster than \textsc{HQRCP} for sufficiently large matrices.}

\subsection{Approximation error}
\label{sec:ch5-errors}

{In this section, we compare the errors in the low-rank approximations} produced by SVD,  \textsc{HQRCP}, \textsc{powerURV}, and \textsc{randUTV}. Given an matrix $\mtx{A} \in \mathbb{R}^{n \times n}$, each rank-revealing factorization  produces a decomposition
\[
\mtx{A} = \mtx{U} \mtx{T} \mtx{V}^*,
\]
where $\mtx{U} \in \mathbb{R}^{n \times n}$ and $\mtx{V} \in \mathbb{R}^{n \times n}$ are orthogonal, and $\mtx{T} \in \mathbb{R}^{n \times n}$ is upper triangular. Given this factorization, a natural rank-$k$ approximation to $\mtx{A}$ is
\begin{equation}
\mtx{A}_k = \mtx{U}(:,1:k) \mtx{T}(1:k,:) \mtx{V}^*.
\end{equation}
Recall in Section \ref{sec:ch5-svd} that the rank-$k$ approximation produced by the SVD is the optimal among all rank-$k$ approximations, so we denote it as $\mtx{A}_k^{\text{optimal}}$.
For each of the factorizations that we study, we evaluated the error
\begin{equation}
\label{eq:errork}
e_{k} = \| \mtx{A} - \mtx{A}_k \|
\end{equation}
{as a function of $k$, and report the results in Figures \ref{fig:error} and \ref{fig:error_n4000}. Three different test matrices are considered:}


\begin{itemize}
\item \textbf{Fast decay:} This matrix is generated by first creating random orthogonal matrices $\mtx{U}$ and $\mtx{V}$ by performing unpivoted QR factorizations on two random matrices with i.i.d~entries drawn according to the standard normal distribution. Then, $\mtx{A}_{\text{fast}}$ is computed with $\mtx{A}_{\text{fast}} = \mtx{U} \mtx{D} \mtx{V}^*$, where $\mtx{D}$ is a diagonal matrix with diagonal entries $d_{ii} = \beta^{(i-1)/(n-1)}$, where $\beta = 10^{-5}$.
\item \textbf{S-shaped decay:} This matrix is generated in the same manner as ``fast decay,'' but the diagonal entries are chosen to first hover around $1$, then quickly decay before leveling out at $10^{-2}$, as shown in Figure \ref{fig:sshape}.
\item \textbf{BIE:} This matrix is the result of discretizing a \emph{boundary integral equation} (BIE) defined on a smooth closed curve in the plane. To be precise, we discretized the so called ``single layer'' operator associated with the Laplace equation using a sixth order quadrature rule designed by Alpert \cite{1999_alpert_hybrid}. This operator is well-known to be ill-conditioned, which necessitates the use of a rank-revealing factorization in order to solve the corresponding linear system in as stable a manner as possible.
\end{itemize}

\begin{figure}
\centering
\begin{subfigure}{0.48\textwidth}
\centering
\includegraphics[width=\textwidth]{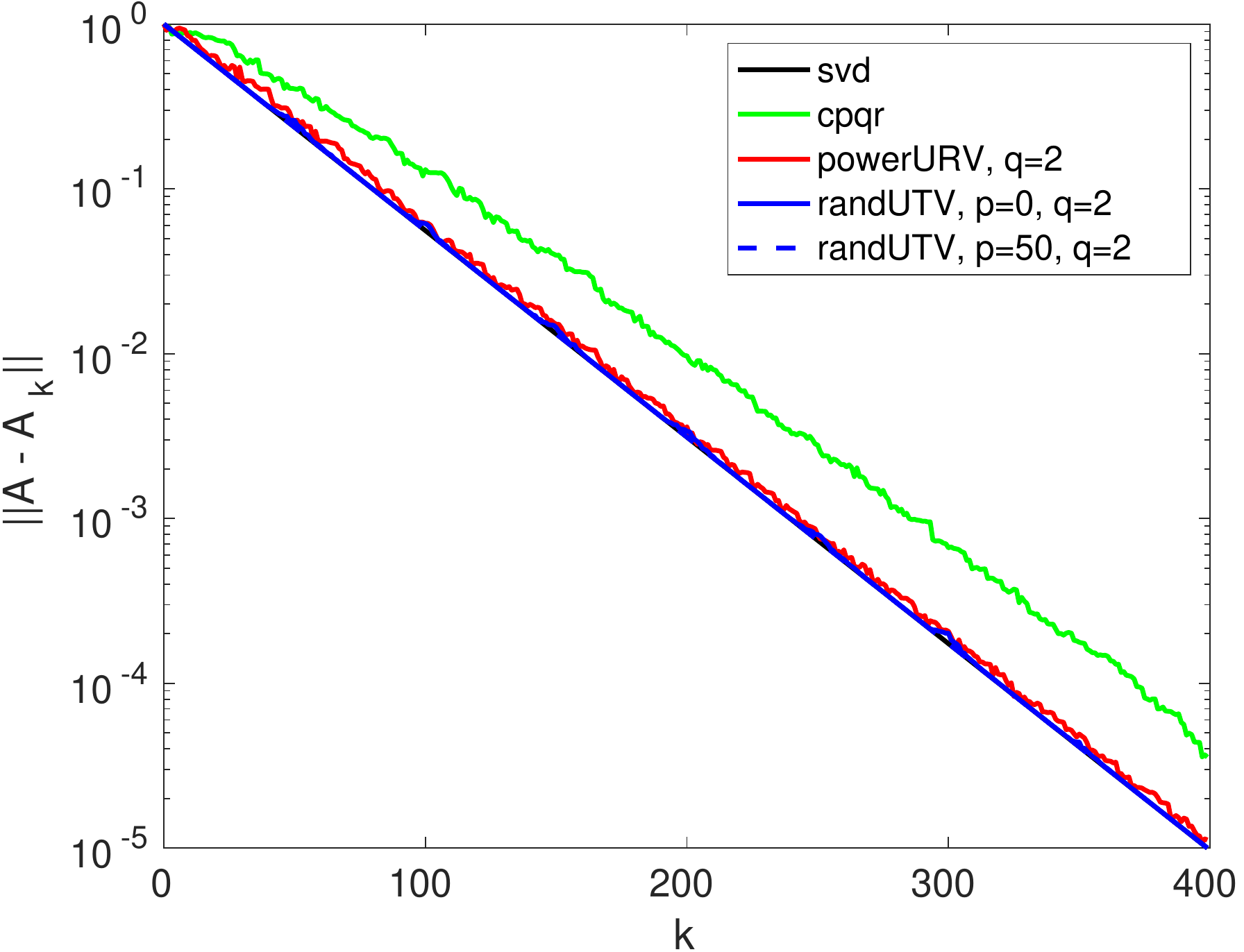}
\caption{``fast decay'', operator norm} \label{fig:fast}
\end{subfigure}
\begin{subfigure}{0.48\textwidth}
\centering
\includegraphics[width=\textwidth]{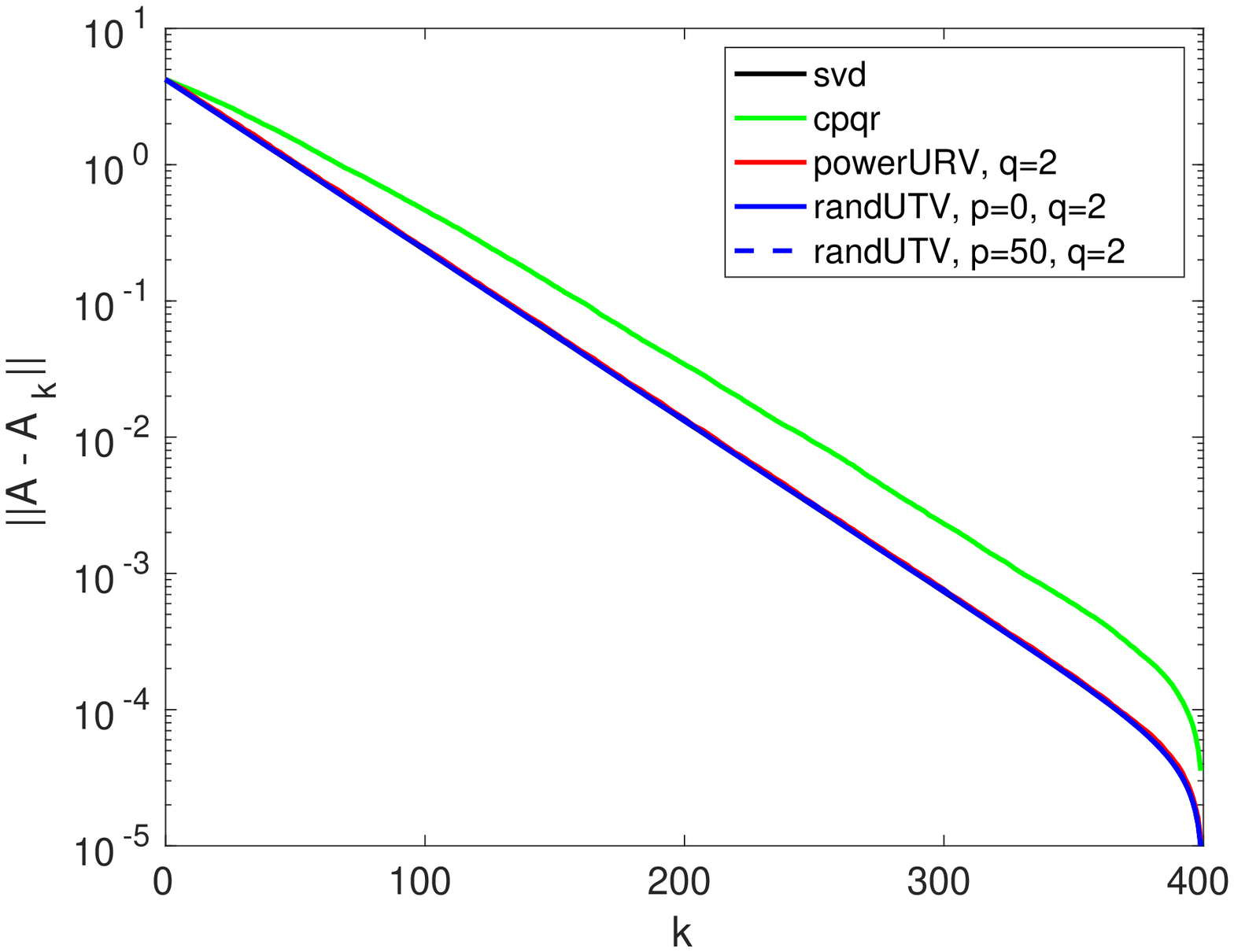}
\caption{``fast decay'', Frobenius norm}
\end{subfigure}
\begin{subfigure}{0.48\textwidth}
\centering
\includegraphics[width=\textwidth]{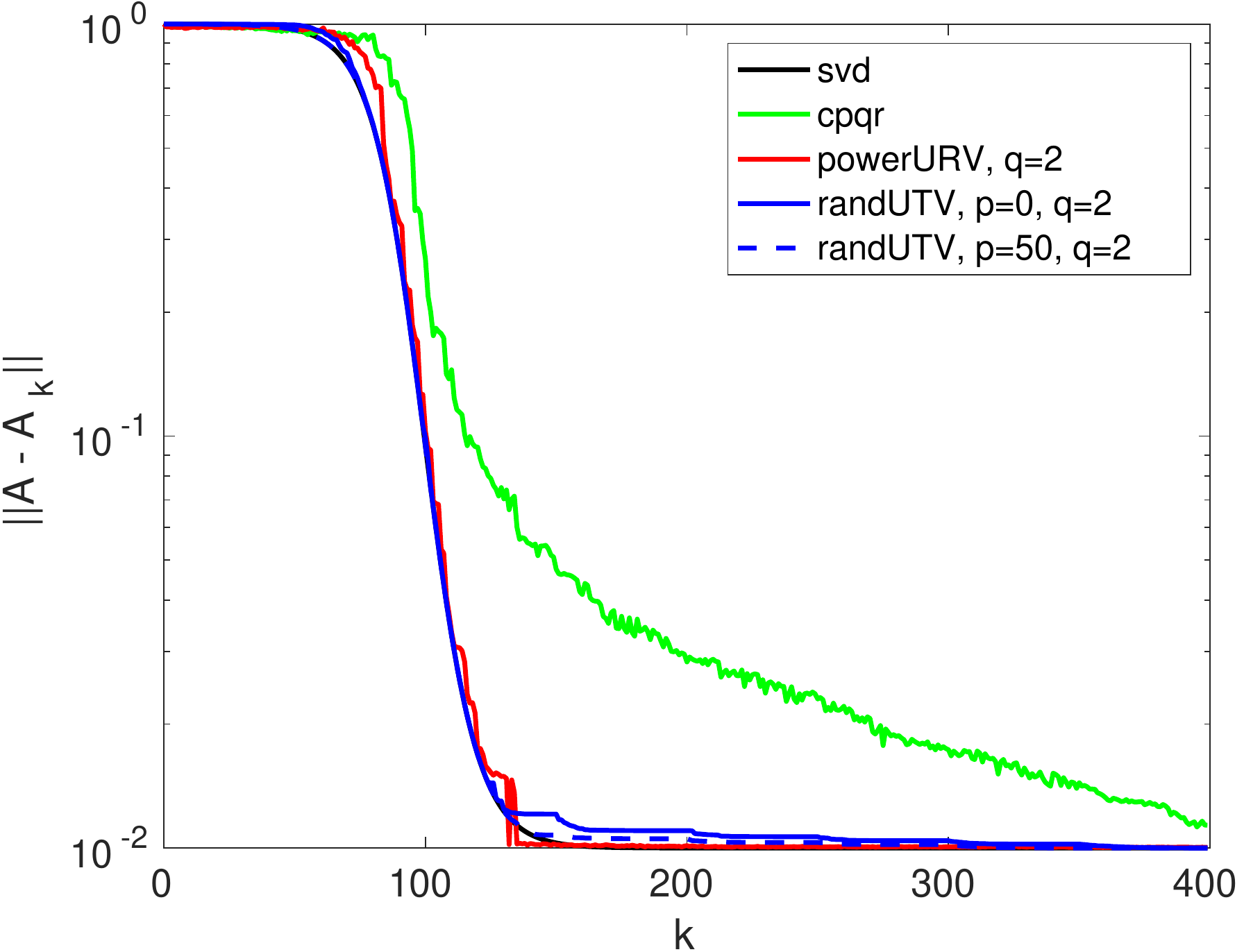}
\caption{``S-shaped decay'', operator norm}
\label{fig:sshape}
\end{subfigure}
\begin{subfigure}{0.48\textwidth}
\centering
\includegraphics[width=\textwidth]{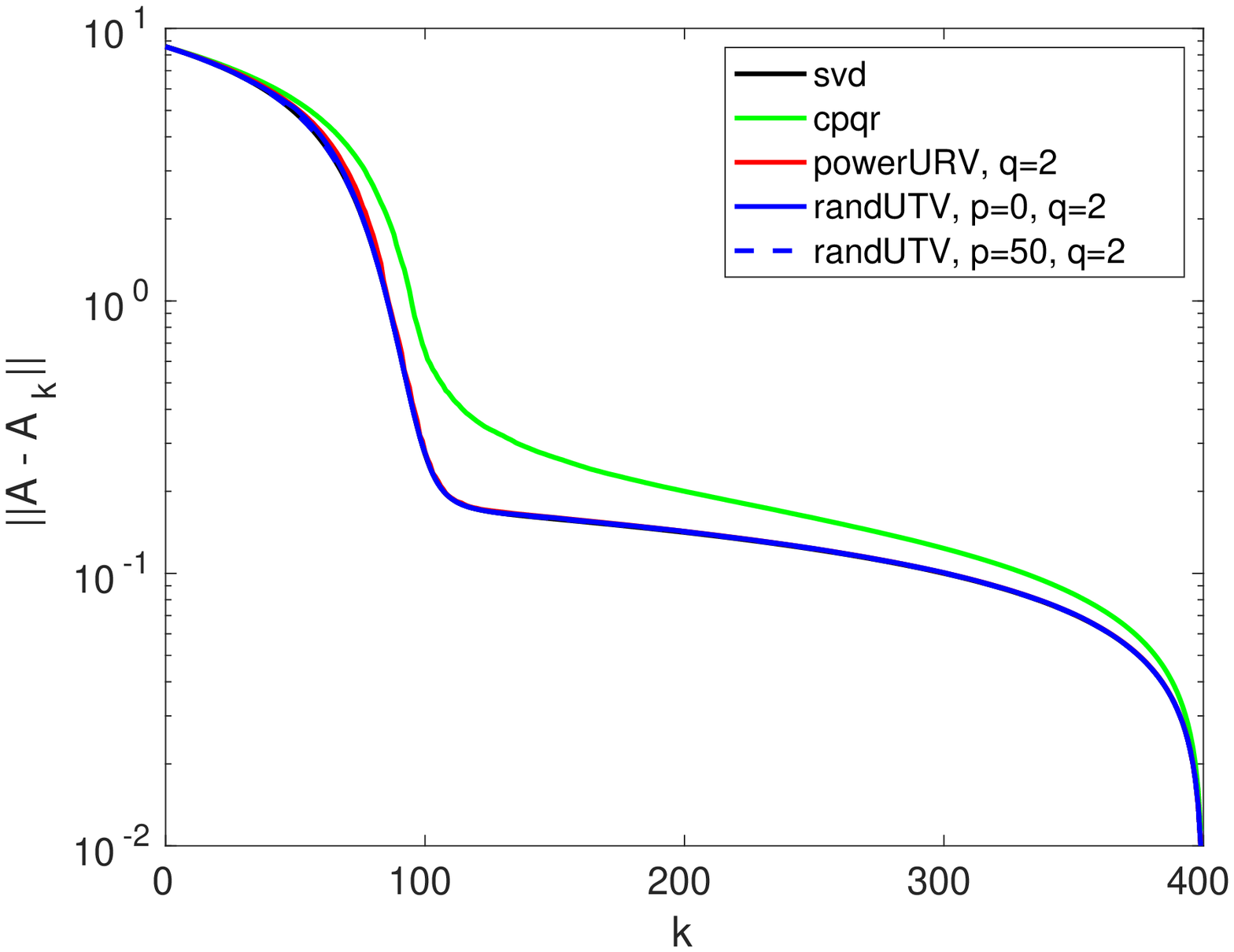}
\caption{``S-shaped decay'', Frobenius norm}
\end{subfigure}
\begin{subfigure}{0.48\textwidth}
\centering
\includegraphics[width=\textwidth]{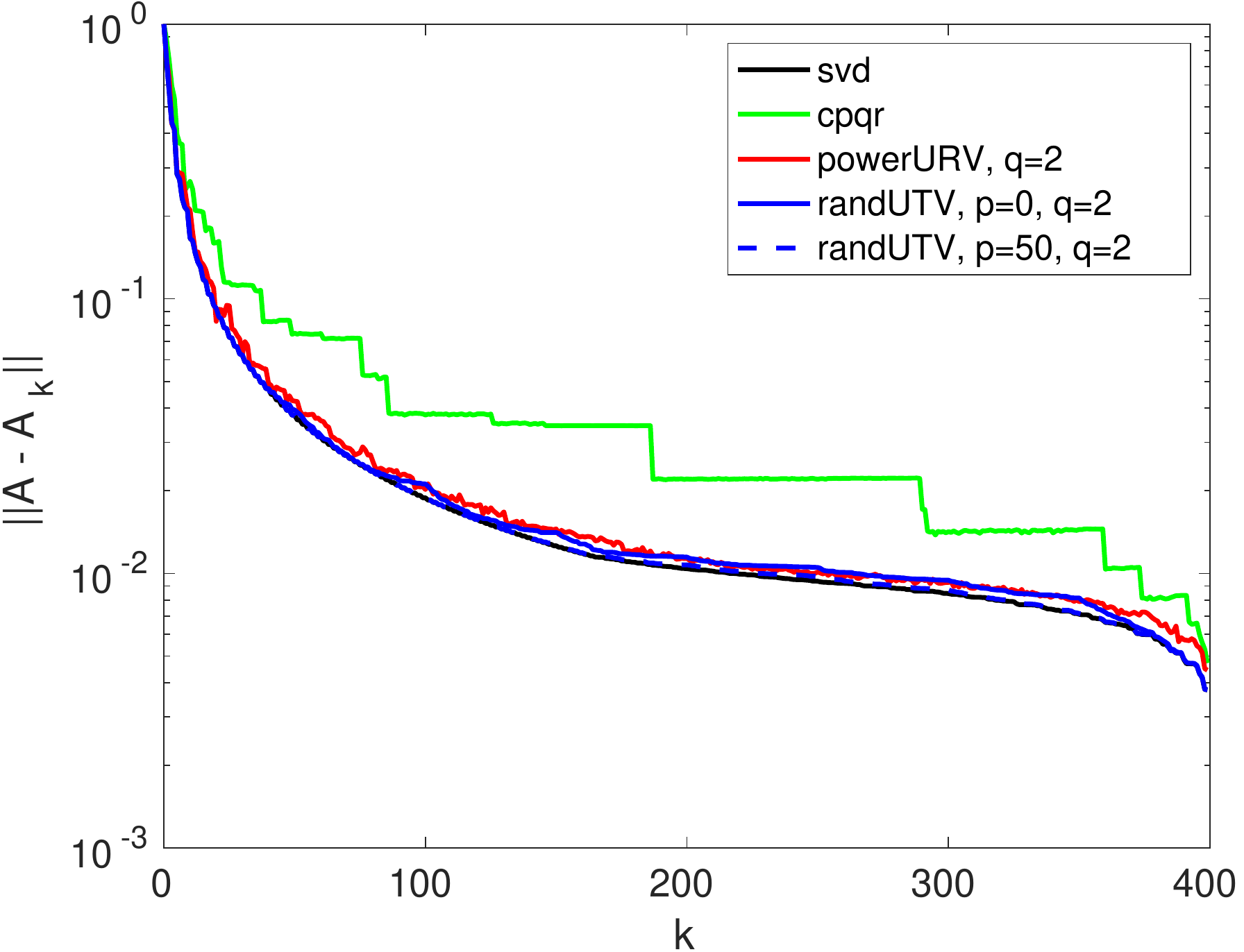}
\caption{``BIE'', operator norm}
\end{subfigure}
\begin{subfigure}{0.48\textwidth}
\centering
\includegraphics[width=\textwidth]{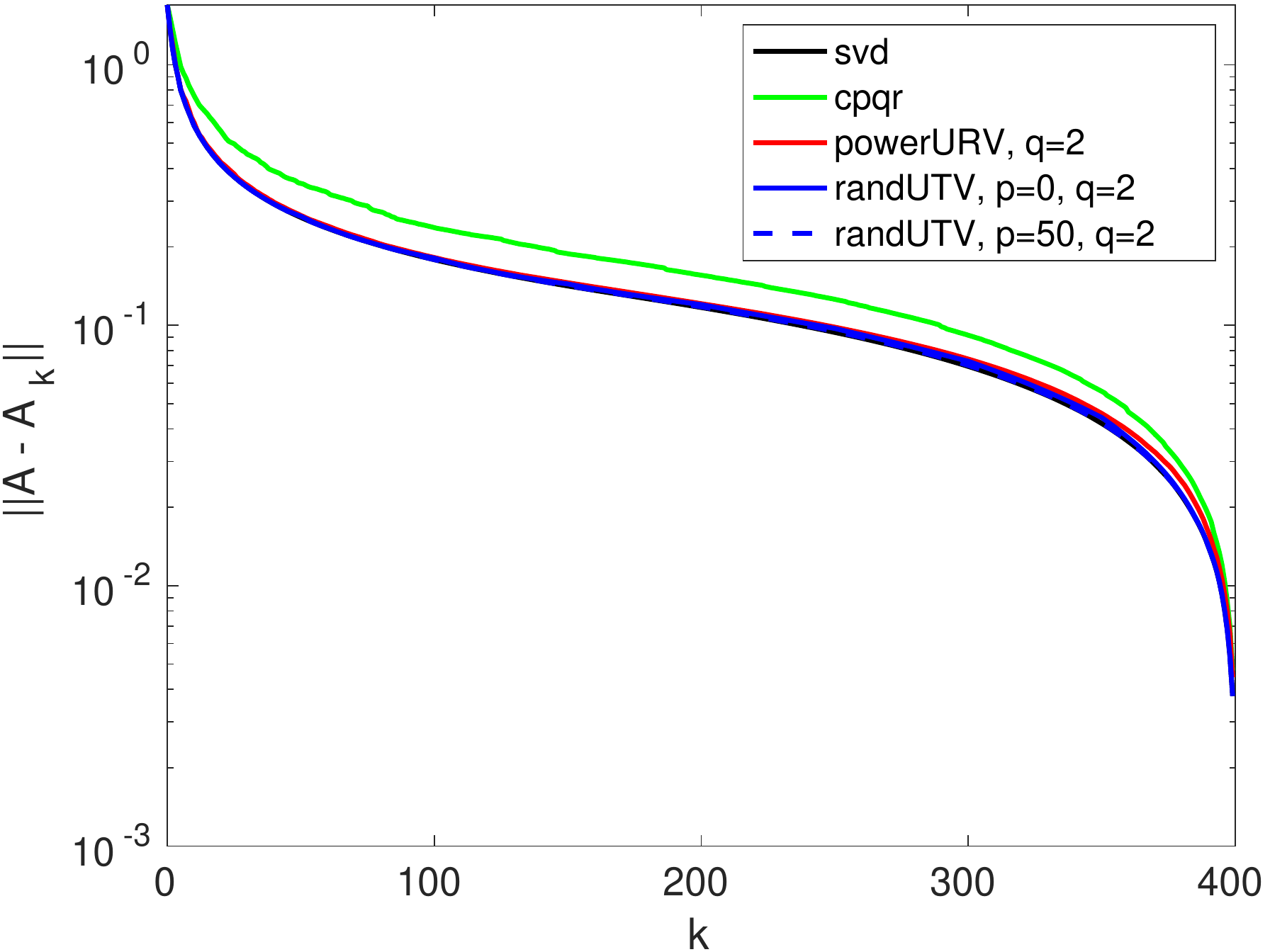}
\caption{``BIE'', Frobenius norm}
\end{subfigure}
\caption{Errors in low-rank approximations for  matrices ``fast decay'', ``S-shaped decay'', and ``BIE'' of size $n=400$. For the \textsc{randUTV} factorizations, the block size was $b=50$. The x-axis is  the rank of corresponding approximations.}
\label{fig:error}
\end{figure}

\begin{figure}
\centering
\begin{subfigure}{0.48\textwidth}
\centering
\includegraphics[width=\textwidth]{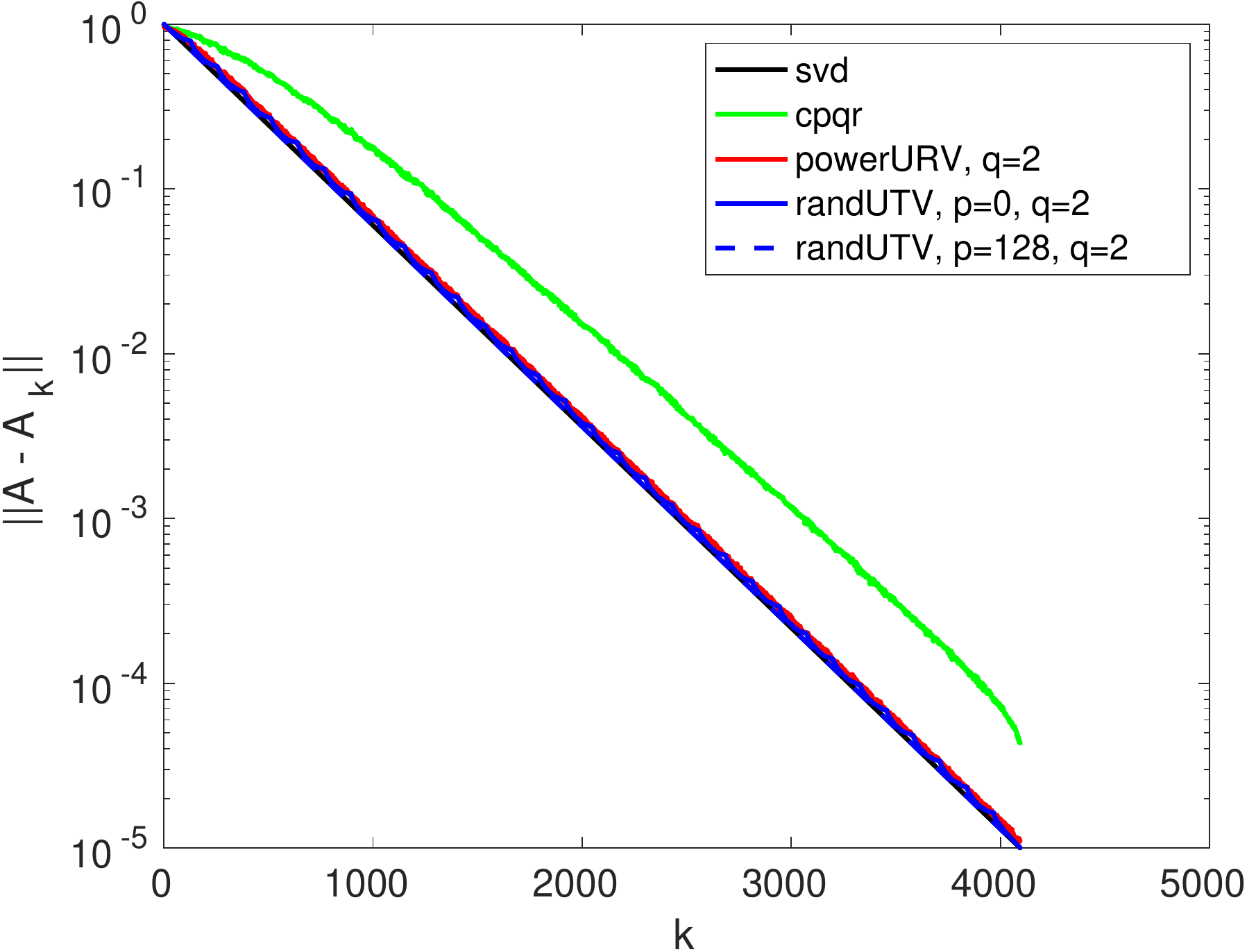}
\caption{``fast decay'', operator norm}
\end{subfigure}
\begin{subfigure}{0.48\textwidth}
\centering
\includegraphics[width=\textwidth]{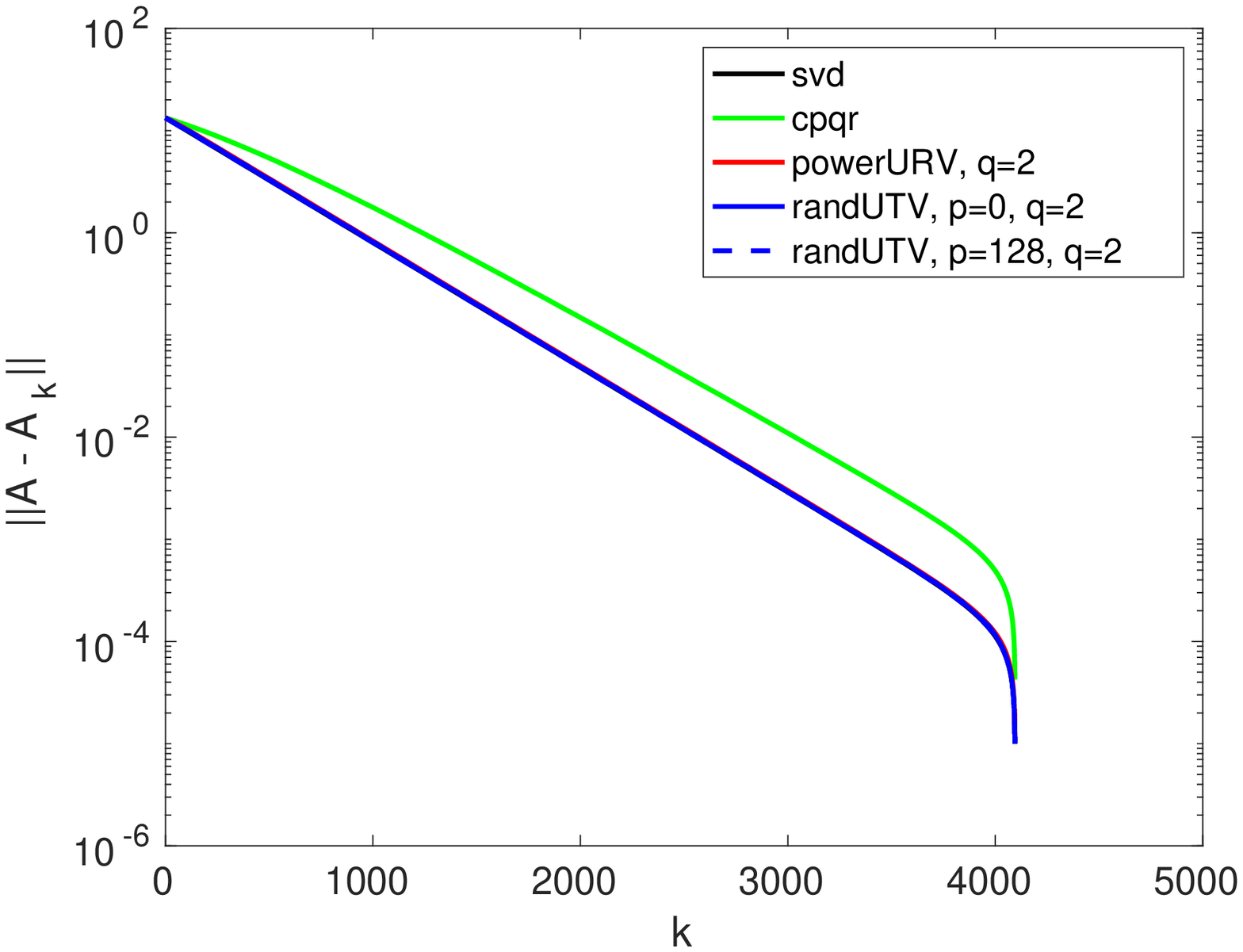}
\caption{``fast decay'', Frobenius norm}
\end{subfigure}
\begin{subfigure}{0.48\textwidth}
\centering
\includegraphics[width=\textwidth]{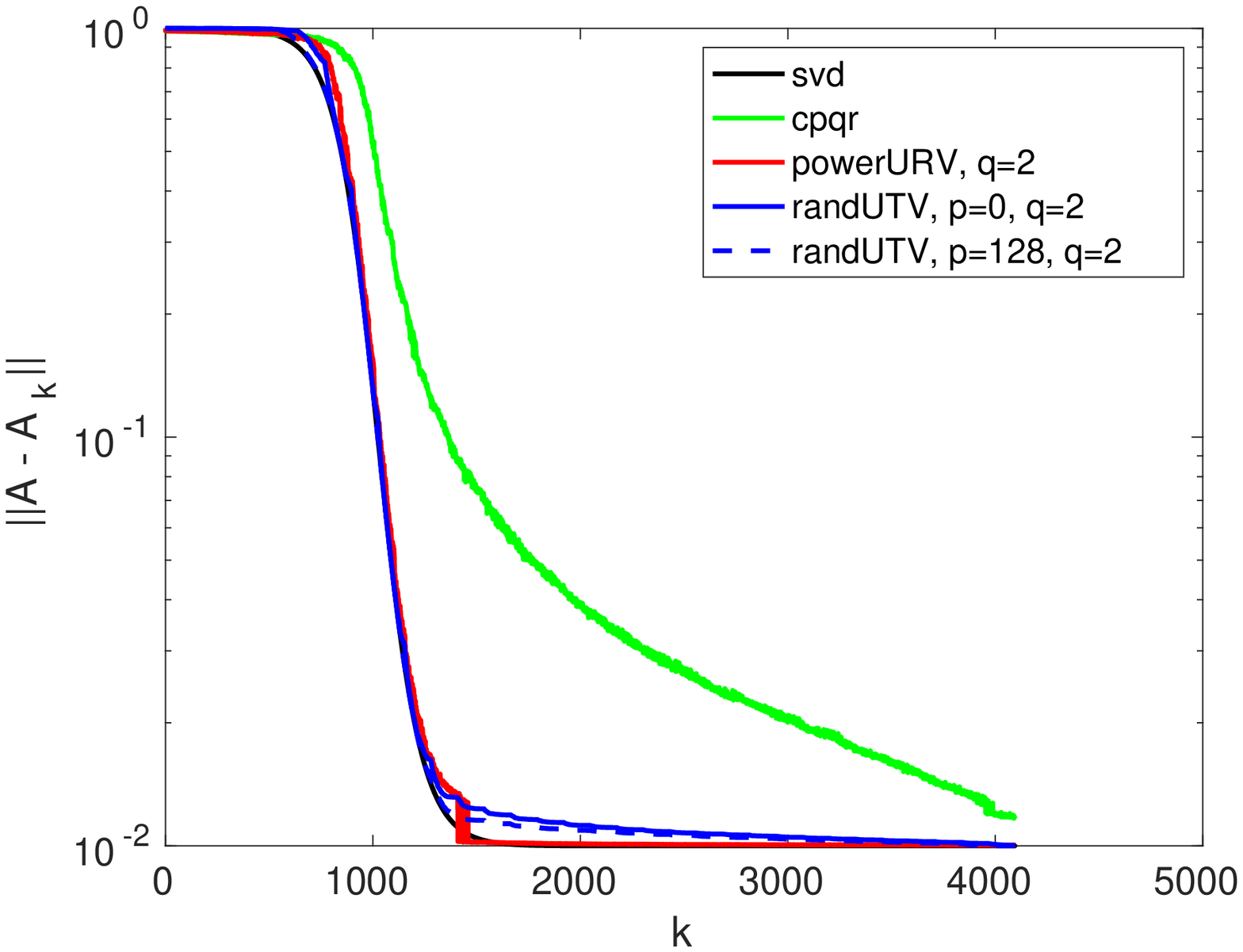}
\caption{``S-shaped decay'', operator norm}
\end{subfigure}
\begin{subfigure}{0.48\textwidth}
\centering
\includegraphics[width=\textwidth]{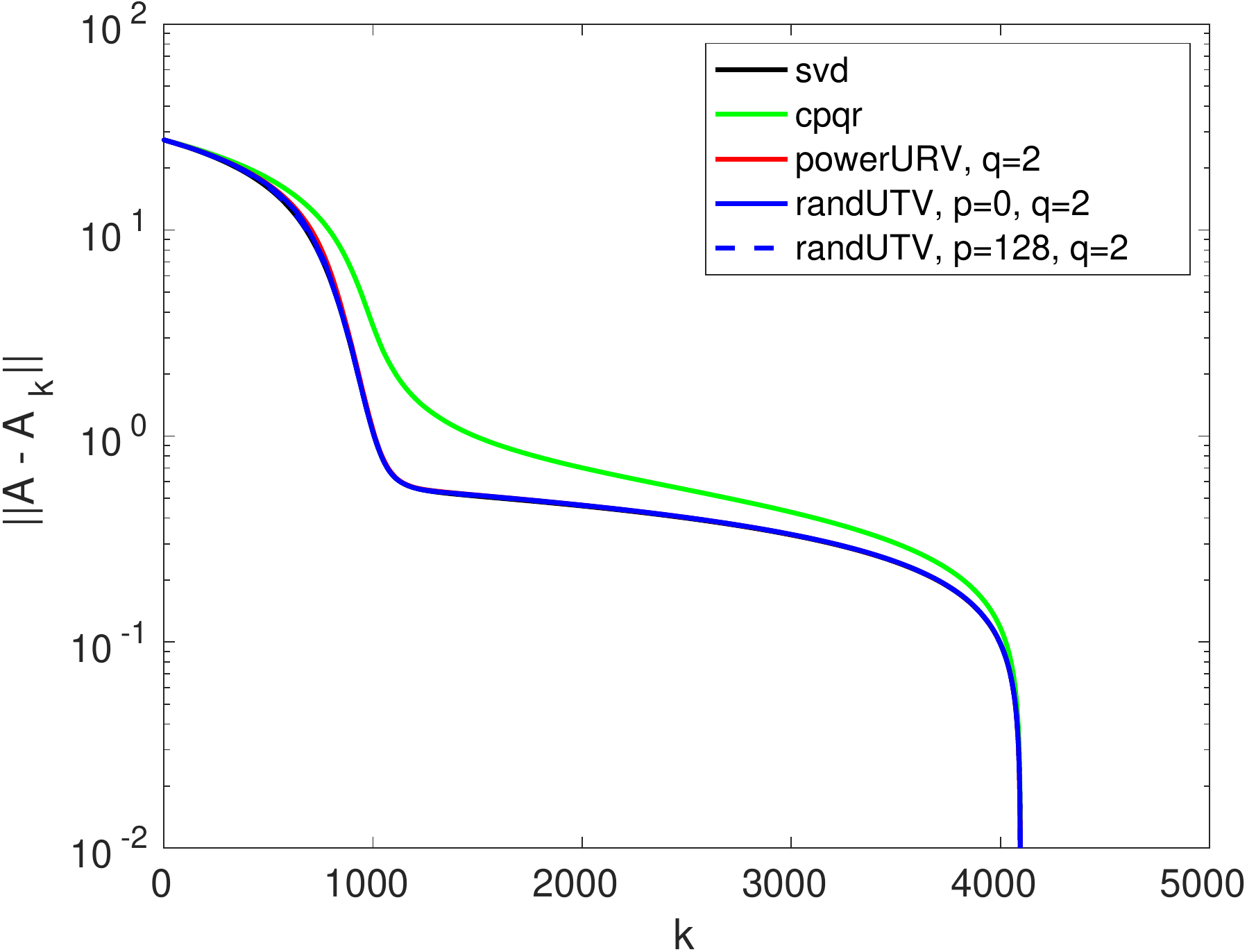}
\caption{``S-shaped decay'', Frobenius norm}
\end{subfigure}
\begin{subfigure}{0.48\textwidth}
\centering
\includegraphics[width=\textwidth]{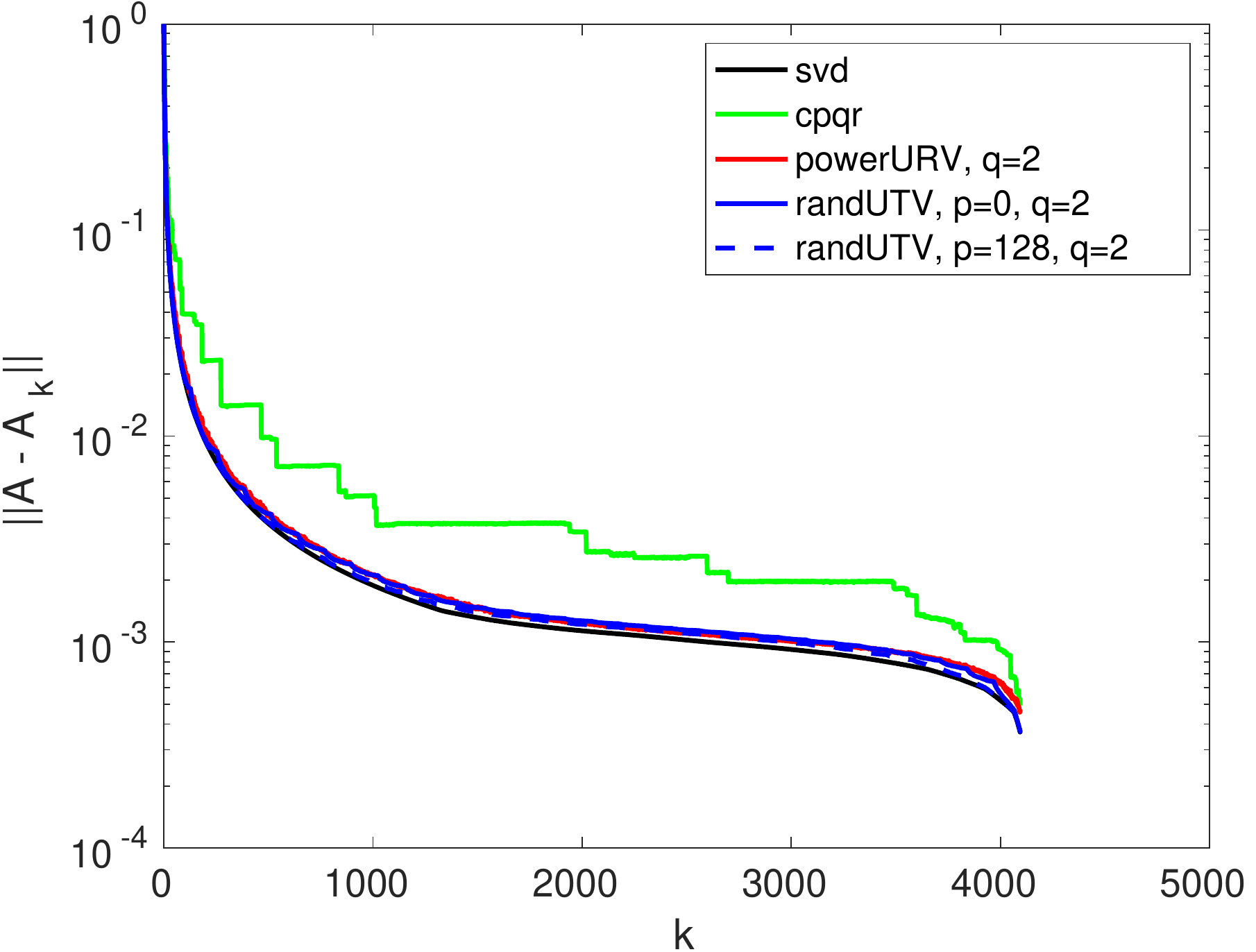}
\caption{``BIE'', operator norm}
\end{subfigure}
\begin{subfigure}{0.48\textwidth}
\centering
\includegraphics[width=\textwidth]{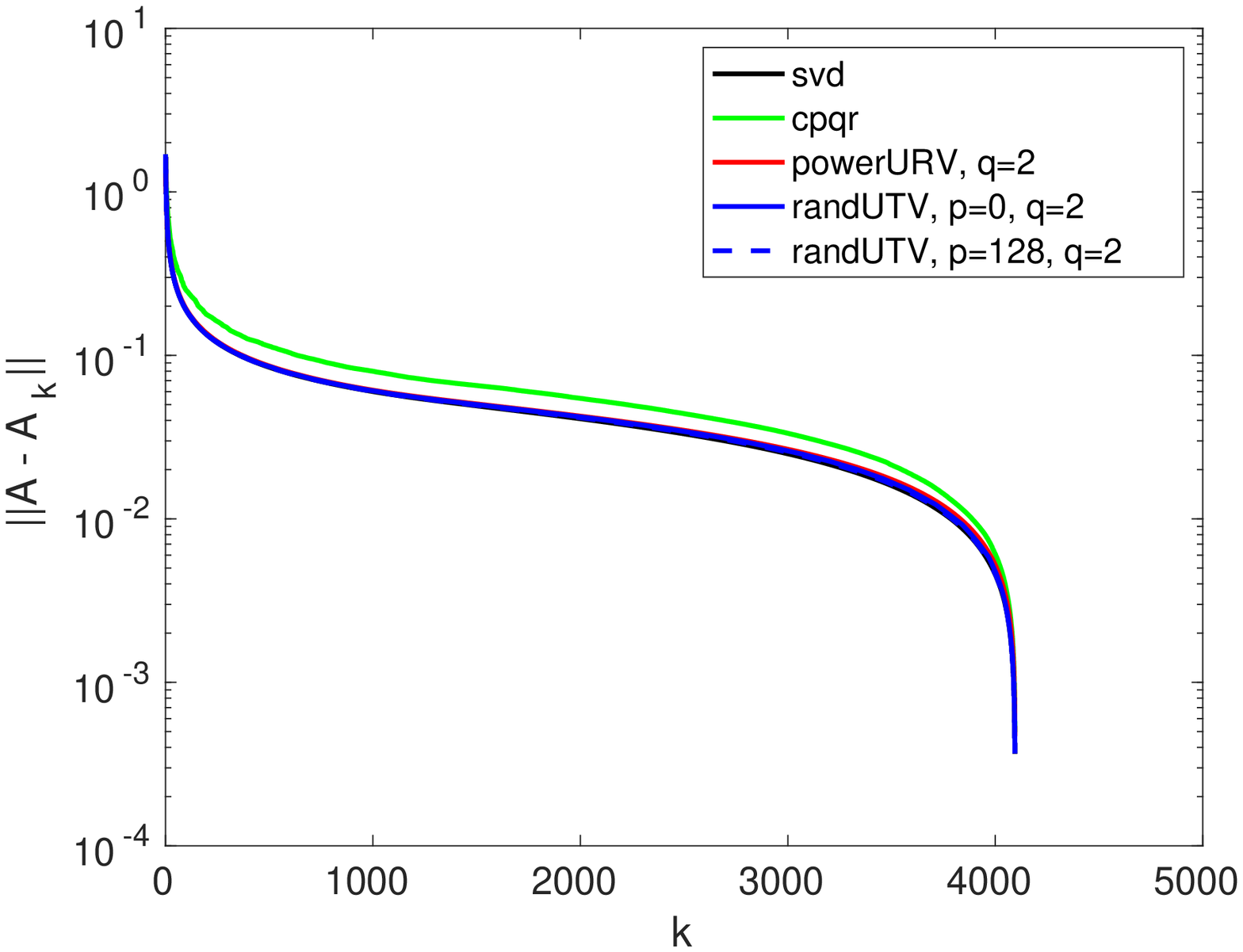}
\caption{``BIE'', Frobenius norm}
\end{subfigure}
\caption{Errors in low-rank approximations for matrices ``fast decay'', ``S-shaped decay'', and ``BIE'' of size $\emph{n=4\,000}$. For the \textsc{randUTV} factorizations, the block size was $b=128$ (as used in numerical experiments in Section~\ref{sec:ch5-speed}). The x-axis is the rank of corresponding approximations.}
\label{fig:error_n4000}
\end{figure}

\begin{figure}
\centering
\begin{subfigure}{0.48\textwidth}
\centering
\includegraphics[width=\textwidth]{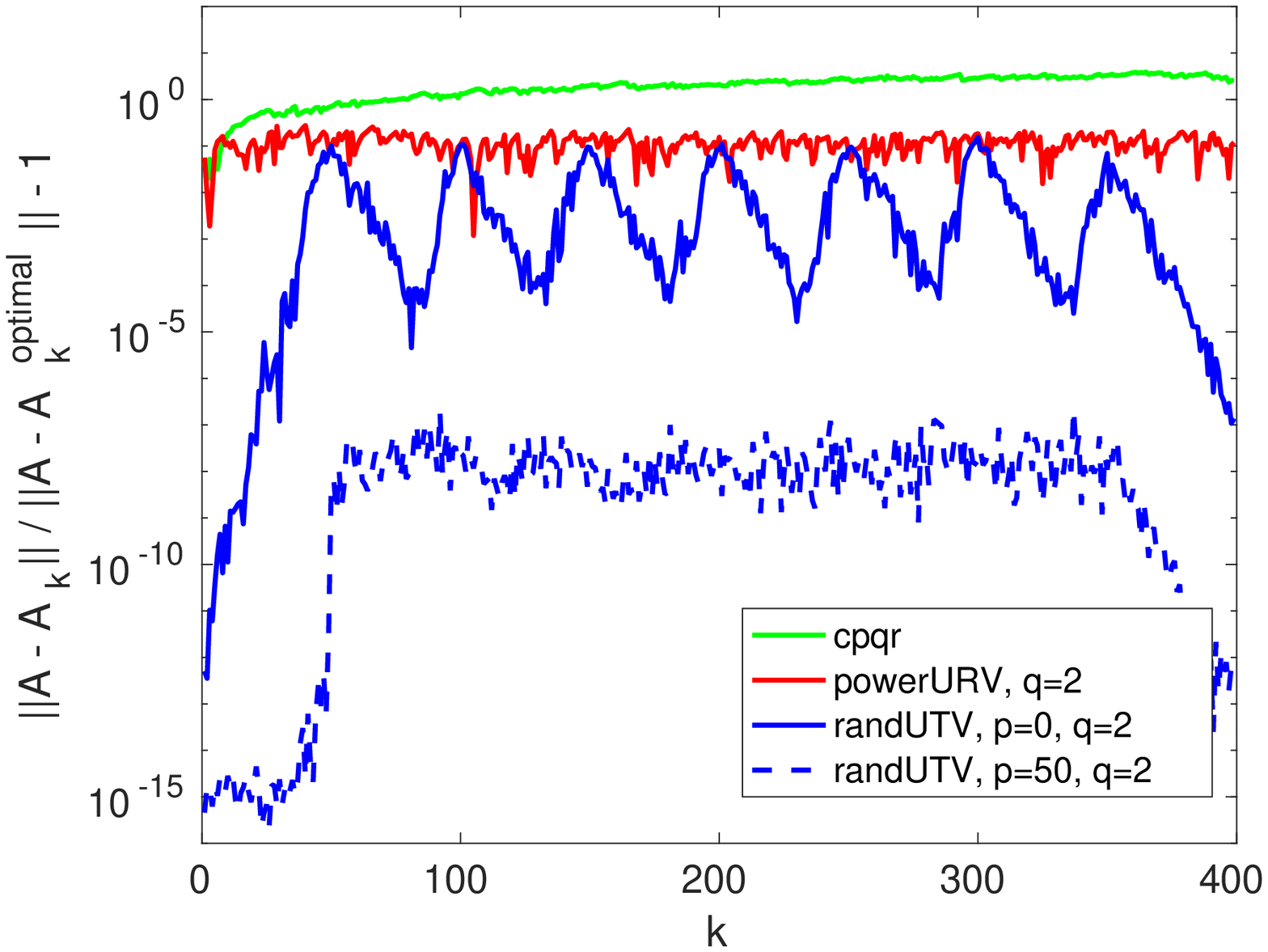}
\caption{``fast decay'', operator norm} \label{fig:relerror_fast} 
\end{subfigure}
\begin{subfigure}{0.48\textwidth}
\centering
\includegraphics[width=\textwidth]{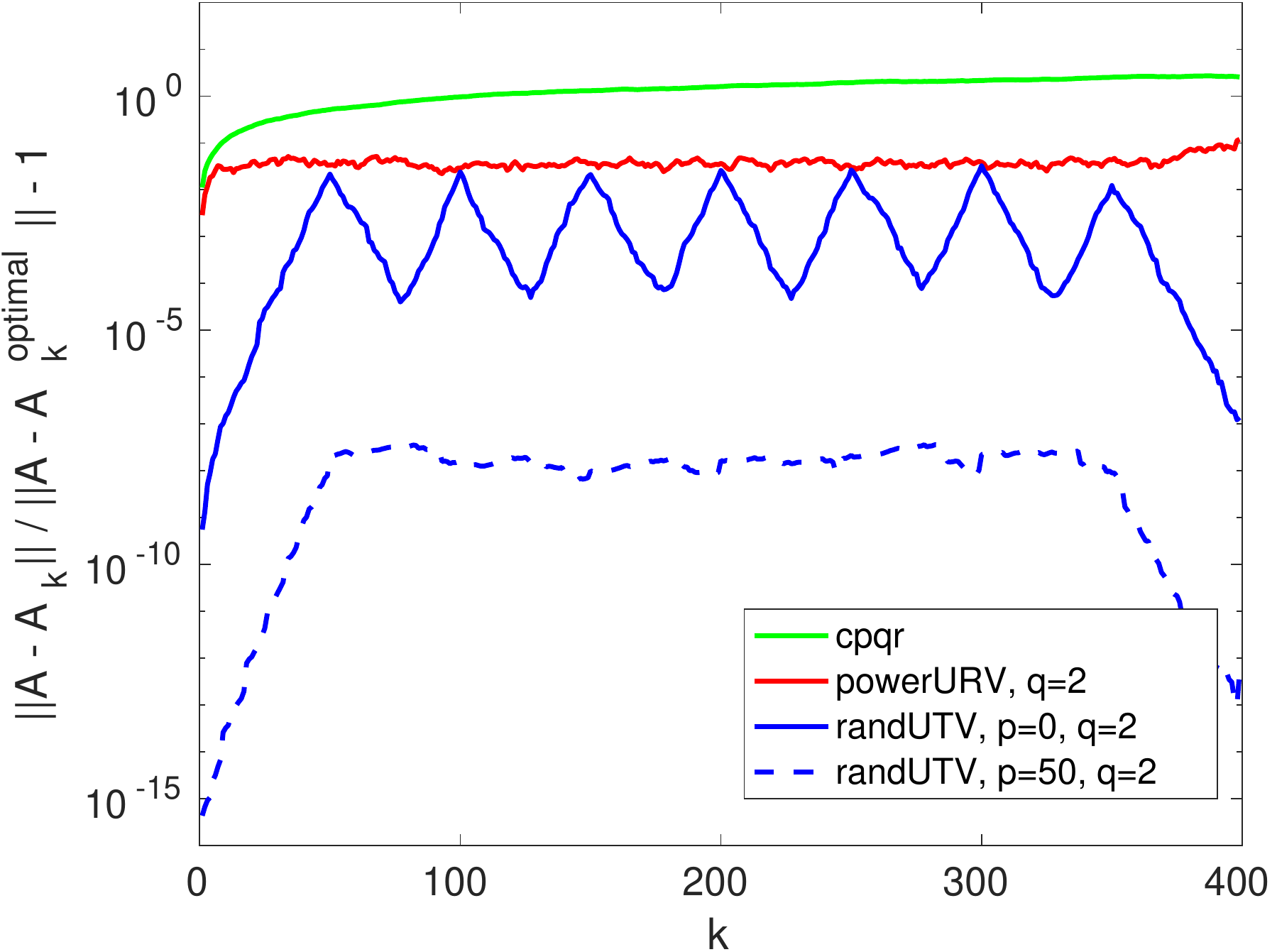}
\caption{``fast decay'', Frobenius norm}
\end{subfigure}
\begin{subfigure}{0.48\textwidth}
\centering
\includegraphics[width=\textwidth]{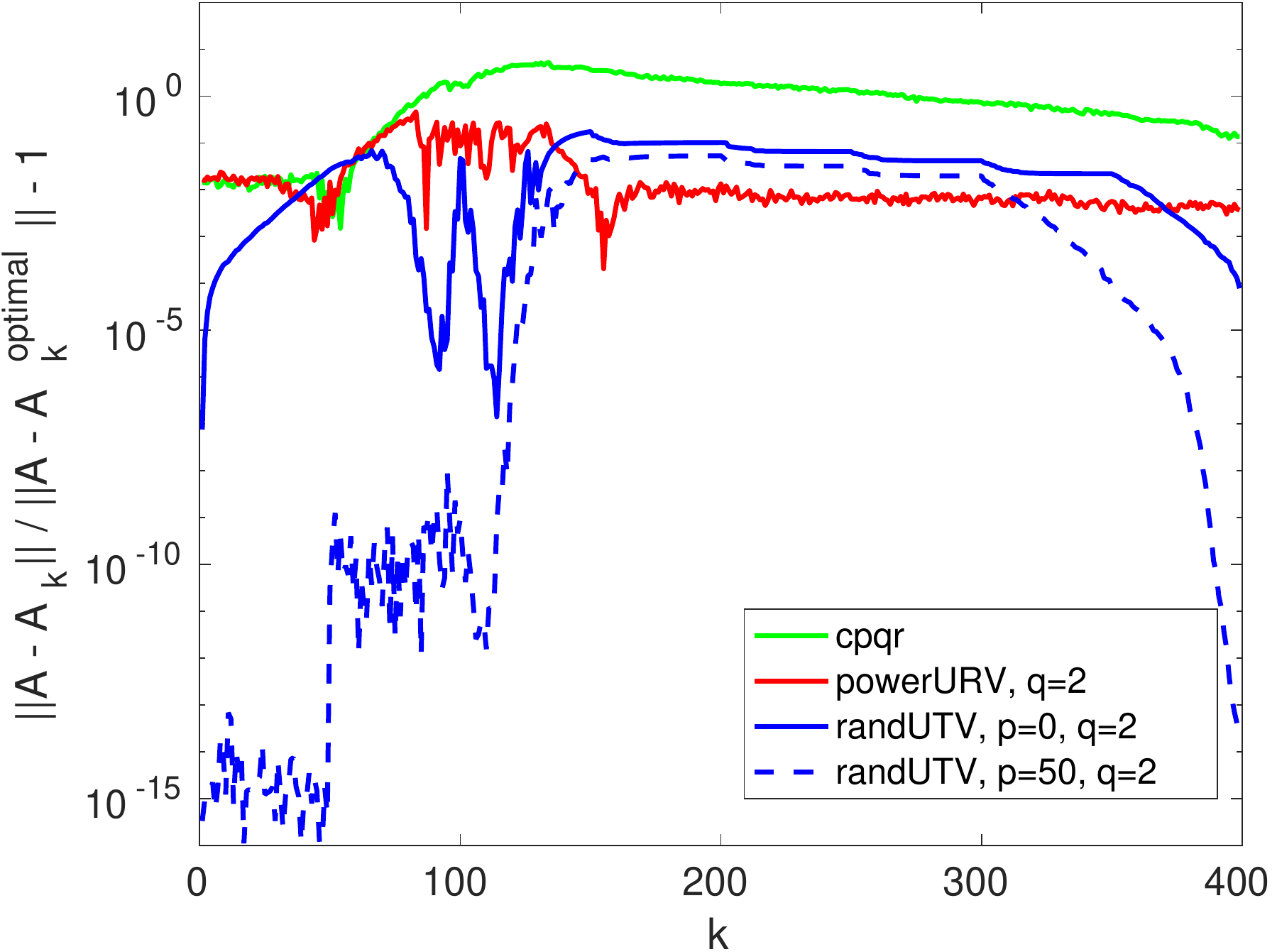}
\caption{``S-shaped decay'', operator norm}
\end{subfigure}
\begin{subfigure}{0.48\textwidth}
\centering
\includegraphics[width=\textwidth]{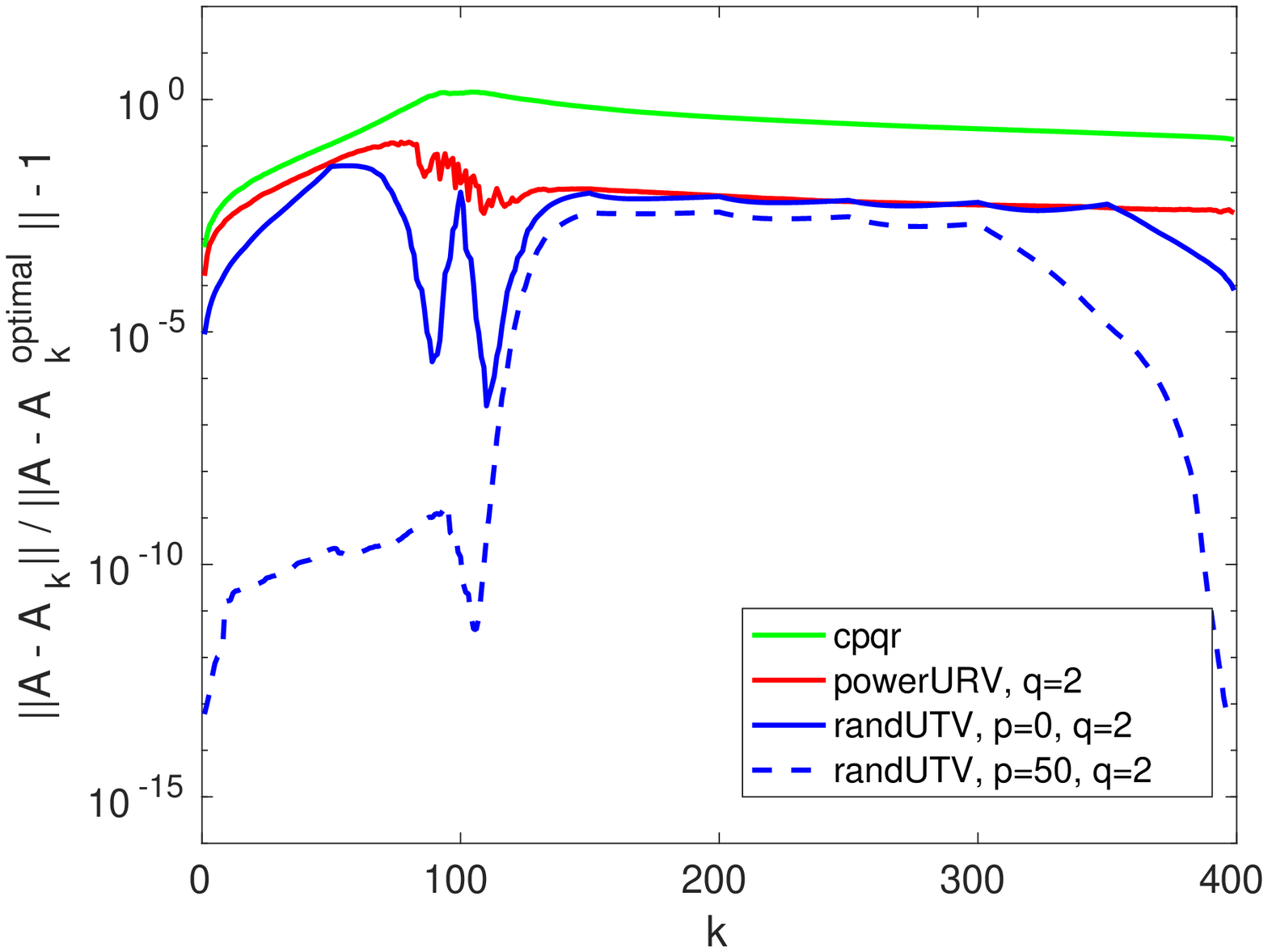}
\caption{``S-shaped decay'', Frobenius norm}
\end{subfigure}
\begin{subfigure}{0.48\textwidth}
\centering
\includegraphics[width=\textwidth]{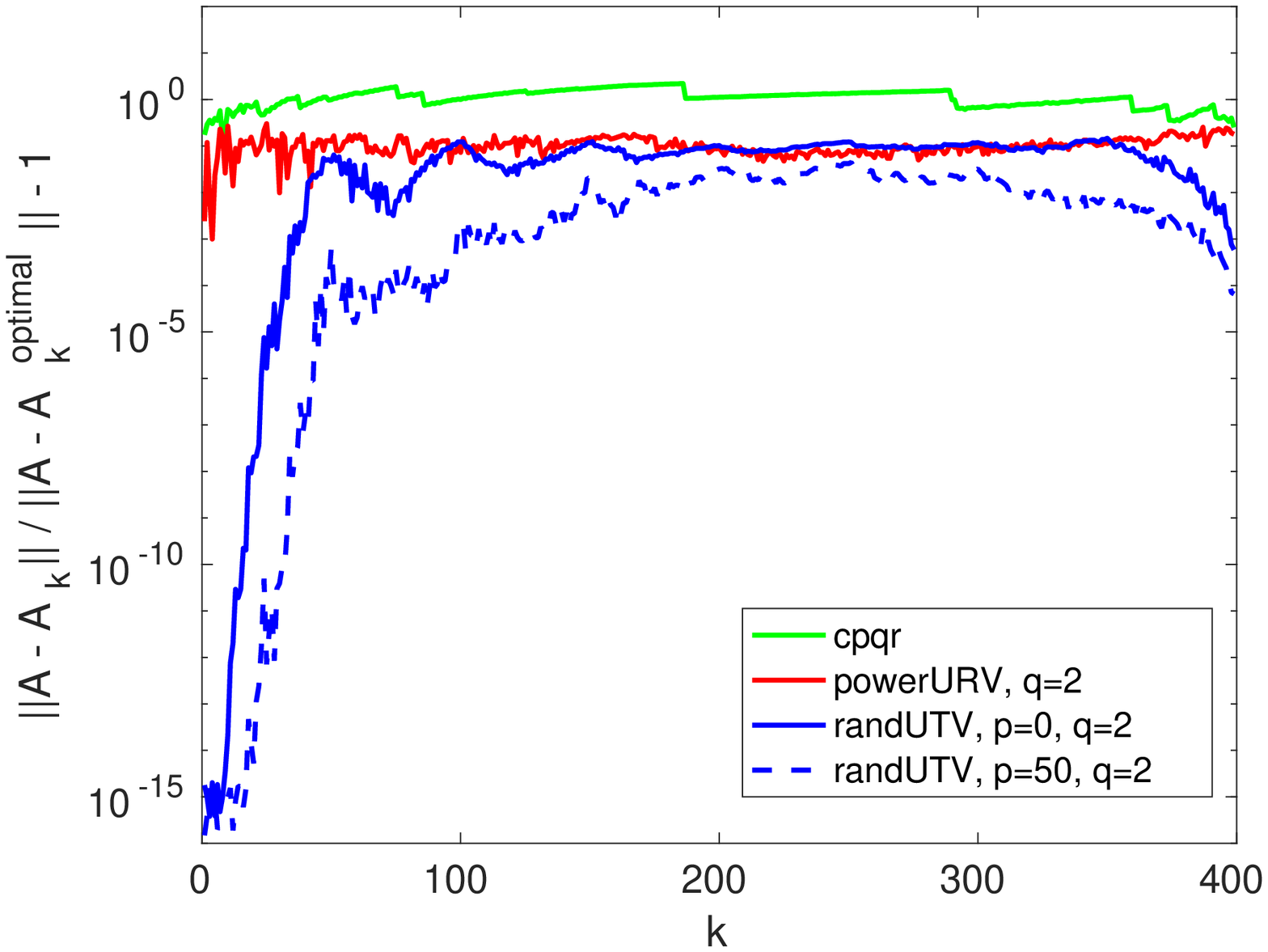}
\caption{``BIE'', operator norm}
\end{subfigure}
\begin{subfigure}{0.48\textwidth}
\centering
\includegraphics[width=\textwidth]{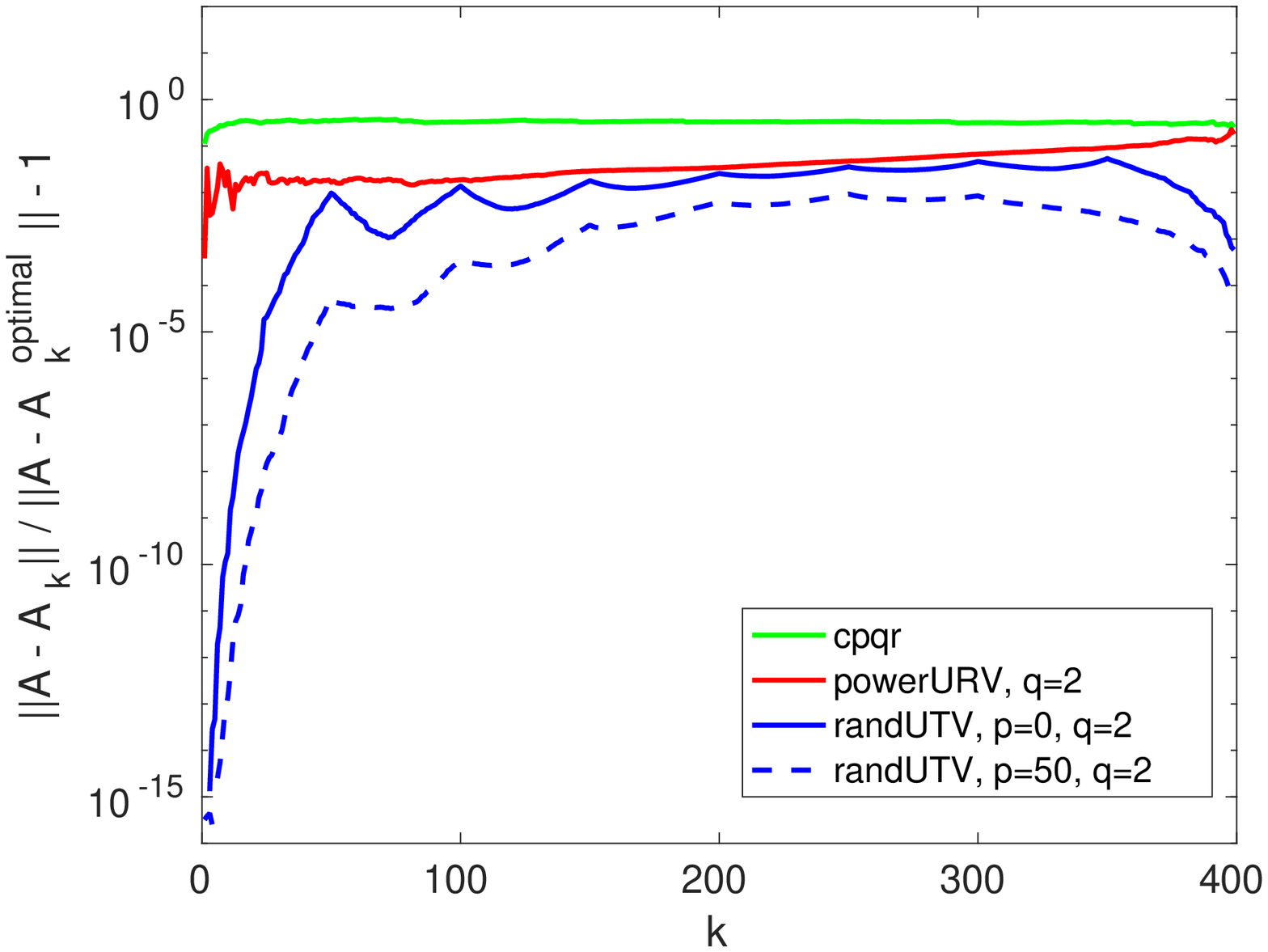}
\caption{``BIE'', Frobenius norm}
\end{subfigure}
\caption{Relative errors in low-rank approximations for  matrices ``fast decay'', ``S-shaped decay'', and ``BIE'' of size $n=400$. For the \textsc{randUTV} factorizations, the block size was $b=50$. The x-axis is the rank of corresponding approximations.}
\label{fig:relerror}
\end{figure}

\begin{remark}
{Evaluating the error $e_{k}$ defined in (\ref{eq:errork}) for all $k=1,2,\ldots,n$ requires as much as $O(n^3)$ work, which in practice took significantly longer than computing the rank revealing factorizations. Figure~\ref{fig:error_n4000} shows results corresponding to relatively large matrices of size $n=4\,000$, which look similar to those in Figure~\ref{fig:error}.}
\end{remark}

\begin{remark}
{
As a curiosity, let us note that we have empirically observed that \textsc{randUTV} does surprisingly
well at controlling \textit{relative} errors. Figure \ref{fig:relerror} reports the error metric
$$
\frac{\| \mtx{A} - \mtx{A}_k \|}{ \| \mtx{A} - \mtx{A}_k^{\text{optimal}} \|} - 1 =
\frac{\| \mtx{A} - \mtx{A}_k \| - \| \mtx{A} - \mtx{A}_k^{\text{optimal}} \| }{ \| \mtx{A} - \mtx{A}_k^{\text{optimal}} \|} \ge 0
$$
for our three test matrices. We see that while CPQR does not manage to control this error, \textsc{randUTV}
does a decent job, in particular when the power iteration is incorporated.}
\end{remark}

We make three observations based on results in Figures \ref{fig:error}, \ref{fig:error_n4000}, and \ref{fig:relerror}. First, \textsc{powerURV} and \textsc{randUTV} {are more accurate than \textsc{HQRCP} for a given $k$.}
Errors from \textsc{powerURV} and \textsc{randUTV} are substantially smaller in almost all cases studied. Take Figure \ref{fig:fast} as an example. The errors from \textsc{HQRCP} (green lines) are far from the minimals (black lines) achieved by SVD, whereas the errors from \textsc{powerURV} and \textsc{randUTV} (red and blue lines) are much closer to the minimals. It is also obvious in Figure \ref{fig:relerror} that the normalized errors from \textsc{HQRCP} (green lines) are usually larger (higher) than those from \textsc{powerURV} and \textsc{randUTV} (red and blue lines).

Second, \textsc{randUTV} is better than  \textsc{powerURV} when the same steps of power iteration are used. For a large number of cases, errors from \textsc{randUTV} are significantly smaller. This is shown in  Figure~\ref{fig:error} that the blue lines (\textsc{randUTV}) lie in between the black lines (SVD) and the red lines (\textsc{powerURV}). It is more obvious in Figure~\ref{fig:relerror} that the blue lines (\textsc{randUTV}) are lower  than the red lines (\textsc{powerURV}).

Third, the new oversampling scheme in \textsc{randUTV} provides a  boost in accuracies of low rank approximations obtained, even when the singular values decay slowly. The accuracy improvement is most pronounced when the singular values decay fast as shown in Figure \ref{fig:relerror_fast}. The figure also shows that without oversampling in \textsc{randUTV}, the accuracies of  rank-$k$ approximations deteriorate when $k$ is approximately a multiple of the block size $b$. This phenomenon is known in~\cite{halko2011finding} and the reason of incorporating the oversampling scheme in \textsc{randUTV}.

\begin{remark}
The error $e_{k}$ defined in (\ref{eq:errork}) measures how well condition (b) in
Section \ref{ss:rrf} is satisfied. In the interest of space, we do not report
analogous metrics for condition (a). Generally speaking, the results for \textsc{powerURV} and \textsc{randUTV}
are similar, as reported in \cite{gopal2018powerURV,martinsson2019randutv}.
\end{remark}


\section{{Conclusions}}

The computation of a rank-revealing factorization has important applications in, \textit{e.g.}, low-rank approximation and subspace identification. While the SVD is theoretically optimal, computing the SVD of a matrix requires a significant amount of work that can not fully leverage modern parallel computing environments such as a GPU. For example, computing the SVD of an $15\,000$ by $15\, 000$ matrix took more than a minute (79.4 seconds) on an NVIDIA V100 GPU.

We have described two randomized algorithms---\textsc{powerURV} and \textsc{randUTV}---as  economical alternatives to the SVD. As we demonstrate through numerical experiments,  both methods are much faster than the SVD on a GPU since they are designed to fully leverage parallel communication-constrained hardwares, and they  provide close-to-optimal accuracy. The main features of the two new methods, respectively, include (1) \textsc{powerURV} has a simple formulation as shown in Algorithm~\ref{alg:powerURV-stable} that it can be implemented easily on a GPU (or other parallel architectures), and (2) \textsc{randUTV} is a blocked incremental method that can be used to compute partial factorizations efficiently.

Compared to the CPQR factorization, which is commonly used as an economical alternative to the SVD for low-rank approximation, the  new  methods---\textsc{powerURV} and \textsc{randUTV}---are much more accurate and are similar or better in terms of speed on a GPU. Between the two methods, the \textsc{randUTV} is more accurate and generally faster, especially when power iteration is used. The accuracy of the \textsc{randUTV} can be further improved through the described oversampling technique (Section~\ref{sec:ch5-oversampling}), which requires  a modest amount of extra computation.

The two proposed methods can be viewed as evolutions of the RSVD. The distinction, however, is that the RSVD and related randomized methods for low-rank approximation~\cite{halko2011finding,liberty2007randomized,martinsson2011randomized}) work best when the numerical rank $k$ of an input matrix is much smaller than its dimensions. The key advantage of the two new methods is computational efficiency, in particular on GPUs, and they provide high speed for any rank.

\appendix

\section{Results related to random matrices} \label{a:proof}

\begin{thm} \label{th:invert}
Let $\mtx{G} \in \mathbb{R}^{m \times n}$ be a standard Gaussian matrix. Then, with probability 1, $\mtx{G}$ has full rank.
\end{thm}

\begin{proof}
Without loss of generality, assume $m \ge n$. It is obvious that
\[
\mbox{Pr [ $\mtx{G}(1:n,1:n)$ is rank-deficient ] } \ge \mbox{Pr [ $\mtx{G}$ is rank-deficient ] }.
\]
According to~\cite[Equation (3.2)]{rudelson2010non}, the probability that $\mtx{G}(1:n,1:n)$ is singular equals zero. Therefore, the theorem holds.
\end{proof}

\begin{thm} \label{th:rank}
Let $\mtx{A} \in \mathbb{R}^{m \times n}$ have rank $k \le \min(m,n)$ and $\mtx{G} \in \mathbb{R}^{n \times \ell}$ be a standard Gaussian matrix. Let $r = \min(k, \ell)$. Then, with probability 1, matrix $\mtx{B} = \mtx{A} \mtx{G}$ has rank $r$ and the first $r$ columns of $\mtx{B}$ are linearly independent.
\end{thm}

\begin{proof}

Let the thin SVD of $\mtx{A} \in \mathbb{R}^{m \times n}$ be $\mtx{A} = \mtx{U} \mtx{\Sigma} \mtx{V}^*$, where $\mtx{U} \in \mathbb{R}^{m \times k}$ is an orthonormal matrix, $\mtx{\Sigma} \in \mathbb{R}^{k \times k}$ is a diagonal matrix, and $\mtx{V} \in \mathbb{R}^{k \times n}$ is an orthonormal matrix. Therefore,
\[
\mtx{B}(:, 1:r) = \mtx{U} \mtx{\Sigma} \mtx{V}^* \mtx{G}(:,1:r).
\]
Since $\mtx{V}^* \mtx{G}(:,1:r) \in \mathbb{R}^{k \times r}$ also has the standard Gaussian distribution, it is full rank with probability 1 according to Theorem~\ref{th:invert}. So it is obvious that $\mtx{B}(:, 1:r)$ has full rank.

On the other hand, we know
\[
\mbox{rank}(\mtx{B}) \le \min(\mbox{rank}(\mtx{A}), \mbox{rank}(\mtx{G})) = r.
\]
Therefore, it holds that $\mbox{rank}(\mtx{B}) = r$.
\end{proof}

\begin{corollary} \label{th:rank2}
Let $\mtx{A} \in \mathbb{R}^{m \times n}$ have rank $k \le \min(m,n)$ and $\mtx{G} \in \mathbb{R}^{n \times n}$ be a standard Gaussian matrix. Let $[\mtx{Q}, \mtx{R}] = \textsc{HQR\_full}(\mtx{G})$ ($\mtx{G}$ is invertible with probability 1 according to Theorem~\ref{th:invert}). Then, with probability 1, matrix $\mtx{B} = \mtx{A} \mtx{Q}$ has rank $k$ and the first $k$ columns of $\mtx{B}$ are linearly independent.
\end{corollary}

\begin{proof}
Since $\mtx{Q}$, a unitary matrix, has full rank, we know $\text{rank}(\mtx{B}) = \text{rank}(\mtx{A}) = k$.

Let $\mtx{C} = \mtx{B} \mtx{R} = \mtx{A} \mtx{G}$. We know that
\[
\mtx{C}(:,1:k) = \mtx{B}(:,1:k) \mtx{R}(1:k,1:k)
\]
has full rank according to Theorem~\ref{th:rank}. So it is obvious that $\mtx{B}(:, 1:k)$ has full rank since $\mtx{R}$ is invertible with probability 1.

\end{proof}

\begin{thm} \label{th:rank3}
Given a matrix $\mtx{A} \in \mathbb{R}^{m \times n}$, it holds that, for a positive integer $p$,
\[
\mbox{rank} \left( (\mtx{A}^* \mtx{A})^p \right) = \mbox{rank}( \left( \mtx{A} \mtx{A}^* \right)^p ) = \mbox{rank}( \mtx{A} ),
\]
and, for a non-negative integer $q$,
\[
\mbox{rank} \left( \mtx{A} (\mtx{A}^* \mtx{A})^q \right) = \mbox{rank}( \left( \mtx{A} \mtx{A}^* \right)^q \mtx{A} ) = \mbox{rank} \left( (\mtx{A}^* \mtx{A})^q \mtx{A}^* \right) = \mbox{rank}( \mtx{A}^* \left( \mtx{A} \mtx{A}^* \right)^q ) = \mbox{rank}( \mtx{A} ).
\]
\end{thm}

\begin{proof}
Suppose $\mbox{rank}( \mtx{A} ) = k$. Let the thin SVD  of $\mtx{A}$ be $\mtx{A} = \mtx{U} \mtx{\Sigma} \mtx{V}^*$, where $\mtx{U} \in \mathbb{R}^{m \times k}$ and $\mtx{V} \in \mathbb{R}^{k \times n}$ are orthonormal matrices, and $\mtx{\Sigma} \in \mathbb{R}^{k \times k}$ is a diagonal matrix with the positive singular values. We know that
\[
(\mtx{A}^* \mtx{A})^p  = \mtx{V} \mtx{\Sigma}^{2p} \mtx{V}^* \text{ and }
(\mtx{A} \mtx{A}^*)^p  = \mtx{U} \mtx{\Sigma}^{2p} \mtx{U}^*,
\]
and
\[
\mtx{A} (\mtx{A}^* \mtx{A})^q  =  (\mtx{A} \mtx{A}^*)^q \mtx{A}  = \mtx{U} \mtx{\Sigma}^{2q+1} \mtx{V}^* \text{ and }
(\mtx{A}^* \mtx{A})^q \mtx{A}^* = \mtx{A}^* (\mtx{A} \mtx{A}^*)^q = \mtx{V} \mtx{\Sigma}^{2q+1} \mtx{U}^*,
\]
So it is obvious that the theorem holds.
\end{proof}

\section{\textsc{randUTV} algorithm adapted from \cite{2015_martinsson_blocked}.} \label{a:utv}

\begin{algorithm}
\begin{algorithmic}[1]
 \REQUIRE {matrix $\mtx{A} \in \mathbb{R}^{m \times n}$, a positive integer $b$, and a non-negative integer $q$.}
 \ENSURE {$\mtx{A} = \mtx{U} \mtx{T} \mtx{V}^*$, where $\mtx{U} \in \mathbb{R}^{m \times m}$  and $\mtx{V} \in \mathbb{R}^{n \times n}$ are orthogonal, and $\mtx{T} \in \mathbb{R}^{m \times n}$ is upper trapezoidal.}
\STATE $\mtx{T} = \mtx{A}; \mtx{U} = \mtx{I}_m; \mtx{V} = \mtx{I}_n;$
\FOR{$i=1$: min$(\lceil m/b \rceil, \lceil n/b \rceil)$}
  \STATE $I_1 = 1:(b(i-1)); I_2 = (b(i-1)+1):\min(bi,m); I_3 = (bi+1):m;$
  \STATE $J_1 = 1:(b(i-1)); J_2 = (b(i-1)+1):\min(bi,n); J_3 = (bi+1):n;$
  \IF{($I_3$ and $J_3$ are both nonempty)}
    \STATE {$\mtx{G} = \textsc{randn}(m - b(i-1),b)$}
    \STATE $\mtx{Y} = \mtx{T}([I_2,I_3],[J_2,J_3])^* \mtx{G}$
    \FOR{$j=1:q$}
    \STATE $\mtx{Y} = \mtx{T}([I_2,I_3],[J_2,J_3])^*(\mtx{T}([I_2,I_3],[J_2,J_3])\mtx{Y})$
    \ENDFOR
    \STATE $[\mtx{V}^{\si},\sim] = \textsc{HQR\_full}(\mtx{Y})$
    \STATE $\mtx{T}(:,[J_2,J_3]) = \mtx{T}(:,[J_2,J_3]) \mtx{V}^{\si}$
    \STATE $\mtx{V}(:,[J_2,J_3]) = \mtx{V}(:,[J_2,J_3]) \mtx{V}^{\si}$
    \STATE
    \STATE $[\mtx{U}^{\si},\mtx{R}] = \textsc{HQR\_full}(\mtx{T}([I_2,I_3],J_2))$
    \STATE $\mtx{U}(:,[I_2,I_3]) = \mtx{U}(:,[I_2,I_3]) \mtx{U}^{\si}$
    \STATE $\mtx{T}([I_2,I_3],J_3) = \left(\mtx{U}^{\si}\right)^* \mtx{T}([I_2,I_3],J_3)$
    \STATE $\mtx{T}(I_3,J_2) = \textsc{zeros}(m-bi,b)$
    \STATE
    \STATE $[\mtx{U}_{\text{small}},\mtx{D}_{\text{small}},\mtx{V}_{\text{small}}] = \textsc{svd}(\mtx{R}(1:b,1:b))$
    \STATE $\mtx{U}(:,I_2) = \mtx{U}(:,I_2)  \mtx{U}_{\text{small}}$
    \STATE $\mtx{V}(:,J_2) = \mtx{V}(:,J_2) \mtx{V}_{\text{small}}$
    \STATE $\mtx{T}(I_2,J_2) = \mtx{D}_{\text{small}}$
    \STATE $\mtx{T}(I_2,J_3) = \mtx{U}_{\text{small}}^* \mtx{T}(I_2,J_3)$
    \STATE $\mtx{T}(I_1,J_2) = \mtx{T}(I_1,J_2) \mtx{V}_{\text{small}}$
  \ELSE
    \STATE $[\mtx{U}_{\text{small}},\mtx{D}_{\text{small}},\mtx{V}_{\text{small}}] = \textsc{svd}(\mtx{T}([I_2,I_3],[J_2,J_3]))$
    \STATE $\mtx{U}(:,[I_2,I_3]) = \mtx{U}(:,[I_2,I_3]) \mtx{U}_{\text{small}}$
    \STATE $\mtx{V}(:,[J_2,J_3]) = \mtx{V}(:,[J_2,J_3]) \mtx{V}_{\text{small}}$
    \STATE $\mtx{T}([I_2,I_3],[J_2,J_3]) = \mtx{D}_{\text{small}}$
    \STATE $\mtx{T}([I_1,[J_2,J_3]) = \mtx{T}(I_1,[J_2,J_3]) \mtx{V}_{\text{small}}$
  \ENDIF
\ENDFOR
\end{algorithmic}
\caption{[\mtx{U},\mtx{T},\mtx{V}] = \textsc{randUTV\_basic}(\mtx{A},b,q)}
\label{alg:randutv-basic}
\end{algorithm}

\section{Raw data for Figure \ref{fig:ch5-powurv-times}, \ref{fig:ch5-randutv-basic}, and \ref{fig:ch5-randutv-boosted}} \label{a:data}

\begin{table}
    \centering
    \caption{Timing results (in seconds) for computing rank-revealing factorizations on a GPU using the SVD, the \textsc{HQRCP}, the \textsc{powerURV}, and the \textsc{randUTV}. The SVD is calculated using the MAGMA routine \textsc{magma\_dgesdd}. The \textsc{HQRCP} is calculated using the MAGMA routine \textsc{magma\_dgeqp3}, and  the orthogonal matrix $\mtx{Q}$ is calculated using the MAGMA routine \textsc{magma\_dorgqr2}. The \textsc{powerURV} and the \textsc{randUTV} are given in Algorithm~\ref{alg:powerURV-stable} and~\ref{alg:randutv-boosted}, respectively.}
    \label{t:raw}
    \begin{tabular}{cccccc}
    \toprule
    \multirow{2}{*}{$N$} & \multirow{2}{*}{SVD} & \multirow{2}{*}{\textsc{HQRCP}} & \multicolumn{3}{c}{\textsc{powerURV}}\\
      &  &  & $q=1$ & $q=2$ & $q=3$   \\
    \midrule
     3,000 & 1.98e+00 & 9.26e-01 & 5.06e-01 & 6.59e-01 & 9.38e-01  \\
     4,000 & 4.51e+00 & 1.58e+00 & 7.80e-01 & 9.79e-01 & 1.42e+00    \\
     5,000 & 5.97e+00 & 2.52e+00 & 1.13e+00 & 1.46e+00 & 1.92e+00   \\
     6,000 & 9.29e+00 & 3.66e+00 & 1.60e+00 & 2.14e+00 & 2.76e+00   \\
     8,000 & 1.85e+01 & 7.42e+00 & 2.68e+00 & 3.87e+00 & 5.29e+00  \\
     10,000 & 3.12e+01 & 1.11e+01 & 4.85e+00 & 7.29e+00 & 9.37e+00    \\
     12,000 & 4.74e+01 & 1.70e+01 & 7.62e+00 & 1.10e+01 & 1.50e+01   \\
     15,000 & 7.94e+01 & 2.84e+01 & 1.29e+01 & 1.94e+01 & 2.51e+01  \\
     20,000 & 1.71e+02 & 5.61e+01 & 2.85e+01 & 4.27e+01 & 6.03e+01 \\
     30,000 & N/A & 1.60e+02 & 9.12e+01 & 1.32e+02 & 1.87e+02    \\
    \bottomrule
    \end{tabular}
    \begin{tabular}{ccccccc}
    \toprule
    \multirow{2}{*}{$N$} & \multicolumn{3}{c}{\textsc{randUTV} ($p=0,b=128$)} & \multicolumn{3}{c}{\textsc{randUTV}  ($p=b=128$)} \\
      & $q=0$ & $q=1$ & $q=2$ & $q=0$ & $q=1$ & $q=2$   \\
    \midrule
     3,000 & 7.20e-01 & 7.43e-01 & 7.34e-01 & 1.45e+00 & 1.72e+00 & 1.72e+00  \\
     4,000 & 1.14e+00 & 1.16e+00 & 1.23e+00 & 2.19e+00 & 2.56e+00 & 2.65e+00  \\
     5,000 & 1.83e+00 & 1.76e+00 & 1.84e+00 & 3.29e+00 & 3.72e+00 & 3.72e+00  \\
     6,000 & 2.49e+00 & 2.61e+00 & 2.53e+00 & 4.31e+00 & 4.95e+00 & 4.69e+00  \\
     8,000 & 4.17e+00 & 4.28e+00 & 4.52e+00 & 7.50e+00 & 7.93e+00 &  8.19e+00 \\
     10,000 & 6.82e+00 & 7.03e+00 & 6.94e+00 & 1.05e+01 & 1.22e+01 & 1.25e+01   \\
     12,000 & 9.89e+00 & 1.01e+01 & 1.09e+01 & 1.52e+01 & 1.72e+01 & 1.79e+01  \\
     15,000 & 1.62e+01 & 1.70e+01 & 1.78e+01 & 2.43e+01 & 2.73e+01 & 2.82e+01  \\
     20,000 & 3.03e+01 & 3.27e+01 & 3.43e+01 & 4.39e+01 & 5.10e+01 & 5.36e+01   \\
     30,000 & 9.25e+01 & 9.63e+01 & 1.04e+02 & 1.27e+02 & 1.50e+02  & 1.55e+02  \\
    \bottomrule
    \end{tabular}
\end{table}

\bibliography{main_bib}
\bibliographystyle{amsplain}

\end{document}